\documentclass[12pt]{article}
\usepackage{a4}
\usepackage{amsthm}
\usepackage{amsmath}
\usepackage{amssymb}
\usepackage{amsfonts}
\usepackage{stmaryrd}
\usepackage{cite}
\usepackage{epsfig}
\usepackage{todonotes}
\usepackage{multirow}
\usepackage{enumerate}
\usepackage{hyperref}
\usepackage{mathrsfs}
\newtheorem{conjecture}{Conjecture}
\newtheorem{theorem}{Theorem}
\newtheorem{lemma}[theorem]{Lemma}
\newtheorem{proposition}[theorem]{Proposition}
\newtheorem{observation}[theorem]{Observation}

\newcommand\AAA{{\mathbb A}}
\newcommand\CC{{\mathcal C}}
\newcommand\JJ{{\mathcal J}}
\newcommand\HH{{\mathcal H}}
\newcommand\NN{{\mathbb N}}
\newcommand\RR{{\mathbb R}}
\newcommand\WW{{\mathcal W}}
\newcommand\dd{\,\mbox{d}}
\newcommand{\wf}{\widehat{f}}
\newcommand{\wg}{\widehat{g}}
\newcommand{\wA}{\widehat{A}}
\newcommand{\wh}{\widehat{h}}
\newcommand{\wm}{\widehat{m}}

\newcommand{\wx}{\widetilde{x}}
\newcommand{\wt}{\widetilde{t}}

\newcommand{\whw}{w'}
\newcommand{\whwtw}{\widetilde{w}'}
\newcommand{\wht}{t'}
\newcommand{\whS}{S'}
\newcommand{\wha}{a'}
\newcommand{\whb}{b'}
\newcommand{\cJ}{\mathcal{J}}

\newcommand{\wtw}{\widetilde{w}}

\newcommand{\Wt}{\widetilde{W}}
\newcommand{\parti}{\mathcal{P}}
\newcommand{\partis}{\mathscr{P}}
\newcommand{\vecz}[1]{\vec{#1}}
\newcommand{\eps}{\varepsilon}
\newcommand\Aut{\mbox{Aut}}
\newcommand{\coord}[1]{\left\llbracket #1 \right\rrbracket}
\allowdisplaybreaks

\begin{document}
\title{Elusive extremal graphs}

\author{Andrzej Grzesik\thanks{Faculty of Mathematics and Computer Science, Jagiellonian University, {\L}ojasiewicza 6, 30-348 Krak\'{o}w, Poland. Previous affiliation: DIMAP and Department of Computer Science, University of Warwick, Coventry CV4 7AL, UK. E-mail: {\tt Andrzej.Grzesik@uj.edu.pl}. The work of this author was supported by the National Science Centre grant number 2016/21/D/ST1/00998 and by the Engineering and Physical Sciences Research Council Standard Grant number EP/M025365/1.}\and
        Daniel Kr\'al'\thanks{Faculty of Informatics, Masaryk University, Botanick\'a 68A, 602 00 Brno, Czech Republic, and Mathematics Institute, DIMAP and Department of Computer Science, University of Warwick, Coventry CV4 7AL, UK. E-mail: {\tt dkral@fi.muni.cz}. The work of this author was supported by the European Research Council (ERC) under the European Union’s Horizon 2020 research and innovation programme (grant agreement No 648509) and by the Engineering and Physical Sciences Research Council Standard Grant number EP/M025365/1. This publication reflects only its authors' view; the European Research Council Executive Agency is not responsible for any use that may be made of the information it contains.}\and
	L\'aszl\'o Mikl\'os Lov\'asz\thanks{Department of Mathematics, Massachusetts Institute of Technology, Cambridge, MA, 02139, USA. E-mail: {\tt lmlovasz@mit.edu}. This author was supported by NSF Postdoctoral Fellowship Award DMS-1705204.}}
\date{}
\maketitle
\begin{abstract}
We study the uniqueness of optimal solutions to extremal graph theory problems.
Lov\'asz conjectured that 
every finite feasible set of subgraph density constraints
can be extended further by a finite set of density constraints so that
the resulting set is satisfied by an asymptotically unique graph.
This statement is often referred to as saying that ``every extremal graph theory problem has a finitely forcible optimum''.
We present a counterexample to the conjecture.
Our techniques also extend to a more general setting involving other types of constraints.
\end{abstract}

%\begin{flushleft}
%2010 Mathematics Subject Classification: 05C35, 05C82
%\end{flushleft}

\section{Introduction}
\label{sec-intro}

Many problems in extremal graph theory do not have asymptotically unique solutions.
As an example,
consider the problem of minimizing the sum of the induced subgraph densities of $K_3$ and its complement.
It can be shown that this sum is minimized by any $n$-vertex graph where all vertices have degrees equal to $n/2$.
For example, the complete bipartite graph $K_{n/2,n/2}$,
the union of two $(n/2)$-vertex complete graphs, or (with high probability) an Erd\H{o}s-R\'enyi random graph $G_{n,1/2}$ all minimize the sum.
However, the structure of an optimal solution can be made unique by adding additional density constraints.
In our example, setting the triangle density to be zero forces the structure to be that of the complete bipartite graph
with parts of equal sizes.
Alternatively, fixing the density of cycles of length four forces the structure to be that of a quasirandom graph.
The most frequently quoted conjecture concerning dense graph limits is
a conjecture of Lov\'asz~\cite{bib-lovasz-open,bib-lovasz-large,bib-lovasz-book,bib-lovasz11+}, stated as Conjecture~\ref{conj-main} below,
which asserts that this is a general phenomenon for a large class of problems in extremal graph theory.
We disprove this conjecture.

We treat Conjecture~\ref{conj-main} in the language of the theory of graph limits, to
which we provide a brief introduction.
This theory has offered analytic tools to represent and analyze large graphs, and has
led to new tools and views on various problems in mathematics and computer science.
We refer the reader to a monograph by Lov\'asz~\cite{bib-lovasz-book} for a detailed
introduction to the theory.
The theory is also closely related to the flag algebra method of Razborov~\cite{bib-razborov07},
which changed the landscape of extremal combinatorics~\cite{bib-razborov-interim}
by providing solutions and substantial progress on many long-standing open problems, see,
e.g.~\cite{bib-flag1, bib-flag2, bib-flagrecent, bib-flag3, bib-flag13, bib-flag4, bib-flag5, bib-flag6,
bib-flag7, bib-flag8, bib-flag9, bib-flag10, bib-razborov07,bib-flag11, bib-flag12}.

The \emph{density} of a $k$-vertex graph $H$ in $G$, denoted by $d(H,G)$,
is the probability that a uniformly randomly chosen $k$-tuple of vertices of $G$ induce a subgraph isomorphic to $H$;
if $G$ has less than $k$ vertices, we set $d(H,G)=0$.
A sequence of graphs $(G_n)_{n\in\NN}$
is \emph{convergent} if the sequence $(d(H,G_n))_{n\in\NN}$ is convergent for every graph $H$.
In this paper, we only consider convergent sequences of graphs where
the number of vertices tends to infinity.

A convergent sequence $(G_n)_{n\in\NN}$ of graphs is \emph{finitely forcible}
if there exist graphs $H_1,\ldots,H_{\ell}$ with the following property:
if $(G'_n)_{n\in\NN}$ is another convergent sequence of graphs such that
$$\lim_{n\to\infty}d(H_i,G_n)=\lim_{n\to\infty}d(H_i,G'_n)$$
for every $i=1,\ldots,\ell$, then
$$\lim_{n\to\infty}d(H,G_n)=\lim_{n\to\infty}d(H,G'_n)$$
for every graph $H$.
For example,
one of the classical results on quasirandom graphs~\cite{bib-thomason, bib-thomason2, bib-chung89+} is equivalent to saying that a sequence of Erd\H os-R\'enyi random
graphs is finitely forcible (by densities of $4$-vertex subgraphs) with probability one. Lov\'asz and S\'os~\cite{bib-lovasz08+} generalized this result 
to graph limits corresponding to stochastic block models (which are represented by step graphons).
Additional examples of finitely forcible sequences can be found, e.g., in~\cite{bib-perm,bib-lovasz11+}.

One of the most commonly cited problems concerning graph limits is the following conjecture of Lov\'asz,
which is often referred to as saying that ``every extremal problem has a finitely forcible optimum'',
see~\cite[p.~308]{bib-lovasz-book}.
The conjecture appears in various forms, also sometimes as a question, in the literature;
we include only some of the many references with its statement below.
\begin{conjecture}[{Lov\'asz \cite[Conjecture 3]{bib-lovasz-open},\cite[Conjecture 9.12]{bib-lovasz-large},\cite[Conjecture 16.45]{bib-lovasz-book}, and \cite[Conjecture 7]{bib-lovasz11+}}]
\label{conj-main}
Let $H_1,\ldots,H_\ell$ be graphs and $d_1,\ldots,d_\ell$ reals.
If there exists a convergent sequence of graphs with the limit density of $H_i$ equal to $d_i$, $i=1,\ldots,\ell$,
then there exists a finitely forcible such sequence.
\end{conjecture}
Our main result (Theorem~\ref{thm-main}) implies that the conjecture is false. 
The proof of Theorem~\ref{thm-main} uses analytic objects (graphons) provided by the theory of graph limits,
however, unlike in earlier results concerning finitely forcible graph limits,
we have parametrized a whole family of objects that represent large graphs and
applied analytic tools to understand the behavior of these objects.
To control the objects in the family, we have adopted the method of decorated constraints.
This method builds on the flag algebra method of Razborov, which we have mentioned earlier, and
was used to enforce an asymptotically unique structure of graphs.
In the setting of the proof of Theorem~\ref{thm-main}, the method is used to enforce
a unique structure inside a large graph while keeping some of its other parts variable.

We now explain the general strategy of the proof of Theorem~\ref{thm-main} in more detail.
By fixing finitely many subgraph densities,
we will enforce the asymptotic structure of graphs except for certain specific parts.
The structure of the variable parts will depend on a vector $z\in [0,1]^\NN$ in a controlled way,
in particular, one of the variable parts of the graphs will encode the values $z_i$ as edge densities.
We will also have a countable set of polynomial inequalities that the values $z_i$
will satisfy for any graph in the family that we consider;
this will be enforced by the variable parts and the subgraph densities.
This set of polynomial inequalities is constructed iteratively, and each time we add new constraints,
the graphs in the family change in a controlled way.
We eventually construct a large subset of $[0,1]^\NN$,
which is defined by a countable set of polynomial inequalities,
such that \emph{any} subgraph density in the graph corresponding to $z$ is determined by finitely many coordinates of $z$.
This will yield a counterexample to Conjecture~\ref{conj-main}.

While we present our arguments to avoid forcing the unique structure of a graph by subgraph densities,
our approach also applies to additional graph parameters, as discussed in Section~\ref{sec-concl}.
For example, one may hope that even if it is not possible to force an asymptotically unique structure
by fixing a finite number of subgraph densities,
it may be possible to force a unique typical structure,
i.e., obtain an asymptotically unique structure maximizing the entropy.
However, this hope is dismissed by Theorem~\ref{thm-entropy}.

The proof of Theorem~\ref{thm-main} presented in this paper is not constructive.
However, as we discuss in Section~\ref{sec-concl},
a constructive proof can be obtained using the main result of~\cite{bib-turing},
which is an earlier (but unrefereed) version of \cite{bib-universal}.
It is also worth noting that
all graphs $H_1,\ldots,H_{\ell}$ in the statement of Theorem~\ref{thm-main} have at most $225$ vertices;
we give further details in Section~\ref{sec-concl} after presenting the proof of Theorem~\ref{thm-main}.

The paper is organized as follows.
In Section~\ref{sec-prelim},
we start by introducing the notation that we use.
We next give a broad outline and set up some basic concepts for the proof of Theorem~\ref{thm-main} in Section~\ref{sec-gensetup},
including a summary of the construction of a family of graphons with finitely forcibly structure,
which is parameterized by $z\in [0,1]^\NN$.
We also state the properties of the construction that we need for the proof.
We next detail the proof, assuming such a construction exists.
We first present an iterative way of restricting values of
well-behaved functions in countably many variables in Section~\ref{sec-diag}.
In Section~\ref{sec-main}, we show how this implies Theorem~\ref{thm-main}.
In Section~\ref{sec-family},
we detail the construction of the family of graphons and prove that it can be forced by finitely many constraints.
Finally, in Section~\ref{sec-analyze}, we analyze the dependence of the graphons in the family on $z\in [0,1]^\NN$.

\section{Preliminaries}
\label{sec-prelim}

In this section, we introduce notation used throughout the paper.
We start with some general notation.
The set of all positive integers is denoted by $\NN$,
the set of all non-negative integers by $\NN_0$, and
the set of integers between $1$ and $k$ (inclusive) by $[k]$.
All measures considered in this paper are Borel measures on $\RR^d$, $d\in\NN$.
If a set $X\subseteq\RR^d$ is measurable, then we write $|X|$ for its measure, and
if $X$ and $Y$ are two measurable sets, then we write $X\sqsubseteq Y$ if $|X\setminus Y|=0$.

We will be working with vectors with both finitely many and countably infinitely many coordinates, and
sometimes with double coordinates.
If $X\subseteq [0,1]$ and $J$ is a set, then the set $X^J$ is the set of vectors with coordinates from $X$ indexed by $J$.
In particular, $[0,1]^\NN$ denotes the set of all vectors with coordinates indexed by positive integers and
each coordinate from the interval $[0,1]$.
We will always understand $X^d$ to be $X^{[d]}$.
If $x\in\RR^J$, $J'\subseteq J$ and $\eps>0$,
then we define
$$N_{\eps,J'}(x)=\{x'\mbox{ s.t. } |x_j-x'_j|<\eps \mbox{ for } j\in J', \mbox{ and }x_j=x'_j \mbox{ for } j\not\in J'\}\,\mbox{.}$$
If $X\subseteq\RR^J$, $J'\subseteq J$ and $\eps>0$,
then we set
\[N_{\eps,J'}(X)=\bigcup_{x\in X}N_{\eps,J'}(x)\,\mbox{.}\]
If $J'=J$, we will omit the second index in the subscript.
Finally, we use $\overline{N}_{\eps,J'}(X)$ for the closure of $N_{\eps,J'}(X)$.

We will often work with the following set of double indexed vectors:
$$\AAA=\prod_{i\in\NN}[0,1]^{i+1}\,\mbox{.}$$
The elements of $\AAA$ are vectors, each having coordinates indexed by $\NN$ and
the $i$-th coordinate being a vector from $[0,1]^{i+1}$.
For example, if $\vecz{a}\in\AAA$, then $\vecz{a}_2\in [0,1]^3$ and the coordinates of $\vecz{a}_2$ are $\vecz{a}_{2,1}$, $\vecz{a}_{2,2}$ and $\vecz{a}_{2,3}$.
We define the following bijection between the elements of $[0,1]^{\NN}$ and $\AAA$:
if $a\in [0,1]^{\NN}$, we define $\vecz{a}\in\AAA$ so that
$a=(\vecz{a}_{1,1},\vecz{a}_{1,2},\vecz{a}_{2,1},\vecz{a}_{2,2},\vecz{a}_{2,3},\ldots)$.
Using this bijection $a \to \vecz{a}$,
we will understand the corresponding elements of $[0,1]^{\NN}$ and $\AAA$ to be the same, and
we use the arrow overscript to distinguish between the corresponding elements of $[0,1]^{\NN}$ and $\AAA$,
i.e., whether the vector is single or double indexed.

We say that the function $f:[0,1]^{\NN}\to\RR$ is {\em totally analytic}
if the following two properties hold:
\begin{itemize}
	\item for any finite multiset $M$ of positive integers, the partial derivative 
	\[\frac{\partial^{|M|}}{(\partial z)^M} f(z)
	\] is a continuous function with respect to the product topology on $[0,1]^\NN$, and
	\item for any fixed $(z_{n+1}^0,z_{n+2}^0,\ldots)\in [0,1]^{\NN\setminus [n]}$, the function
	\[f(z_1,z_2,\ldots,z_n,z_{n+1}^0,\ldots)
	\] can be extended to an analytic function of the variables $(z_1,z_2,\ldots,z_n)$ on an open set containing $[0,1]^n$.
\end{itemize}
We say that two totally analytic functions are {\em $\eps$-close}
if the functions and all their first partial derivatives differ by at most $\eps$ on $[0,1]^\NN$.
We say that a function $f$ is a {\em polynomial} in $z_i$, $i\in\NN$,
if it is a polynomial in finitely many of the variables $z_i$, $i\in\NN$.
Clearly, every polynomial in $z_i$, $i\in\NN$, is totally analytic.

\subsection{Graph limits}

We now introduce some notions from the theory of graph limits,
which we have not covered in Section~\ref{sec-intro}.
To simplify our notation, if $G$ is a graph,
we will write $|G|$ for its order, i.e., its number of vertices.

Each convergent sequence of graphs $(G_n)_{n\in \NN}$ can be associated with an analytic limit object,
which is called a graphon. Formally,
a \emph{graphon} is a symmetric measurable function $W$ from $[0,1)^2$ to the unit interval $[0,1]$,
where \emph{symmetric} refers to the property that $W(x,y) = W(y,x)$ for all $x, y \in [0, 1)$.
We remark that it is more usual to define graphons as symmetric measurable functions from the closed unit square $[0,1]^2$;
however, both definitions represent the same notion and it is more convenient for us to work with half-open intervals.
Graphons are often visualized by a unit square filled in with different shades of gray
representing the values of $W(x,y)$ (with $0$ being represented by white and $1$ by black).
We will often refer to the points of $[0,1)$ as {\em vertices}, and
we think of a graphon $W$ intuitively as a continuous version of an adjacency matrix.
Because of this, the origin $(0,0)$ is always drawn in the top left corner in visualizations of graphons.

We next present a connection between convergent sequences of graphs and graphons.
Given a graphon $W\!$, a {\em $W$-random graph} of order $n$ is a graph obtained from $W$
by sampling $n$ points from $[0, 1)$ independently and uniformly at random,
associating each point with one of the $n$ vertices, and
joining two vertices by an edge with probability $W(x,y)$ where $x$ and $y$ are the points in $[0,1)$ associated to them.
The density of a graph $H$ in a graphon $W$, denoted by $d(H, W)$,
is the probability that a $W$-random graph of order $|H|$ is isomorphic to $H$.
Observe that the expected density of $H$ in a $W$-random graph of order $n\ge |H|$ is equal to $d(H,W)$. It also holds that
the density of $H$ in a $W$-random graph is concentrated around its expected density by standard martingale arguments.
We will also use the labeled version of the density of a graph $H$ in a graphon $W$,
which is the probability that a $W$-random graph is the graph $H$ and
the $i$-th vertex of the $W$-random graph is the $i$-th vertex of the graph $H$.
This probability will be denoted by $\tau(H,W)$ and it holds that
\[\tau(H,W)=\int_{[0,1)^{V(H)}} \prod_{\{u,v\} \in E(H)} W(x_u,x_v) \prod_{\{u,v\} \in \binom{V(H)}{2}\setminus E(H)} (1-W(x_u,x_v)) dx_{V(H)}
.\]
Note that $d(H,W)$ is equal to $\tau(H,W)$ multiplied by $|H|!/|\Aut(H)|$.

A graphon $W$ is a {\em limit} of a convergent sequence of graphs $(G_n)_{n\in \NN}$ if
$$\lim_{n\to\infty} d(H,G_n)=d(H,W)$$
for every graph $H$.
A standard martingale argument yields that
if $W$ is a graphon,
then a sequence of $W$-random graphs with increasing orders converges to $W$ with probability one.
In the other direction, Lov\'asz and Szegedy~\cite{bib-lovasz06+} showed that
every convergent sequence of graphs has a limit graphon.
However, the limit is in general not unique.
We say that two graphons $W_1$ and $W_2$ are {\em weakly isomorphic} if $d(H, W_1) = d(H, W_2)$ for every graph $H$,
i.e., the graphons $W_1$ and $W_2$ are limits of the same convergent sequences of graphs.
Borgs, Chayes and Lov\'asz~\cite{bib-borgs10+} proved that two graphons $W_1$ and $W_2$ are weakly isomorphic
if and only if there exists a third graphon $W$ and measure preserving maps $\varphi_1, \varphi_2: [0, 1)\to [0, 1)$ such that
$W_i(\varphi_i(x),\varphi_i(y))=W(x,y)$ for almost every $(x, y) \in [0, 1)^2$ and $i=1,2$.

Because of Conjecture~\ref{conj-main},
we are interested in graphons that are uniquely determined, up to weak isomorphism, by the densities of a finite set of graphs.
Formally, a graphon $W$ is {\em finitely forcible} if
there exist graphs $H_1,\ldots,H_\ell$ such that
if a graphon $W'$ satisfies $d(H_i, W')=d(H_i,W)$ for $i\in [\ell]$, then $W$ and $W'$ are weakly isomorphic.
The following characterization of finitely forcible graphons is one of the motivations coming from extremal graph theory
for Conjecture~\ref{conj-main}.
\begin{proposition}
\label{prop-extr}
A graphon $W$ is finitely forcible if and only if
there exist graphs $H_1,\ldots,H_\ell$ and reals $\alpha_1,\ldots,\alpha_\ell$ such that
$$\sum_{i=1}^\ell\alpha_i d(H_i,W)\le\sum_{i=1}^\ell\alpha_i d(H_i,W')$$
for every graphon $W'$, with equality if only if $W$ and $W'$ are weakly isomorphic.
\end{proposition}

While the initial results on the structure of finitely forcible graphons suggested that
all finitely forcible graphons could possess a simple structure~\cite[Conjectures~9 and 10]{bib-lovasz11+},
this turned out not to be the case~\cite{bib-reg,bib-inf,bib-comp} in general.
In particular, every graphon can be a subgraphon of a finitely forcible graphon~\cite{bib-universal}.
Here, we say that $W'$ is a {\em subgraphon} of $W$
if there exists a non-null subset $Z\subseteq [0,1)$ and a measure preserving map $\varphi:Z\to [0,|Z|)$ such that
$W(x,y)=W'\left(\varphi(x)/|Z|,\varphi(y)/|Z|\right)$ for all $x,y\in Z$.

\begin{theorem}[{Cooper et al.~\cite{bib-universal}}]
\label{thm-universal}
There exist graphs $H_1,\ldots,H_\ell$ with the following property.
For every graphon $W_F$,
there exist a graphon $W_0$ and reals $d_1,\ldots,d_\ell\in [0,1]$ such that any graphon $W$ with $d(H_i,W)=d_i$ for all $i\in [\ell]$ is weakly isomorphic to $W_0$, and
$W_F$ is a subgraphon of $W_0$ formed by a $1/14$ fraction of its vertices. %of $W_0$.
In particular, the graphon $W_0$ is finitely forcible.
\end{theorem}

\begin{figure}
\begin{center}
\epsfxsize 9cm
\epsfbox{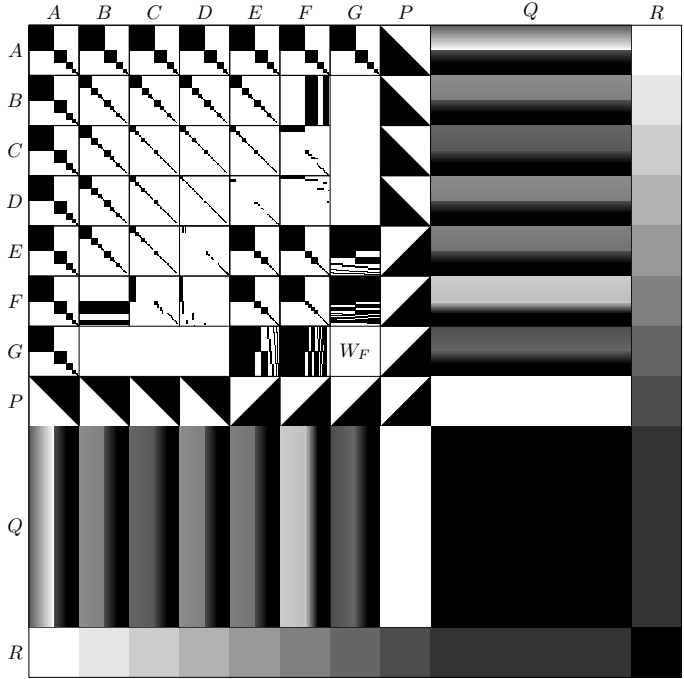}
\end{center}
\caption{The graphon $W_0$ from Theorem~\ref{thm-universal}.}
\label{fig-universal}
\end{figure}

The graphon $W_0$ from Theorem~\ref{thm-universal} is visualized in Figure~\ref{fig-universal}.
We will revisit the proof of Theorem~\ref{thm-universal} in Subsection~\ref{sub-blackbox} and
additional properties of the graphons $W_0$, which we state in the next proposition,
will be needed in Section~\ref{sec-analyze}.
A {\em dyadic square} is a square of the form
\mbox{$[(i-1)/2^k,i/2^k)\times [(j-1)/2^k,j/2^k)$}, where $k\in\NN_0$ and $i,j\in [2^k]$.
Recall that the \emph{standard binary representation} of a real number
is the binary representation such that set of digits equal to zero is not finite.

\begin{proposition}
\label{prop-change}
For every $\eps>0$, there exists $k\in\NN$ with the following property.
Let $W_F$ and $W'_F$ be any two graphons such that the densities of their dyadic squares of sizes at least $2^{-k}$
agree up to the first $k$ bits after the decimal point in the standard binary representation, and
let $W_0$ and $W'_0$ be the corresponding graphons from Theorem~\ref{thm-universal} that
contain $W_F$ and $W'_F$, respectively.
The $L_1$-distance between the graphons $W_0$ and $W'_0$ viewed as functions in $L^1([0,1)^2)$
is at most $\eps$.
\end{proposition}

We finish by presenting graph limit analogues of several standard graph theory notions.
Let $W$ be a graphon.
The {\em degree} of a vertex $x \in [0, 1)$ is defined as
$$\deg_W(x) = \int_{[0,1)} W(x,y) \dd y\,\mbox{.}$$
When it is clear from the context which graphon we are referring to, we will omit the subscript,
i.e., we will just write $\deg(x)$ instead of $\deg_W(x)$.
The {\em neighborhood} of $x$ is the set of all $y$ such that $W(x,y)>0$.
This set is denoted by $N_W(x)$, where we again drop the subscript if it is clear from the context.
Note that $\deg_W(x)\le |N_W(x)|$ and the inequality can be strict.
If $A$ is a non-null subset of $[0,1)$,
then the {\em relative degree} of a vertex $x\in [0,1)$ with respect to $A$ is
$$\deg_W^A(x) = \frac{\int_A W(x,y) \dd y}{|A|},$$
i.e., the degree of $x$ with respect to $A$ normalized by the measure of $A$.
Similarly, $N_W^A(x)=N_W(x)\cap A$ is the {\em relative neighborhood} of $x$ with respect to $A$.
As in the previous cases, we drop the subscripts if $W$ is clear from the context.

\subsection{Decorated constraints}

The arguments related to forcing the structure of graphons introduced in Section~\ref{sec-family}
are based on the method of decorated constraints from~\cite{bib-inf,bib-comp}.
This method uses the language of the flag algebra method of Razborov developed in~\cite{bib-razborov07},
which has had many significant applications in extremal combinatorics~\cite{bib-razborov-interim}.

We now give a formal definition of decorated constraints.
A {\em density expression} is a polynomial in graphs with real coefficients.
Formally, it can be defined recursively as follows.
A real number or a graph are density expressions, and
if $D_1$ and $D_2$ are density expressions, then $D_1 + D_2$ and $D_1 \cdot D_2$ are also density expressions.
An {\em ordinary density constraint} is an equality between two density expressions.
The value of a density expression for a graphon $W$
is obtained by substituting $\tau(H,W)$ for each graph $H$.
A graphon $W$ {\em satisfies} an ordinary density constraint
if the density expressions on the two sides have the same value for $W$.
Note that we obtain an equivalent notion if we use $d(H,W)$ instead of $\tau(H,W)$
when computing the value of a density of expression (adjusting the coefficients appropriately).

We next introduce a formally stronger type of density constraint: a rooted density constraint.
A \emph{rooted graph} is a graph $H$ with $m$ distinguished vertices,
which are labeled with the elements of $[m]$ and are referred to as the \emph{roots}.
Let $H_0$ be the subgraph of $H$ induced by the roots.
We then define $\tau(H,W|x_1,\ldots,x_m)$ to be the probability that
the $W$-random graph is $H$, conditioned on $x_i$ being associated with the $i$-th root.
It follows that 
\[\tau(H,W|x_1,\ldots,x_m)=
\int\limits_{[0,1)^{U}} \prod_{\{u,v\} \in E(H)} W(x_u,x_v)\mkern-9mu
                 \prod_{\{u,v\} \in \binom{V(H)}{2}\setminus E(H)}\mkern-9mu (1-W(x_u,x_v)) \dd x_{U}
,\]
where $U$ is the set of the non-root vertices of $H$.
Two rooted graphs are \emph{compatible}
if the subgraphs induced by their roots are isomorphic through an isomorphism preserving their labels.
A \emph{rooted density expression} is a density expression containing compatible rooted graphs.
We define $\tau(C,W|x_1,\ldots,x_m)$ to be the value of the expression $C$
when each rooted graph $H$ is substituted with $\tau(H,W|x_1,\ldots,x_m)$.
Let $C_1=C_2$ be a constraint where all rooted graphs are compatible.
We say that the graphon $W$ {\em satisfies} the constraint $C_1=C_2$
if 
\[\tau(C_1,W|x_1,\ldots,x_m)=\tau(C_2,W|x_1,\ldots,x_m)\]
holds for almost every $m$-tuple $x_1,\ldots,x_m\in [0,1)$.
Note that if the $m$-tuple $x_1,x_2,\ldots,x_m$ is chosen in such a way that
a $W$-random graph with the vertices associated with $x_1,\ldots,x_m$ is $H_0$ with probability zero,
then both sides of the above equality are equal to zero.

A graphon $W$ is said to be {\em partitioned}
if there exist an integer $k \in \NN$, positive reals $a_1,\ldots, a_k$ summing to one, and
distinct reals $d_1,\ldots,d_k\in [0, 1]$, such that for each $i \in [k]$, the set of vertices in $W$ with degree $d_i$ has measure $a_i$.
The set of all vertices with degree $d_i$ will be referred to as a {\em part};
the {\em size} of a part is its measure and its {\em degree} is the common degree of its vertices.
For example, the graphon depicted in Figure~\ref{fig-universal} has ten parts and all but $Q$ have the same size.
If $X$ and $Y$ are parts, then we will refer to the restriction of $W$ to $X\times Y$ as to the {\em tile} $X\times Y$.
The following lemma was proven in~\cite{bib-inf,bib-comp}.

\begin{lemma}
\label{lm-partition}
Let $a_1,\ldots,a_k$ be positive real numbers such that $a_1+\ldots+a_k=1$ and let $d_1,\ldots,d_k\in [0,1]$ be distinct reals.
There exists a finite set of ordinary density constraints $\CC$ such that a
graphon satisfies all constraints in $\CC$ if and only if it is a partitioned graphon with parts of sizes $a_1,\ldots,a_k$ and degrees $d_1,\ldots,d_k$.
%any graphon satisfying all constraints in $\CC$ is a partitioned graphon with parts of sizes $a_1,\ldots,a_k$ and degrees $d_1,\ldots,d_k$, and
%every partitioned graphon with parts of sizes $a_1,\ldots,a_k$ and degrees $d_1,\ldots,d_k$ satisfies all constraints in $\CC$.
\end{lemma}

We next introduce a formally even stronger type of density constraint.
Fix $a_1,\ldots$, $a_k$ and $d_1,\ldots,d_k$ as in Lemma~\ref{lm-partition}.
A {\em decorated} graph is a graph $H$ with $m\le |H|$ distinguished vertices,
which are called {\em roots} and labeled with $[m]$, and
with each vertex (distinguished or not) assigned one of the $k$ parts, which will be called the {\em decoration} of a vertex.
We allow $m=0$, i.e., a decorated graph can have no roots.
Suppose that $x_1,\ldots,x_m \in [0,1)$ are given such that each $x_i$ belongs to the part that the $i$-th root is decorated with.
We then define $\tau(H,W|x_1,\ldots,x_m)$
to be the probability that a $W$-random graph is $H$, conditioned on the $i$-th root being associated with $x_i$, and
conditioned on each non-root vertex being associated with a point from the part that it is decorated with.
Two decorated graphs are {\em compatible} if the subgraphs induced by their roots are isomorphic
through an isomorphism preserving the labels of the roots and the decorations of all vertices. 
A \emph{decorated density expression} $C$ is an expression that contains compatible decorated graphs only, and
we define $\tau(C,W|x_1,\ldots,x_m)$ in the obvious way.
A {\em decorated density constraint} is an equality between two density expressions that
contain compatible decorated graphs instead of ordinary graphs;
we will often say \emph{decorated constraint} instead of decorated density constraint.

Before defining when a decorated constraint is satisfied,
let us introduce a way of visualizing decorated graphs and decorated constraints,
which was also used in~\cite{bib-reg,bib-universal,bib-inf}.
We will draw decorated graphs as graphs with each vertex labeled by its decoration;
the root vertices will be drawn as squares and non-root vertices as circles.
The decorated constraint will be presented as an expression containing decorated graphs,
where the roots of all the graphs appearing in the constraint will be presented in the same position
in all the graphs of the constraint; this establishes the correspondence between the roots
in individual decorated graphs.
A solid line between two vertices will represent an edge and a dashed line a non-edge.
The absence of a line between two root vertices indicates that the decorated constraint
should hold for both the root graph containing this edge and the one not containing it.
The absence of a line between a non-root vertex and another vertex represents the sum over all decorated graphs
with and without this edge.
In particular,
if there are $k$ such lines missing, the drawing represents the sum of $2^k$ decorated graphs.

We next define when a graphon $W$ satisfies a decorated constraint;
the definition is followed by an example of evaluating a decorated graph. 
Suppose that $W$ is a partitioned graphon and $C_1=C_2$ is a decorated constraint such that
each decorated graph in the constraint has (the same) $m$ roots.
The graphon $W$ {\em satisfies} the constraint $C$ if 
\[\tau(C_1,W|x_1,\ldots,x_m)=\tau(C_2,W|x_1,\ldots,x_m)\]
for almost every $m$-tuple $x_1,\ldots,x_m\in [0,1)$ such that
$x_i$ belongs to the part that the $i$-th root is decorated with.

To illustrate this definition, we give an example.
Consider the partitioned graphon $W$ that has two parts $A$ and $B$, each of size $1/2$, and
that is equal to $2/3$, $1/3$ and $1$ on the tiles $A\times A$, $A\times B$ and $B\times B$, respectively.
The graphon is depicted in Figure~\ref{fig-ex-dec};
the three decorated graphs depicted in Figure~\ref{fig-ex-dec}
are evaluated to $8/27$, $4/9$ and $16/243$ (from left to right) for any admissible choices of the roots.

\begin{figure}
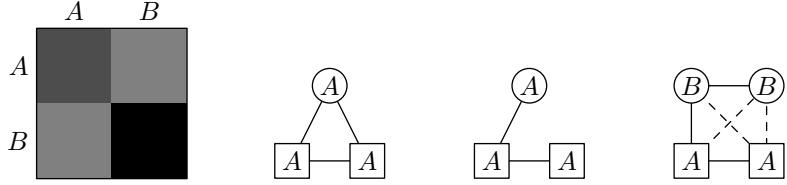

\begin{center}
\epsfbox{exturing-4.mps}
\hskip 10mm
\epsfbox{exturing-5.mps}
\hskip 10mm
\epsfbox{exturing-6.mps}
\hskip 10mm
\epsfbox{exturing-7.mps}
\end{center}
\caption{An example of evaluating decorated constraints. The depicted partitioned graphon has two parts $A$ and $B$, each of size $1/2$, and
is equal to $2/3$, $1/3$ and $1$ on the tiles $A\times A$, $A\times B$ and $B\times B$, respectively.
The three decorated graphs that are depicted are evaluated to $8/27$, $4/9$ and $16/243$ (from left to right) for any admissible choices of the roots.}
\label{fig-ex-dec}
\end{figure}

In~\cite{bib-comp}, it was shown that the expressive power of ordinary and decorated density constraints is the same. To be precise, the following was proven.

\begin{lemma}
\label{lm-decorated}
Let $k\in\NN$, let $a_1,\ldots,a_k$ be positive real numbers summing to one, and
let $d_1,\ldots,d_k$ be distinct reals between zero and one.
For every decorated density constraint $C$, there exists an ordinary density constraint $C'$ such that
every partitioned graphon $W$ with parts of sizes $a_1,\ldots,a_k$ and degrees $d_1,\ldots,d_k$
satisfies $C$ if and only if it satisfies $C'$.
\end{lemma}

Lemmas~\ref{lm-partition} and~\ref{lm-decorated} were used in~\cite{bib-reg,bib-universal,bib-inf,bib-comp}
to argue that certain partitioned graphons are finitely forcible in the following way.
Suppose that $W$ is the unique partitioned graphon (up to weak isomorphism) that satisfies a finite set $\CC$ of decorated constraints.
By Lemmas~\ref{lm-partition} and~\ref{lm-decorated}, the graphon $W$ is the unique graphon that satisfies
a finite set $\CC'$ of ordinary decorated constraints. It follows that the graphon $W$ is the unique graphon such that
the density of each subgraph $H$ appearing in a constraint of $\CC'$ is equal to $d(H,W)$.
In particular, $W$ is finitely forcible.
In this paper, we follow a similar line of arguments to show that a certain family of graphons
can be forced by finitely many subgraph densities.

We next state two auxiliary lemmas,
which can be found explicitly stated in~\cite[Lemmas 7 and 8]{bib-reg}.
The first lemma says that if a graphon $W_0$ is finitely forcible,
then its structure may be forced on a part of a partitioned graphon.
The second lemma is implicit in~\cite[proof of Lemma 2.7 or Lemma 3.3]{bib-lovasz11+}.

\begin{lemma}
\label{lm-subforcing}
Let $k$ be an integer, $a_1,\ldots,a_k$ positive reals summing to one, and
$d_1,\ldots,d_k$ distinct reals between zero and one.
For every finitely forcible graphon $W_0$ and every $m\in [k]$,
then there exists a finite set $\CC$ of decorated constraints such that
any partitioned graphon $W$ with parts $A_1,\ldots,A_k$ of sizes $a_1,\ldots,a_k$ and degrees $d_1,\ldots,d_k$ satisfies $\CC$
if and only if the tile $A_m\times A_m$ is weakly isomorphic to $W_0$,
i.e., there exist measure preserving maps $\varphi_0:[0,1)\to [0,1)$ and $\varphi_m:[0,a_m)\to A_m$ such that
$W(\varphi_m(xa_m),\varphi_m(ya_m))=W_0(\varphi_0(x),\varphi_0(y))$ for almost every $(x,y)\in [0,1)^2$.
\end{lemma}

\begin{lemma}
\label{lm-typ-pairs}
Let $F:X\times Z\to [0,1]$ be a measurable function, $X,Z\subseteq [0,1)$, such that
$$\int_Z F(x,z)F(x',z) \dd z = \xi$$
for almost every $(x,x')\in X^2$.
Then, it holds that
$$\int_Z F(x,z)^2 \dd z = \xi$$
for almost every $x\in X$.
\end{lemma}

\section{General setup}
\label{sec-gensetup}

The proof of our main theorem, Theorem~\ref{thm-main},
builds on techniques developed in relation to the finite forcibility of graphons.
We find a finite set of density constraints with the property that
any graphon that satisfies these constraints has most of its structure forced,
except for a small part that can vary in a controlled manner.
In particular, we will construct a family of graphons with structure depending
on a bounding sequence, which is a notion that we define below, and
the graphons in the family are indexed by elements of a subset of $[0,1]^\NN$.

Before proceeding further, we need to introduce some definitions.
We first define a linear order $\preceq$ on finite multisets of natural numbers.
Each such multiset $A$ can be associated with a vector $\chi(A)\in\NN_0^\NN$,
where $\chi(A)_n$ is the multiplicity of the containment of $n$ in $A$. Let $\Sigma(A)$ be the sum of the elements of $A$, i.e.,
$$\Sigma(A)=\sum_{n\in\NN}n\cdot\chi(A)_n\,\mbox{.}$$
If $A$ and $B$ are two multisets of natural numbers, then we will say that $A\preceq B$
if either $A=B$, or $\Sigma(A)<\Sigma(B)$, or
$\Sigma(A)=\Sigma(B)$ and the first entry different in $\chi(A)$ and $\chi(B)$ is smaller in $\chi(A)$.
Let $M_i$ be the $i$-th multiset in this linear order,
i.e., $M_1=\emptyset$, $M_2=\{1\}$, $M_3=\{2\}$, $M_4=\{1,1\}$, etc.
Since there exists at least one multiset $M$ with $\Sigma(M)=n$ for each $n$,
we get the following observation.
\begin{observation}
\label{obs-monomialsadditive}
It holds that $\Sigma(M_i)\le i$ for every $i\in\NN$.
\end{observation}

We call a sequence $P=(p_i,l_i,u_i)_{i\in\NN}$ a \emph{bounding sequence} if
\begin{itemize}
        \item each $p_i$ is a polynomial and $p_i(z)\in [0,1]$ for all $z\in [0,1]^\NN$,
	\item the coefficient of $\pi_{i,j}$ in the polynomial $p_i$ at the monomial $z^{M_j}$ satisfies $|\pi_{i,j}|\le 2^{-2^j}/9$,
	\item all the partial derivatives of $p_i$ have values between $-1$ and $+1$ (inclusive) on $[0,1]^\NN$,
	\item $l_i\in [0,1]$, $u_i\in [0,1]$ and $l_i\le u_i$.
\end{itemize}
Each bounding sequence describes a subset of $[0,1]^\NN$,
which simply contains all $z\in [0,1]^\NN$ such that $p_i(z) \in [l_i,u_i]$ for each $i$, and
we will use such subsets to describe a subfamily graphons.
The second condition in the definition of a bounding sequence is needed,
in particular, in order to make sure we can encode the polynomial in the graphon.
We say that a bounding sequence $P'$ is a \emph{strengthening} of the bounding sequence $P$
if they agree on all triples except some triples where $(p_i,l_i,u_i)=(0,0,1)$ in $P$.
Note that if $P'$ is a strengthening of $P$,
then the subset of $[0,1]^\NN$ associated with $P'$ is contained in the subset associated with $P$.
If $P'$ is a strengthening of $P$ that agrees with $P$ on the first $k$ elements,
then we say that $P'$ is a \hbox{\emph{$k$-strengthening}}.

In Section~\ref{sec-family}, we will construct, for any bounding sequence $P$, a family $\WW_P$ of graphons $W_P(z)$. Each of the graphons $W_P(z)$ is described by
a bounding sequence $P=(p_i,l_i,u_i)_{i\in\NN}$ and
a vector $z\in [0,1]^{\NN}$ that satisfies $p_i(z) \in [l_i,u_i]$ for each $i$.
These families of graphons will have the following properties.

\begin{theorem}
\label{thm-WPforced}
There exists a finite set of graphs $H_1,H_2,\ldots,H_\ell$ and an integer $D$ such that the following holds.
For any bounding sequence $P=(p_i,l_i,u_i)_{i\in\NN}$,
there exists a polynomial $q$ of degree at most $D$ in $\ell$ variables such that
the following two statements are equivalent for any graphon $W$:
\begin{itemize}
\item The graphon $W$ is weakly isomorphic to a graphon $W_P(z)$ for some $z\in[0,1]^\NN$ such that
      $p_i(z)\in [l_i,u_i]$ for all $i\in\NN$.
\item The graphon $W$ satisfies $q(d(H_1,W),d(H_2,W),\ldots,d(H_\ell,W))=0$.
\end{itemize}
Furthermore, if $W_P(z)$ and $W_P(z')$ are weakly isomorphic, then $z=z'$.
\end{theorem}

We defer the precise description of the family $\WW_P$ and the proof of Theorem~\ref{thm-WPforced} to Section~\ref{sec-family}.
For the proof of our main result, we will need additional properties of graphons in the family $\WW_P$,
which we summarize in the next two lemmas.
In Section \ref{sec-analyze},
we define a function $t_{P,H}:[0,1]^\NN\to\RR$ for each bounding sequence $P$ and each graph $H$.
The next proposition asserts that $t_{P,H}$ is a totally analytic function and
it describes the density of $H$ for those $z$ that satisfy $P$.

\begin{lemma}
\label{lm-densityanalytic}
For any finite graph $H$ and bounding sequence $P=(p_i,l_i,u_i)_{i\in\NN}$,
the function $t_{P,H}(z)$ is totally analytic. Furthermore, for every $z\in[0,1]^\NN$ satisfying $p_i(z) \in [l_i,u_i]$ for all $i \in \NN$,
it holds that $t_{P,H}(z)=\tau(H,W_P(z))$.
\end{lemma}

We will also need the following proposition,
which roughly says that if we change $P$ in a controlled way,
then the functions $t_{P,H}$ do not change by much.

\begin{lemma}
\label{lm-WPvaries}
For any finite graph $H$, bounding sequence $P$, and $\eps>0$, there exists an integer $k$ such that for any $k$-strengthening $P'$ of $P$,
the function $t_{P',H}$ is $\eps$-close to $t_{P,H}$.
\end{lemma}

Lemmas~\ref{lm-densityanalytic} and~\ref{lm-WPvaries} will be proven in Section~\ref{sec-analyze}.

\section{Stabilization}
\label{sec-diag}

Our main argument involves a sequence of steps,
each resulting in a strengthening of a bounding sequence that defines a family of graphons as above.
The whole process is carried in a careful way so that it is still possible to vary some of the parameters
without changing the densities of any given finite set of graphs.
In this section, we present analytic lemmas needed to perform these steps.

We first prove the following analytic statement. The proof follows the lines of the proof of Gr\"onwall's 
inequality \cite{bib-gronwall}, see also~\cite{bib-bellman}, however,
we present an entire proof here for completeness.

\begin{lemma}\label{lm-implicitsmallchangeclose} 
For every integer $n\in\NN$, reals $L>0$, $D>0$ and $\eps>0$, and closed interval $[a,b]\subseteq\RR$,
there exist $\delta>0$ and $\eps_0\in(0,\eps)$ such that the following holds.
Let $g:[a,b]\to\RR^n$ be a function defined on $[a,b]$ and let $V=\{(x,N_\eps(g(x))):x\in [a,b]\}$.
Furthermore, let $f : V \to \RR^n$ be an analytic function such that
$f(x,g(x))=0$ for all $x \in [a,b]$,
$f$ and each partial derivative of $f$ is $L$-Lipschitz on $V$, and
the absolute value of the determinant of the Jacobi matrix
\[\left(\frac{\partial}{\partial y_i} f_j(x,y)\right)_{i,j=1}^n\]
is at least $D$ on $V$.
For every analytic function $\wf : V \to \RR^n$ that is $\delta$-close to $f$ and
for every $(x_0,y_0)\in V$ such that $x_0\in [a,b]$, $\wf(x_0,y_0)=0$ and $\|g(x_0)-y_0\|_\infty \le \eps_0$,
there exists a function $\wg : [a,b] \to V$ such that
$\wg(x_0)=y_0$, $\|\wg(x)-g(x)\|_\infty \le \eps$ and $\wf(x,\wg(x))=0$ for all $x \in [a,b]$.
\end{lemma}

\begin{proof}
Fix a point $(x_0,y_0)\in V$ and the functions $f$, $\wf$ and $g$ for the proof.
Let $A(x,y_1,\ldots,y_n)$ be the $n\times n$ Jacobi matrix of $f$ with respect to $y_1,y_2,\ldots,y_n$, and
let $h(x,y_1,\ldots,y_n)$ be the partial derivative of $f$ with respect to $x$.
We define $\wA$ and $\wh$ for $\wf$ in an analogous way.
Take $\delta$ small enough so that 
the absolute value of the determinant of $\wA$ is at least $D/2$ on the whole set $V$ (note that
this choice of $\delta$ depends only on $n$, $D$ and $L$);
in particular, $\wA$ is invertible on whole set $V$.
By the Implicit Function Theorem,
if $\wg$ is defined for $x\in (a,b)$, $\wf(x,\wg(x))=0$ and $\wA(x,g(x))$ is invertible,
then $\wg$ can be extended to an open neighborhood of~$x$.
Let $U \subseteq [a,b]$ be the maximal interval containing $x_0$ that
$\wg$ can be extended to in a way that $\wf(x,\wg(x))=0$ and $\|\wg(x)-g(x)\|_\infty \le \eps$.

We next analyze the derivative of $f(x,g(x))$ as a function of $x$ on $(a,b)$:
\[\frac{\partial}{\partial x}f(x,g(x))=h(x,g(x))+A(x,g(x))\frac{\partial}{\partial x} g(x).\]
Since this derivative is zero on $(a,b)$, it follows that
\[\frac{\partial}{\partial x} g(x)=-(A(x,g(x)))^{-1} h(x,g(x)).\]
Similarly, it holds that
\[\frac{\partial}{\partial x} \wg(x)= -(\wA(x,\wg(x)))^{-1} \wh(x,\wg(x))\]
for any $x$ in the interior of $U$.
Let $m(x,y)=-(A(x,y))^{-1}h(x,y)$ and analogously $\wm(x,y)=-(\wA(x,y))^{-1}\wh(x,y)$.
Observe that both $m$ and $\wm$ are analytic and
$\|m(x,y)-\wm(x,y)\|_2 \le \rho$ for all $(x,y)\in V$,
where $\rho$ is a constant depending only on $n$, $D$, $L$, $\eps$ and $\delta$.
In addition, if $n$, $D$, $L$, and $\eps$ are fixed,
$\rho$ can be made arbitrarily close to zero by choosing $\delta$ sufficiently small.

Let us now consider the function
\[G(x)=\|g(x)-\wg(x)\|_2^2=\sum_{i=1}^n (g_i(x)-\wg_i(x))^2\]
defined on $U$.
Note that $G(x_0) \le n\eps_0^2$.
We next analyze the derivative of $G(x)$.
\begin{align*}
\frac{\partial}{\partial x} G(x) &= 2 \left\langle \frac{\partial}{\partial x} g(x)-\frac{\partial}{\partial x}\wg(x),g(x) -\wg(x) \right\rangle\\
  &= 2\left\langle m(x,g(x))-\wm(x,\wg(x)),g(x) - \wg(x) \right\rangle\\
  &= 2\left\langle m(x,g(x))-m(x,\wg(x)),g(x) - \wg(x) \right\rangle\\
  &+ 2\left\langle m(x,\wg(x))-\wm(x,\wg(x)),g(x) - \wg(x) \right\rangle.
\end{align*}  
It follows that
\begin{equation} \label{eq-exppartialderivbound}
\left|\frac{\partial}{\partial x} G(x)\right|\le K\cdot G(x)+2\rho \sqrt{G(x)} \le (K+\rho)G(x)+\rho,
\end{equation}
where $K$ depends only on $n$, $L$, $D$ and $\delta$ (since each entry of the matrix $A(x,y)^{-1}$ is at most $n!L^n/D$, and
each entry of the matrix is Lipschitz on $V$ for a Lipschitz constant depending only on $n$, $L$ and $D$).
We now consider the function \hbox{$H:U\to\RR$} that is defined as
\begin{equation} \label{eqn-Hdef}
H(x)=\left(G(x)+\frac{\rho}{K+\rho}\right)e^{-(x-x_0)(K+\rho)}.
\end{equation}
The derivative of $H$ is equal to
\[\frac{\partial}{\partial x} H(x)=\left(\frac{\partial}{\partial x}G(x) - (K+\rho)G(x)-\rho\right)e^{-(x-x_0)(K+\rho)} \le 0,\]
where the inequality follows from \eqref{eq-exppartialderivbound}.
This implies that $H(x) \le H(x_0)$ for $x\in U$ such that $x\ge x_0$.
Hence, we obtain using \eqref{eqn-Hdef} that
\[\left(G(x)+\frac{\rho}{K+\rho}\right) e^{-(x-x_0)(K+\rho)} \le G(x_0)+\frac{\rho}{K+\rho} \le n\eps_0^2 +\frac{\rho}{K+\rho}.\]
In particular, it is possible to choose $\eps_0$ and $\delta$ small enough (the choice of $\delta$ makes $\rho$ small enough) that
it holds that
\[\left(n\eps_0^2+\frac{\rho}{K+\rho}\right)e^{(b-x_0)(K+\rho)} <\eps^2.\]
Hence, we obtain that
\[G(x) \le \left(n\eps_0^2+\frac{\rho}{K+\rho}\right)e^{(x-x_0)(K+\rho)}-\frac{\rho}{K+\rho}<\eps^2.\]
for all $x\in U$ such that $x\ge x_0$.

We next show that the interval $[x_0,b]$ is contained in $U$.
Let $\beta$ be the supremum of $U$ and suppose that $\beta<b$.
Since the left limit $y_{\beta} = \lim_{x \nearrow \beta} \wg(x)$ satisfies that $\wf(\beta,y_{\beta})=0$,
we can assume that $\beta\in U$.
However, this implies that $\|\wg(\beta)-g(\beta)\|_\infty\le\|\wg(\beta)-g(\beta)\|_2<\eps$ and
the function $\wg$ can be extended to a neighborhood of $\beta$,
contradicting the fact that $\beta$ is the supremum of $U$.

An analogous argument using the function 
$$H(x)=\left(G(x)+\frac{\rho}{K+\rho}\right)e^{(x-x_0)(K+\rho)}$$ 
yields that 
it is possible to choose $\eps_0$ and $\delta$ small enough that
$G(x)\le\eps^2$ for all $x\in U$ such that $x\le x_0$,
which then yields that the interval $[a,x_0]$ is a subset of~$U$.
This completes the proof.
\end{proof}

To state the next lemma, we need a definition.
A \emph{variable system} consists of
\begin{itemize}
\item a sequence of non-degenerate closed intervals $U_k \subseteq [0,1]$, $k \in \NN$, and
\item a sequence of closed sets $V_k$, $k\in\NN$, with $V_1 \subseteq U_1$ and $V_k \subseteq V_{k-1} \times U_{k}$.
\end{itemize}
We say that a variable system is \emph{$c$-strong} for $c>0$
if the measure of the set $\{x':(x,x') \in V_k\}$ is at least $c|U_k|$ for every $x \in V_{k-1}$.
A variable system is \emph{$k_0$-finite} if $V_k=V_{k-1} \times U_k$ for each $k>k_0$.

We will need the following lemma on variable systems.

\begin{lemma} \label{lm-Vkshrink}
Suppose that $(U_k,V_k)_{k \in \NN}$ is a $c$-strong variable system with $c>0$, and
$T$ is a non-zero analytic function on $V_\ell$, $\ell\in\NN$.
For any $c'<c$, there exist subsets $V_k' \subseteq V_k$, $k \in [\ell]$, such that
$(U_k,V_k')_{k \in \NN}$, with $V_k'=V_k \cap (V_{k-1}' \times U_k)$ for $k>\ell$,
is a $c'$-strong variable system and $T$ is non-zero on $V_\ell$ everywhere.
\end{lemma}

\begin{proof}
We first consider the case when $\ell=1$.	
Let $Z$ be the set of all $x\in U_1$ such that $T(x)=0$.
Since $T$ is analytic and $U_1$ is bounded, the set $Z$ is finite.
Let $X$ be the union of open $\delta$-neighborhoods of the points in $Z$,
where $\delta>0$ is sufficiently small so that $|X|<(c-c')|U_1|$.
The set $V_1'=V_1\setminus X$ is a closed set that satisfies the statement of the lemma.
	
The argument for $\ell>1$ extends the one presented for $\ell=1$ but is more complex.
Let $\eps=c-c'$.
Let $Z$ be the set of all $x\in V_{\ell}$ such that
$$T(x_1,\ldots,x_{\ell})=0\,\mbox{.}$$
Since $T$ is analytic and not everywhere zero, $Z$ is a closed set with measure zero.
	
Let $X\subseteq V_{\ell}$ be the set of points with $|T(x)|<\delta$ 
where $\delta>0$ is sufficiently small so that the measure of $X$ is less than $\eps^{\ell}\prod_{j=1}^{\ell}|U_j|$. Note that $X$ is open.
For every $i\in[\ell-1]$,
let $X_{i}$ be the set of the points $x\in V_i$ such that
the measure of the set $\{x'\in \prod_{i<j \le \ell}U_j, (x,x') \in X\}$
is greater than $\eps^{\ell-i}\prod_{j=i+1}^{\ell}|U_j|$.
Let $X_{\ell}=X$.
Observe that each of the sets $X_i$, $i\in [\ell]$, is open.
In addition,
the measure of $X_1$ is at most $\eps|U_1|$, and
for any $i\in[\ell-1]$ and $x \in V_{i}$,
if the measure of the points $x' \in U_i$ such that $(x,x')\in X_{i+1}$ is greater than $\eps|U_{i+1}|$,
then $x\in X_{i}$.

We now define $V_i'$ for each $i\in [\ell]$ in increasing order
by setting $V_i'=\left(V_i \setminus X_{i}\right) \cap \left( V_{i-1}' \times U_i \right)$.
Since each set $X_i$ is open, we get that $V_i'$ is closed.
If $x\in V_i \setminus X_{i}$ for $i\in[\ell-1]$,
then the measure of $x'\in U_{i+1}$ such that $(x,x')\in X_{i+1}$ is at most $\eps|U_{i+1}|$.
It follows that for every $x\in V'_i$,
the measure of $x'\in U_{i+1}$ such that $(x,x')\in V'_{i+1}$ is at least $(c-\eps)|U_{i+1}|=c'|U_{i+1}|$.
Since the measure of $X_1$ is at most $\eps|U_1|$,
it also follows that the measure of $V'_1$ is at least $(c-\eps)|U_1|=c'|U_1|$.
We set $V_k'=V_k \cap (V_{k-1}' \times U_k)$ for $k>\ell$, and conclude that $(U_k,V_k')_{k \in \NN}$is a $c'$-strong variable system.
\end{proof}

A \emph{stabilizing system} $S=(U_k,V_k,I_k,J_k,d_k,b_k,w_k)_{k\in\NN}$ is a sequence of $7$-tuples such that
\begin{itemize}
\item $(U_k,V_k)_{k\in\NN}$ is a variable system,
\item $I_k$ is a subset of $[k]$,
\item $J_k$ is a subset of $[k+1]$ with $|I_k|=|J_k|$,
\item $d_k\in [k+1]$ such that $d_k\not\in J_k$,
\item $b_k\in (0,1)^{k+1}$ and $b_{k,d_k}$ is in the interior of $U_k$, and
\item $w_k$ is an analytic function from $V_{k-1}\times U_k\to [0,1]^{J_k\cup\{d_k\}}$ such that
      $w_k(x)_{d_k}=x_{k}$ for every $x\in V_{k-1}\times U_k$ and
      $w_k(x,b_{k,d_k})_i=b_{k,i}$ for every $x\in V_{k-1}$ and $i\in J_k$.
\end{itemize}
The functions $w_k$, $k\in\NN$, will be referred to as \emph{stabilizers}.
We will say that a stabilizing system is \emph{$c$-strong} if the variable system $(U_k,V_k)_{k\in\NN}$ is $c$-strong.
Finally, a stabilizing system is \emph{$k_0$-finite}
if the variable system $(U_k,V_k)_{k\in\NN}$ is \hbox{$k_0$-finite}, and
$J_k = \emptyset$ and $d_k=1$ for every $k>k_0$;
we will say that a stabilizing system is \emph{finite} if it is $k_0$-finite for some $k_0\in\NN$.

Given a stabilizing system $S$,
we define the \emph{extended stabilizers} $\wtw_k:V_{k-1} \times U_k\to [0,1]^{k+1}$
by
\[\wtw_k(x_1,x_2,\ldots,x_k)_i=
\begin{cases}
w_k(x_1,x_2,\ldots,x_k)_i & \text{if $i \in J_k\cup\{d_k\}$, and}\\
b_{k,i} & \text{otherwise.}
\end{cases}
\]
We also define a function $\wtw_{\le k}:V_{k-1} \times U_k\to [0,1]^{2+3+\ldots+(k+1)}$ as
\[\wtw_{\le k}(x_1,\ldots,x_k)=(\wtw_1(x_1),\wtw_2(x_1,x_2),\ldots,\wtw_k(x_1,x_2,\ldots,x_k)).\]
Finally, we write $b_{>k}$ for the point $(b_{k+1},b_{k+2},\ldots)$.

Suppose that $S$ is a stabilizing system and
$t_k:[0,1]^\NN \to [0,1]$, $k \in \NN$, are totally analytic functions.
Let $M_k(\vecz{a})$ for $k\in\NN$ and $\vecz{a}\in\AAA$ be the matrix
\[M_k(\vecz{a})=\left(\frac{\partial}{\partial \vecz{a}_{k,j}}t_{i}(a)\right)_{i\in [k],j\in [k+1]}.\]
We say that the system $S$ \emph{stabilizes} the functions $t_k$, $k\in\NN$,
if for every $k \in \NN$ and $x \in V_{k-1} \times U_k$, it holds that
\begin{enumerate}[(P1)]
\item the matrix $M_k(\wtw_{\le k}(x),b_{>k})$, restricted to the entries indexed by $I_k \times J_k$ is invertible, and
\item for each $i\in I_k$, and $x\in V_k$, we have
      \[t_i(\wtw_{\le k}(x),b_{>k})=t_i(\wtw_{\le k-1}(x_1,\ldots,x_{k-1}),b_k,b_{>k}).\]
\end{enumerate}
Note that the properties (P1) and (P2) always hold if $I_k=J_k=\emptyset$.

We are now ready to state the following lemma.

\begin{lemma}
\label{lm-smallchangestabilize}
Let $S=(U_k,V_k,I_k,J_k,d_k,b_k,w_k)_{k\in\NN}$ be a finite stabilizing system that
stabilizes totally analytic functions $t_k$, $k\in\NN$.
For every $\eps>0$, there exists $\delta>0$ such that the following holds.
Suppose that for each $k\in\NN$, we have $\whb_k \in (0,1)^{k+1}$ such that
$\whb_{k,d_k}$ is in the interior of $U_k$ and $\|b_{k}-\whb_{k}\|_\infty\le\delta$, and we have $\wht_k:[0,1]^\NN \to [0,1]$ totally analytic functions such that $\wht_k$ is $\delta$-close to $t_k$. Then
there exist functions $\whw_k:V_{k-1} \times U_k \to (0,1)^{J_k\cup\{d_k\}}$ such that
the stabilizing system $\whS=(U_k,V_k,I_k,J_k,d_k,\whb_k,\whw_k)_{k\in\NN}$ stabilizes $\wht_k$, $k\in\NN$, and
$\|w_k(x)-\whw_k(x)\|_\infty \le \eps$ for every $k\in\NN$ and $x\in V_{k-1} \times U_k$.
\end{lemma}

\begin{proof}
Let $k_0$ be such that $S$ is $k_0$-finite. 
We construct the stabilizers inductively for $k\in [k_0]$.
Suppose that we have constructed $\whw_1,\whw_2,\ldots,\whw_{k-1}$ such that
for each $i \in [k-1]$ and $x\in V_{i-1}\times U_i$,
$\|\whw_i(x)-\whw'_i(x)\|_\infty < \eps'$, where the value of $\eps'<\eps$ will be chosen later.
For $(x,y) \in U_k\times [0,1]^{J_k}$,
let $a(x,y) \in [0,1]^{k+1}$ be the vector with $a(x,y)_i=y_i$ for $i \in J_k$, $a(x,y)_{d_k}=x$, and $a(x,y)_i=b_{k,i}$ otherwise.
Define $\wha(x,y)$ analogously using $\whb_{k,i}$.
Since $V_{k-1}$ is compact, there exists a constant $L$ such that for every $\wx\in V_{k-1}$
the function $f_{\wx}:U_k\times [0,1]^{J_k}\to [0,1]^{I_k}$, defined as
\[f_{\wx}(x,y)=\left(t_{i}(\wtw_{\le k-1}(\wx),a(x,y),b_{>k})-t_{i}(\wtw_{\le k-1}(\wx),b_k,b_{>k})\right)_{i\in I_k}\]
has the property that
the function and each partial derivative is $L$-Lipschitz.
Similarly, define the function $\wf_{\wx}:U_k\times (0,1)^{J_k}\to [0,1]^{I_k}$, $\wx\in V_{k-1}$ to be
\[\wf_{\wx}(x,y)=\left(\wht_{i}(\whwtw_{\le k-1}(\wx),\wha(x,y),\whb_{>k})-\wht_{i}(\whwtw_{\le k-1}(\wx),\whb_k,\whb_{>k})\right)_{i\in I_k}.\]

We claim that for any $\delta'$,
we can take $\delta$ and $\eps'$ small enough so that for each $\wx \in V_{k-1}$, $\wf_{\wx}$ is $\delta'$-close to $f_{\wx}$.
Fix $i\in [k]$.
Since $t_i$ is continuous on the product topology on $[0,1]^\NN$,
it is uniformly continuous with respect to the metric $d(x,y)=\sum_{n\in\NN} 2^{-n} |x_n-y_n|$.
In particular, this means that there exists $\eps''>0$ so that
if $||x-y||_{\infty}<\eps''$, $x,y\in [0,1]^\NN$, then
\begin{equation}
|t_i(x)-t_i(y)| <\delta'/4.\label{eq-delta3}
\end{equation}
We next obtain the following estimate.
\begin{align*}
&\left|\wht_{i}(\whwtw_{\le k-1}(\wx),\wha(x,y),\whb_{>k})
-t_{i}(\wtw_{\le k-1}(\wx),a(x,y),b_{>k})\right| \\
\le &\left|\wht_{i}(\whwtw_{\le k-1}(\wx),\wha(x,y),\whb_{>k})
-t_{i}(\whwtw_{\le k-1}(\wx),\wha(x,y),\whb_{>k})\right|\\
+&\left|t_{i}(\whwtw_{\le k-1}(\wx),\wha(x,y),\whb_{>k})
-t_{i}(\wtw_{\le k-1}(\wx),a(x,y),b_{>k})\right|
\le \delta + \delta'/4. 
\end{align*}
Similarly, we have
\begin{align*}
&\left|\wht_{i}(\whwtw_{\le k-1}(\wx),\whb_k,\whb_{>k})
-t_{i}(\wtw_{\le k-1}(\wx),b_k,b_{>k})\right| \\
\le &\left|\wht_{i}(\whwtw_{\le k-1}(\wx),\whb_k,\whb_{>k})
-t_{i}(\whwtw_{\le k-1}(\wx),\whb_k,\whb_{>k})\right|\\
+&\left|t_{i}(\whwtw_{\le k-1}(\wx),\whb_k,\whb_{>k})
-t_{i}(\wtw_{\le k-1}(\wx),b_k,b_{>k})\right|
\le \delta + \delta'/4. 
\end{align*}
Hence, if $\eps'<\eps''$ and $\delta<\delta'/4$, we obtain that
\begin{align*}
\left|f_{\wx}(x,y)_i-\wf_{\wx}(x,y)_i\right|\le 
\left|\wht_{i}(\whwtw_{\le k-1}(\wx),\wha(x,y),\whb_{>k})
-t_{i}(\wtw_{\le k-1}(\wx),a(x,y),b_{>k})\right|\\
+ \left|\wht_{i}(\whwtw_{\le k-1}(\wx),\whb_k,\whb_{>k})-t_{i}(\wtw_{\le k-1}(\wx),b_k,b_{>k})\right| \le \delta'.
\end{align*}
An analogous argument applies to each partial derivative of $t_i$.

Let $T_k(\wx,z)$, $\wx\in V_{k-1}$ and $z\in [0,1]^{k+1}$, be
the determinant of the submatrix of $M_k(\wtw_{\le k-1}(\wx),z,b_{>k})$ formed by the entries indexed by $I_k \times J_k$.
Since the set $V_{k-1}\times \overline{N}_{\eps,J_k}( \wtw_k(V_{k-1}\times U_k))$ is closed for every $\eps>0$,
we can decrease $\eps>0$ (note that this strengthens the conclusion of the lemma) so that $T_k^2$
has a positive minimum $D^2$ on the set $V_{k-1}\times \overline{N}_{\eps,J_k}( \wtw_k(V_{k-1}\times U_k))$,
i.e., $|T_k|$ is at least $D$ on this set, for a positive real $D$.
We apply Lemma~\ref{lm-implicitsmallchangeclose} with $n=|J_k|$, $L$, $D$, $\eps$ as above, and $[a,b]=U_k$
to obtain $\delta'$ and $\eps_0$.
We then make sure that $\delta$ and $\eps'$
are small enough so that 
$\wf_{\wx}$ is $\delta'$-close to $f_{\wx}$ for every $\wx \in V_{k-1}$.

Using the triangle inequality, we get the following estimate:
\begin{align*}
\|\whb_{k}-\wtw_k(\wx,\whb_{k,d_k})\|_\infty \le &  \|\whb_{k}-b_{k}\|_\infty+\|b_{k}-\wtw_k(\wx,b_{k,d_k})\|_\infty\\
 & +\|\wtw_k(\wx,b_{k,d_k})-\wtw_k(\wx,\whb_{k,d_k})\|_\infty.
\end{align*}
The first summand on the right hand side is at most $\delta$, the middle one is zero and
the last summand can be made arbitrarily small by choosing $\delta$ small enough ($\delta$ that
is universal for all choices of $\wx\in V_{k-1}$ exists since $V_{k-1}$ is compact).
Hence, by choosing $\delta$ small enough, we can guarantee that $\|\whb_{k}-w_k(\whb_{k,d_k})\|_\infty \le \eps_0$.
Since $f_{\wx}(w_k(x))=0$ for every $x\in U_k$,
Lemma~\ref{lm-implicitsmallchangeclose} implies that
there exists $g_{\wx}:U_k\to [0,1]^{J_k\cup\{d_k\}}$ such that
$\wf_{\wx}(g_{\wx}(x))=0$ and $\|w_k(\wx,x)-g_{\wx}(x)\|\le \eps$ for all $x\in U_k$.
We set $\whw_k(\wx,x)=g_{\wx}(x)$ for $(\wx,x)\in V_{k-1} \times U_k$.
\end{proof}

We next consider stabilizing systems with a stronger property.
We say that a stabilizing system $(U_k,V_k,I_k,J_k,d_k,b_k,w_k)_{k\in\NN}$ that
stabilizes totally analytic functions $t_k$, $k\in\NN$, is \emph{$m$-excellent}
if the rank of the matrix $M_k(\wtw_{\le k}(x),b_{>k})$ is equal to $|J_k|$ for all $x\in V_{k-1}\times U_k$ and $k\in [m]$.
Note that the definition of stabilizing implies that the submatrix of $M_k(\wtw_{\le k}(x),b_{>k})$
formed by the entries indexed by $I_k \times J_k$ has rank $|J_k|$,
so we require that this submatrix has the same rank as the whole matrix $M_k(\wtw_{\le k}(x),b_{>k})$.

We next prove two auxiliary lemmas on excellent stabilizing systems.

\begin{lemma}
\label{lm-excellent-constant}
Let $S$ be an $m$-excellent stabilizing system that
stabilizes totally analytic functions $t_k$, $k\in\NN$.
It holds that
\[t_{\ell}(\wtw_{\le k}(x),b_{>k})=t_{\ell}(\wtw_{\le k-1}(x_1,\ldots,x_{k-1}),b_k,b_{>k})\]
for all $x\in V_{k-1}\times U_k$, $\ell\in [k]$ and $k\in [m]$.
\end{lemma}

\begin{proof}
Fix integers $k\in [m]$ and $\ell\in [k]$.
Since the rank of matrix $M_k(\wtw_{\le k}(x),b_{>k})$ is equal to $|J_k|$ for every $x\in V_{k-1}\times U_k$,
there exist functions $c_{i}:V_{k-1} \times U_k \to \RR$, $i \in I_k$, such that
\[\frac{\partial}{\partial \vecz{a}_{k,j}}t_\ell(\wtw_{\le k}(x),b_{>k})= \sum_{i \in I_k} c_{i}(x) \frac{\partial}{\partial \vecz{a}_{k,j}}t_{i}(\wtw_{\le k}(x),b_{>k})\]
for every $j\in [k+1]$.
Let $F_i(x)=t_i(\wtw_{\le k}(x),b_{>k})$, $i\in [k]$.
We now analyze the derivative of the function $F_{\ell}(x)$:
\begin{align*}
\frac{\partial}{\partial x_k} F_{\ell}(x)
  & =  \sum_{j=1}^{k+1} \frac{\partial}{\partial \vecz{a}_{k,j}}t_{\ell}(\wtw_{\le k}(x),b_{>k}) \frac{\partial}{\partial x_k}(\wtw_{\le k}(x))_j\\
  & =  \sum_{i \in I_k} \sum_{j=1}^{k+1} c_{i}(x) \frac{\partial}{\partial \vecz{a}_{k,j}}t_{i}(\wtw_{\le k}(x),b_{>k}) \frac{\partial}{\partial x_k}(\wtw_{\le k}(x))_j\\
  & =  \sum_{i \in I_k} c_{i}(x) \frac{\partial}{\partial x_k} F_{i}(x).
\end{align*}
Since the system $S$ stabilizes all functions $t_{i}$ with $i\in I_k$,
it follows that $\frac{\partial}{\partial x_k} F_{i}(x)=0$ for all $x\in V_{k-1}\times U_k$,
which implies that $\frac{\partial}{\partial x_k} F_{\ell}(x)=0$.
We conclude that
\[t_{\ell}(\wtw_{\le k}(x),b_{>k})=t_{\ell}(\wtw_{\le k-1}(x_1,\ldots,x_{k-1}),b_k,b_{>k})\]
for all $x\in V_{k-1}\times U_k$.
\end{proof}

\begin{lemma} \label{lm-makeexcellent}
Let $S=(U_k,V_k,I_k,J_k,d_k,b_k,w_k)_{k\in\NN}$ be a $c$-strong stabilizing system that
stabilizes totally analytic functions $t_k$, $k\in\NN$, and
that is $(m-1)$-excellent but not $m$-excellent.
For any $c'<c$ and $\eps>0$,
there exists a $c'$-strong $m$-finite stabilizing system $S'=(U'_k,V'_k,I'_k,J'_k,d'_k,b'_k,w'_k)_{k\in\NN}$ such that $b_k=b'_k$ for $k\not=m$,
$U'_k=U_k$, $I'_k=I_k$, $J'_k=J_k$ and $d'_k=d_k$ for $k\in [m-1]$,
$\|w_k(x)-w'_k(x)\|_\infty\le\eps$ for any $x\in V'_{k-1}\times U'_k$, $k\in [m-1]$,
$I_m \subsetneq I_m'$, $J_m \subsetneq J_m'$, $\|b_{m}-b'_{m}\|_\infty\le\eps$, and
$U'_k\subseteq [b'_{k,d'_k}-\eps,b'_{k,d'_k}+\eps]$ for $k\ge m$.
In addition, 
it also holds that
$||\wtw'_m(x)-b'_m||_{\infty}\le\eps$ for every $x\in V'_{m-1}\times U'_m$.
\end{lemma}

\begin{proof}
Recall that $M_m(\vecz{a})$ is the $m \times (m+1)$ Jacobian matrix of the partial derivatives of the functions $t_1,t_2,\ldots,t_m$
with respect to the variables $\vecz{a}_{m,i}$, $i\in [m+1]$.
Let $r$ be the maximum rank of $M_m(\vecz{a})$ taken over all $\vecz{a}=(\wtw_{\le m}(x),b_{>m})$, $x\in V_{m-1}\times U_m$.
Since the stabilizing system is not $m$-excellent, we have that $r>|J_m|$.
Let $x\in V_{m-1}\times U_m$ be a point where the rank $M_m(\vecz{a})$ is equal to $r$;
by the analyticity of the functions $t_1,\ldots,t_m$,
we can choose $x$ such that $x_m$ is arbitrarily close to $b_{m,d_m}$.
Set $b'_m=w_m(x_m)$ and let $w'_1,\ldots,w'_{m-1}$ be the stabilizers for the system $S$ with $b_m$ replaced with $b'_m$.
By Lemma \ref{lm-smallchangestabilize}, if $b'_m$ is sufficiently close to $b_m$ (which can be arranged
by choosing $x_m$ sufficiently close to $b_{m,d_m}$),
the stabilizers $w'_1,\ldots,w'_{m-1}$ exist and
it holds that $\|w_k(x)-w'_k(x)\|_\infty\le\eps$ for any $x\in V_{k-1}\times U_k$, $k\in [m-1]$.
Moreover, if $\eps$ is sufficiently small,
the rank of the matrix $M_m(\wtw'_{\le m-1}(x_1,\ldots,x_{m-1}),b'_m,b_{>m})$ stays at least $r$ and
the rank of its submatrix formed by the entries indexed by $I_m \times J_m$ stays $|J_m|$.

Let $I'_m\supseteq I_m$ and $J'_m\supseteq J_m$ be subsets of $[m]$ and $[m+1]$, respectively, such that
the rank of the submatrix of $M_m(\wtw'_{\le m-1}(x_1,\ldots,x_{m-1}),b'_m,b_{>m})$ formed by the entries indexed by $I'_m\times J'_m$ is $r$.
Choose $d'_m\in [m+1]\setminus J'_m$ arbitrarily and
let $T_m(x)$ be the determinant of the submatrix of $M_m(\wtw'_{\le m-1}(x),b'_m,b_{>m})$, $x\in V_{m-1}$.
Set $U_k'=U_k$ for $k\in [m-1]$, $U_m'=[b'_{m,d'_m}-\eps,b'_{m,d'_m}+\eps]\cap [0,1]$, and $U_k'=[b_{k,1}-\eps,b_{k,1}+\eps]\cap [0,1]$, otherwise.
By Lemma \ref{lm-Vkshrink}, there exist closed subsets $V_k' \subseteq V_k$, $k\in [m-1]$, such that
the variable system $(U_k',V_k')_{k\in\NN}$ with $V'_k=V'_{k-1}\times U'_k$ for $k\ge m$ is $c'$-strong and
$T_m(x)\not=0$ for all $x\in V'_{m-1}$.
By the Implicit Function Theorem and the compactness of $V_{k-1}'$,
there exists a non-trivial closed interval $U\subseteq U_m'$ containing $b'_{m,d'_m}$ such that
a stabilizer $w'_m$ exists on $V'_{m-1}\times U$ and
$||\wtw'_m(x)-b'_m||_{\infty}\le\eps$ for every $x\in V'_{m-1}\times U$.
We replace $U_m'$ with $U$ and set $I'_k=I_k$, $J'_k=J_k$ and $d'_k=d_m$ for $k\in [m-1]$.
We also set $I'_k=\emptyset$, $J'_k=\emptyset$, $d'_k=1$ and $w'_k(z):=z$ for $k>m$, and $b'_k=b_k$ for $k\not=m$.
We have obtained an $m$-finite stabilizing system $S'=(U'_k,V'_k,I'_k,J'_k,d'_k,b'_k,w'_k)_{k\in\NN}$
with $J'_k=J_k$ for $k\in [m-1]$ and $J_m \subsetneq J_m'$,
completing the proof.
\end{proof}

We next define a partial order on sequences of finite subsets of integers.
Given two such sequences $\cJ=(J_k)_{k \in \NN}$ and $\cJ'=(J_k')_{k\in \NN}$,
we say that $\cJ \prec \cJ'$ if there exists an integer $m$ such that
\begin{itemize}
	\item $J_k=J_k'$ for $k<m$, and
	\item $J_{m} \subset J_{m}'$, i.e., $J_{m}$ is a proper subset of $J'_{m}$.
\end{itemize}
In other words, we consider the partial order defined
by the lexicographic order together with the partial order given by subset containment.
We will also write $\cJ \preceq \cJ'$ if $\cJ \prec \cJ'$ or $\cJ=\cJ'$.

We are now ready to prove the main lemma of this section.

\begin{lemma} \label{lm-make-excellent}
Suppose that $S=(U_k,V_k,I_k,J_k,d_k,b_k,w_k)_{k\in\NN}$ is a $c$-strong stabilizing system that
stabilizes totally analytic functions $t_k$, $k\in\NN$.
Let $\eps>0$, $c'\in (0,c)$ and $m\in\NN$.
There exists a $c'$-strong $m$-excellent stabilizing system $S'=(U'_k,V'_k,I'_k,J'_k,d'_k,b'_k,w'_k)_{k\in\NN}$
that stabilizes the functions $t_k$, $k\in\NN$, and
that satisfies $(J_k)_{k \in \NN}\preceq (J_k')_{k\in \NN}$ and
$\|b_{k}-b'_{k}\|_\infty<\eps$ for all $k\in\NN$.

Moreover, the following holds.
Let $k_0$ be the largest integer such that $J_k'=J_k$ for all $k\in [k_0]$.
For every $k\in [k_0]$, it holds that
$I_{k}'=I_{k}$, $d_{k}'=d_{k}$, $U'_{k}=U_{k}$, $V'_{k}\subseteq V_{k}$, and
$\|w_{k}(x)-w'_{k}(x)\|_\infty<\eps$ for any $x\in V'_{k-1}\times U'_{k}$.
It also holds that
$||\wtw'_k(x)-b'_k||_{\infty}<\eps$ for every $x\in V'_{k-1}\times U'_k$ and $k>k_0$.
\end{lemma}

\begin{proof}
Set $L=(m+1)!$, $\eps'=\eps/(L+1)$ and $\delta=(c-c')/(L+1)$.
We keep successively applying Lemma~\ref{lm-makeexcellent}.
At each step, we choose the smallest $m'\le m$ such that the current stabilizing system is not $m'$-excellent and
apply Lemma~\ref{lm-makeexcellent} with $\eps'$ and $c-i\delta$ in the $i$-th step.
Note that in each step the sequence $(J_k)_{k \in \NN}$ increases in the order defined by $\prec$ and
this increase is witnessed by $J_k$ with \mbox{$k\le m'\le m$}.
It is not difficult to see that there can be at most $L$ increases in total.
Hence, the whole procedure terminates after at most $L$ steps
with an $m$-excellent stabilizing system $S'=(U'_k,V'_k,I'_k,J'_k,d'_k,b_k,w_k)_{k\in\NN}$.
By the choice of $\delta$, the resulting system is $c'$-strong, and
the choice of $\eps'$ implies that $\|b_{k}-b'_{k}\|_\infty<\eps$ for all $k\in\NN$.
Moreover, if the first $k_0$ elements of the sequence $(J_k)_{k \in \NN}$ stayed the same,
i.e., $J_k'=J_k$ for all $k\in [k_0]$,
then we only applied Lemma \ref{lm-makeexcellent} with $m'>k_0$,
therefore, $U'_{k_0}=U_{k_0}$, $V'_{k_0}\subseteq V_{k_0}$, and
$\|w_{k_0}(x)-w'_{k_0}(x)\|_\infty\le L\eps'<\eps$ for any $x\in V'_{k_0-1}\times U'_{k_0}$.
\end{proof}

\section{Main result}
\label{sec-main}

We are now ready to prove the main theorem.

\begin{theorem}
\label{thm-main}
There exists a family of graphons $\WW$, graphs $H_1,\ldots,H_{\ell}$ and reals $d_1,\ldots,d_{\ell}$ such that
\begin{itemize}
\item a graphon $W$ is weakly isomorphic to a graphon contained in $\WW$ if and only if $d(H_i,W)=d_i$ for every $i\in [\ell]$, and
\item no graphon in $\WW$ is finitely forcible, i.e., for all graphs $H'_1,\ldots,H'_r$ and reals $d'_1,\ldots,d'_{r}$,
      the family $\WW$ contains either zero or infinitely many graphons $W$ with $d(H'_i,W)=d'_i$, $i\in [r]$.
\end{itemize}
\end{theorem}

\begin{proof}
The family $\WW$ will be a family of graphons $\WW_P$,
which we introduced in Section~\ref{sec-gensetup} and
we formally define in Section~\ref{sec-family},
for a suitable choice of bounding sequence $P$.
In this proof,
we will only use Theorem~\ref{thm-WPforced}, Lemma~\ref{lm-densityanalytic} and Lemma~\ref{lm-WPvaries},
which can be found in Section~\ref{sec-gensetup}, and the results in Section~\ref{sec-diag}.

Fix an enumeration $G_k$, $k\in\NN$, of all graphs.
We will iteratively apply the results of Section~\ref{sec-diag} to construct bounding sequences $P^m$, $m\in\NN$,
which converge to the sought bounding sequence $P$, and
corresponding finite stabilizing systems $S^m=(U^m_k,V^m_k,I^m_k,J^m_k,d^m_k,b^m_k,w^m_k)_{k\in\NN}$, 
which stabilize the densities of $G_k$.
Initially, we set $P^0$ to be the sequence where each element is $(0,0,1)$ and
$S^0$ to be the $0$-finite stabilizing system with $b^0_k=(1/2,\ldots,1/2)$,
i.e., $U^0_k=[0,1]$, $V^0_k=[0,1]^k$, $I^0_k=J^0_k=\emptyset$, $d^0_k=1$ and $w^0_k(x)=x$ for all $k\in\NN$.

Let $t^m_k$ be the function $t_{P^m,G_k}$ introduced at the end of Section~\ref{sec-gensetup},
let $Q^m_{k,\eps}=\overline{N}_\eps(\wtw_{\le k}(V^m_{k-1} \times U^m_k)) \subseteq [0,1]^{2+3+\ldots+(k+1)}$ for $k\in\NN$, and
let $\delta_m$ be the value of $\delta$ from Lemma~\ref{lm-smallchangestabilize} applied with $S^{m-1}$, $t^{m-1}_k$, and $\eps=1/2^{m+1}$.
The bounding sequence $P^m$ and the stabilizing system $S^m$ will satisfy the following properties for every $m\in\NN$.
\begin{itemize}
\item The bounding sequence $P^m$ is a strengthening of $P^{m-1}$ and contains an infinite number of elements equal to $(0,0,1)$.
\item Each function $t^m_{k}$, $k\in\NN$, is $\delta_m$-close to $t^{m-1}_{k}$.
\item The stabilizing system $S^m$ stabilizes the functions $t^{m-1}_{k}$, $k\in\NN$, and $S^m$ is $m$-finite, $(1/2+1/2^{m+1})$-strong and $m$-excellent.
\item $\|b^m_k-b^{m-1}_{k-1}\|_\infty<1/2^{m+1}$ for every $k\in\NN$.
\item $(J^{m-1}_k)_{k \in \NN}\preceq (J^m_k)_{k\in \NN}$, and
      if $k_0$ is the largest integer such that $J^{m}_k=J^{m-1}_k$ for all $k\in [k_0]$,
      then $I^m_{k}=I^{m-1}_{k}$, $d^m_{k}=d^{m-1}_{k}$, $U^m_{k}=U^{m-1}_{k}$, $V^m_{k}\subseteq V^{m-1}_{k}$, and
      $\|w^m_{k}(x)-w^{m-1}_{k}(x)\|_\infty<1/2^{m}$ for any $x\in V^m_{k-1}\times U^m_{k}$ with $k\in [k_0]$.
\item $Q^m_{m,1/2^m}\subseteq Q^{m-1}_{m-1,1/2^{m-1}}\times [0,1]^{m+1}$.
\item A vector $\vecz{z} \in \AAA$ belongs to $Q^m_{m,1/2^m} \times \prod_{j=m+1}^{\infty} [0,1]^{j+1}$
      if and only if $p_i(z)\in [l_i,u_i]$ for every element $(p_i,l_i,u_i)$ of the sequence $P^m$.
\end{itemize}

Fix an integer $m\in\NN$ and suppose that
we have defined bounding sequences $P^0,\ldots,P^{m-1}$ and stabilizing systems $S^0,\ldots,S^{m-1}$.
By Lemma \ref{lm-smallchangestabilize},
there exists a stabilizing system $\whS$ for $t^{m-1}_{k}$ and the same $b_k$, $k\in\NN$, such that $\|\whw_k(x)-w_k^{m-1}(x)\|_\infty<1/2^{m+1}$ for any $x \in V_{k-1}^{m-1} \times U_k^{m-1}$ with $k \in \NN$;
if $m=1$, we just set $\whS$ to be $S^0$.
We next apply Lemma~\ref{lm-make-excellent} with $c=1/2+1/2^m$, $c'=1/2+1/2^{m+1}$ and $\eps=1/2^{m+1}$
to obtain an $m$-excellent stabilizing system $S^m$ for $t^{m-1}_{k}$, $k\in\NN$,
which is $m$-finite and $(1/2+1/2^{m+1})$-strong.
Also note that $\|b^m_k-b^{m-1}_{k}\|_\infty<1/2^{m+1}$ for every $k\in\NN$, and
$(J^{m-1}_k)_{k \in \NN}\preceq (J^m_k)_{k\in \NN}$.
Furthermore, if $k_0$ is the largest integer such that $J^{m}_k=J^{m-1}_k$ for all $k\in [k_0]$,
then $I^m_{k}=I^{m-1}_{k}$, $d^m_{k}=d^{m-1}_{k}$, $U^m_{k}=U^{m-1}_{k}$, $V^m_{k}\subseteq V^{m-1}_{k}$, and
\begin{align*}
\|w^m_{k}(x)-w^{m-1}_{k}(x)\|_\infty & \le  \|w^m_{k}(x)-\whw_{k}(x)\|_\infty+\|\whw_{k}(x)-w^{m-1}_{k}(x)\|_\infty\\
                                         & <  1/2^{m+1}+1/2^{m+1} = 1/2^m
\end{align*}
for every $x\in V^m_{k}$ and $k\in [k_0]$.
This implies $Q^m_{k_0,1/2^m}\subseteq Q^{m-1}_{k_0,1/2^{m-1}}$.
For $k>k_0$,
we have that $||\wtw^m_k(x)-b^m_k||_{\infty}<1/2^{m+1}$ for all $x\in V^m_{k-1}\times U^m_k$ and
$||b^m_k-b^{m-1}_k||_{\infty}<1/2^{m+1}$.
It follows that $Q^m_{m,1/2^m}\subseteq Q^{m-1}_{m-1,1/2^{m-1}}\times [0,1]^{m+1}$.
This finishes the definition of the stabilizing system $S^m$ and verifies its properties as stated above.

We next define the bounding sequence $P^m$.
Set $K$ to be the maximum integer obtained
when Lemma~\ref{lm-WPvaries} is applied with $P^{m-1}$, $\eps=\delta_{m}$ and each of the graphs $G_1,\ldots,G_m$.
Choose an infinite increasing sequence $k_i$, $i\in\NN$, with $k_1\ge K$, such that
$P^{m-1}_{k_i}=(0,0,1)$ for every $i\in\NN$ and
there exists an infinite number of indices $j\in\NN$ such that $P^{m-1}_j=(0,0,1)$ and $j\not=k_i$ for all $i\in\NN$.
The Stone-Weierstrass Theorem implies that, for every $i\in\NN$,
there exists a polynomial $q_i(z_1,\ldots,z_m)$, $z_j\in [0,1]^{j+1}$, $j\in[m]$, such that
\begin{itemize}
\item $q_i(z_1,\ldots,z_m)\in [0,1]$ for every $(z_1,\ldots,z_m)\in [0,1]^{2+\ldots+(m+1)}$,
\item $q_i(z_1,\ldots,z_m)\in [1-2^{-i},1]$ for every $(z_1,\ldots,z_m)\in Q^m_{m,1/2^{m}}$, and
\item $q_i(z_1,\ldots,z_m)\in [0,2^{-i}]$ for $(z_1,\ldots,z_m)$ at Hausdorff distance at least $2^{-i}$ from $Q^m_{m,1/2^{m}}$.
\end{itemize}

Let $\xi_i\in (0,1]$ be a real such that all coefficients of $\xi_i\cdot q_i$ have absolute value at most $2^{-2^{k_i}}/9$ and
all the partial derivatives of $\xi_i\cdot q_i$ belong to $[-1,+1]$ on $[0,1]^{2+\ldots+(m+1)}$.
The bounding sequence $P^m$ is defined as follows:
$P^m_{k_i}$ is \mbox{$(\xi_i\cdot q_i,\xi_i(1-2^{-i}),\xi_i)$} for $i\in\NN$, and
$P^m_j$ is $P^{m-1}_j$ for the remaining indices $j$, i.e., if $j\not=k_i$ for all $i\in\NN$.
Observe that $(z_1,\ldots,z_m)\in [0,1]^{2+\ldots+(m+1)}$ belongs to $Q^m_{m,1/2^m}$ if and only if
$p_i(z_1,\ldots,z_m)\in [l_i,u_i]$ for every element $(p_i,l_i,u_i)$ of the sequence $P^m$.
Indeed, 
the construction implies that $Q^m_{m,1/2^m}\subseteq Q^{m-1}_{m-1,1/2^{m-1}}\times [0,1]^{m+1}$;
in particular, any point in $Q^m_{m,1/2^m}$ satisfies all constraints implied by $P^{m-1}$.
Moreover, the functions $t^m_{k}$ and $t^{m-1}_{k}$ for $k\le m$ are $\delta_{m}$-close.

We have now defined the bounding sequence $P^m$ and the stabilizing system $S^m$ for every $m\in\NN$.
Since each element of the bounding sequences $P^m$ changes at most once during the iterative procedure described above,
we can define a bounding sequence $P$ as
$$P_i=\lim_{m\to\infty}P^m_{i}\,\mbox{.}$$
The bounding sequence $P$ determines the family $\WW_P$ of graphons;
the family $\WW$ from the statement of the theorem will be constructed as a subset of $\WW_P$.

Recall that $(J^{m-1}_k)_{k \in \NN}\preceq (J^m_k)_{k\in \NN}$ for every $m\in\NN$. This implies that for every $k\in\NN$,
there exists $m_0$ such that $J^m_k$ is the same for all $m\ge m_0$, and
consequently $I^m_k$, $d^m_k$ and $U^m_k$ are the same and $V^{m+1}_k\subseteq V^m_k$ for all $m\ge m_0$.
Hence, we can define
\[J_k=\lim_{m\to\infty} J^m_{k}\mbox{, }\;d_k=\lim_{m\to\infty} d^m_{k}\mbox{ and } U_{k}=\lim_{m\to\infty} U^m_{k}\,\mbox{.}\]
Likewise, we can define
\[V_k=\bigcap_{m\in\NN\setminus [k]} V^m_{k}\mbox{ and }V=\bigcap_{k\in\NN}\left(V_k\times [0,1]^{\NN\setminus [k]}\right).\]
Since each $V^m_{k}$ satisfies for every $(x_1,\ldots,x_{k-1})\in V^m_{k-1}$ that
the measure of $x_k\in U^m_k$ such that $(x_1,\ldots,x_k)\in V^m_k$ is at least $\left(1/2+1/2^{m+1}\right)|U^m_k|$,
it follows that $V_k$ satisfies for every $(x_1,\ldots,x_{k-1})\in V_{k-1}$ that
the measure of $x_k\in U_k$ such that $(x_1,\ldots,x_k)\in V_k$ is at least $|U_k|/2$.

Since the sequence $(b^m_k)_{m\in\NN}$ is Cauchy for every $k\in\NN$,
we can also define $b_k$ to be the limit of the sequence $(b^m_k)_{m\in\NN}$.
Finally, the sequence of functions $(w^m_{k})_{m\in\NN}$ is uniformly convergent on $V_{k-1} \times U_k$, and
we set $w_k$ to be the limit.
In this way, we have defined a stabilizing system $S=(U_k,V_k,I_k,J_k,d_k,b_k,w_k)_{k\in\NN}$.
Observe that $S$ is $1/2$-strong.
Also observe that
\begin{equation*}
w_k(x,b_{k,d_k})_j  =  \lim_{m\to\infty} w_k(x,b^m_{k,d_k})_j 
                    =  \lim_{m\to\infty} w^m_k(x,b^m_{k,d_k})_j 
                    =  \lim_{m\to\infty} b^m_{k,j} = b_{k,j}\mbox{.}
\end{equation*}
for every $x \in V_{k-1}$ and $j \in J_k$,
where for the second equality we used the fact that $w_k^m$ converges uniformly to $w_k$.
It follows that $\wtw_k(x,b_{k,d_k})=b_k$ for every $x\in V_{k-1}$.
Finally,
observe that $z\in [0,1]^\NN$ satisfies $p_i(z)\in [l_i,u_i]$ for every element $(p_i,l_i,u_i)$ of the sequence $P$
if and only if there exist $x\in V$ such that $\vec z=(\wtw_1(x_1),\wtw_2(x_1,x_2),\ldots)$.

Let $t_k$ be the function $t_{P,G_k}$ and
observe that the sequence $(t^m_k)_{m\in\NN}$ uniformly converges to $t_k$
on the set containing all $z\in [0,1]^\NN$ such that $p(z)\in [l,u]$ for every element $(p,l,u)$ of the sequence $P$.
Consequently, we obtain that
\begin{equation*}
t_{i}(\wtw_{\le k}(x),b_{>k})  =  \lim_{m\to\infty} t_{i}(\wtw^m_{\le k}(x),b_{>k}) 
                               =  \lim_{m\to\infty} t^m_{i}(\wtw^m_{\le k}(x),b_{>k})
\end{equation*}
for every $i\in [k]$, $x\in V_k$ and $k\in\NN$, where we used that $t_i^m$ converges uniformly to $t_i$.
Using Lemma \ref{lm-excellent-constant}, this implies that
\begin{align*}
t_{i}(\wtw_{\le k}(x),b_{>k}) =&  \lim_{m\to\infty} t^m_{i}(\wtw^m_{\le k}(x),b_{>k}) =  \lim_{m\to\infty} t^m_{i}(\wtw^m_{\le k-1}(x_1,\ldots,x_{k-1}),b_k,b_{>k}) \\=&  t_i(\wtw_{\le k-1}(x_1,\ldots,x_{k-1}),b_k,b_{>k})
\end{align*}
for every $i\in [k]$, $x\in V_k$ and $k\in\NN$.
It follows that
\[t_{i}(\wtw_1(x_1),\wtw_2(x_1,x_2),\ldots)=t_i(\wtw_{\le i-1}(x_1,\ldots,x_{i-1}),b_{\ge i})\]
for every $x\in V$.
Thus, if two elements $x$ and $x'$ from $V$ agree on the first $i-1$ coordinates,
then $t_{i}(\wtw_1(x_1),\wtw_2(x_1,x_2),\ldots)$ and
$t_{i}(\wtw_1(x'_1),\wtw_2(x'_1,x'_2),\ldots)$ are the same,
i.e.,
\[\tau(G_i,W_P(\wtw_1(x_1),\wtw_2(x_1,x_2),\ldots))=\tau(G_i,W_P(\wtw_1(x'_1),\wtw_2(x'_1,x'_2),\ldots)).\]
Consequently, if $x$ and $x'$ from $V$ agree on the first $r-1$ coordinates,
then the densities of all graphs $G_1,\ldots,G_r$
in the graphons $W_P(\wtw_1(x_1),\wtw_2(x_1,x_2),\ldots)$ and $W_P(\wtw_1(x'_1),\wtw_2(x'_1,x'_2),\ldots)$
are the same.

We are now ready to define the family $\WW$ of graphons with the properties from the statement of the theorem.
Theorem~\ref{thm-WPforced} implies that there exists graphs $H_1,\ldots,H_\ell$ and
a polynomial $q$ in $\ell$ variables such that a graphon $W$ is weakly isomorphic to a graphon contained in $\WW_P$
if and only if $q(d(H_1,W),\ldots,d(H_{\ell},W))=0$.
Set $d_i=d(H_i,W_P(b_1,b_2,\ldots))$ for $i\in [\ell]$ and
define the family $\WW$ as follows
\[\WW=\{W\mbox{ such that }d(H_i,W)=d_i\mbox{ for all }i\in[\ell]\}.\]
Observe that $\WW\subseteq\WW_P$.

Suppose that a graphon $W\in\WW$ is finitely forcible.
We can assume that there exist $r\in\NN$ and $d'_1,\ldots,d'_r$ such that
$W$ is the unique graphon in $\WW$ with the density of $G_i$ equal to $d'_i$ for $i\in [r]$, and
the graphs $G_1,\ldots,G_r$ include all graphs $H_1,\ldots,H_{\ell}$.
Since $W$ is weakly isomorphic to a graphon in $\WW_P$,
there exists $x\in V$ such that $W$ and $W_P(\wtw_1(x_1),\wtw_2(x_1,x_2),\ldots)$ are weakly isomorphic.
Since the measure of $y_r$ such that $(x_1,\ldots,x_{r-1},y_r)\in V_r$ is at least $|U_r|/2>0$,
there exists $x'_r\in V_r$ such that $x'_r\not=x_r$ and $(x_1,\ldots,x_{r-1},x'_r)\in V_r$.
We set $x'_i=x_i$ for $i\in [r-1]$ and
iteratively find $x'_i$ for $i>r$ as follows:
if $x'_1,\ldots,x'_{i-1}$ has been constructed for some $i>r$,
then we choose $x'_i$ to be an arbitrary element of $U_i$ such that
$(x'_1,\ldots,x'_{i-1},x'_i)\in V_i$;
note that such $x'_i$ exists since the measure of suitable choices of $x'_i$ is at least $|U_i|/2>0$.
Let $W'$ be the graphon $W_P(\wtw_1(x'_1),\wtw_2(x'_1,x'_2),\ldots)$.
Since $x'\in V$, the graphon $W'$ belongs to the family $\WW_P$.
In addition, the density of each of the graphs $G_1,\ldots,G_r$ in $W$ and in $W'$ is the same.
On the other hand, $W'$ is not weakly isomorphic to $W$ by Theorem~\ref{thm-WPforced}
because $\wtw_r(x_1,\ldots,x_r)\not=\wtw_r(x'_1,\ldots,x'_r)$.
We conclude that the graphon $W$ is not finitely forcible.
\end{proof}

\section{Finitely forcible graphon families}
\label{sec-family}

In this section, we define the family of graphons $W_P(z)$, and prove Theorem \ref{thm-WPforced}.
The general structure of graphons $W_P(z)$ is visualized in Figure~\ref{fig-main}.
The graphon $W_P(z)$ is a partitioned graphon with 14 parts,
denoted $A$, $B$, $C$, $D_A$, $D_B$, \dots, $D_G$, $E$, $F$, $Q$ and $R$.
The size of each part besides $Q$ is $1/25$, and the size of $Q$ is $12/25$.
The degree of each part is given in Table~\ref{tab-parts} (we do not compute the exact value of the degree of the part~$Q$).
Each part $A$, $B$, $C$, $D_A$, $D_B$, \dots, $D_G$, $E$, $F$, $Q$ and $R$
is a half-open subinterval of $[0,1)$.

\begin{figure} 
\begin{center}
\epsfxsize=0.95\textwidth \epsfbox{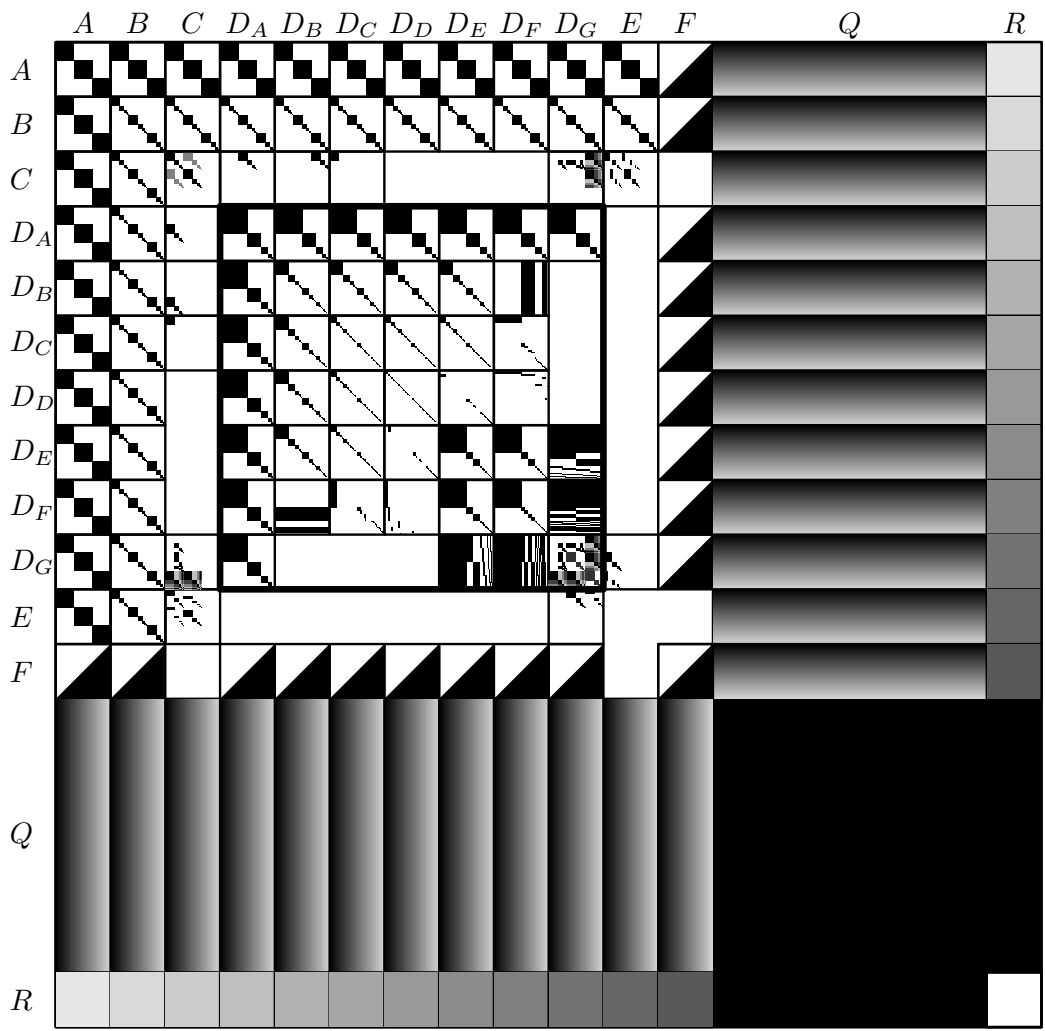}
\end{center}
\caption{The structure of the graphon $W_P(z)$.
         The parts of the graphon corresponding to the graphon $W_0$ from Theorem~\ref{thm-universal} are $D_A, \ldots, D_G$ and
	 are framed by a thicker line. Note that the size of the part $Q$ is not drawn to scale.}
\label{fig-main}
\end{figure}

\begin{table}
\begin{center}
\renewcommand*{\arraystretch}{1.2}
\begin{tabular}{|l|cccccccccc|}
\hline
Part & $A$ & $B$ & $C$ & $D_A$ & $\cdots$ & $D_G$ & $E$ & $F$ & $Q$ & $R$ \\
\hline
Degree & 
$\frac{1201}{2500}$ &  %A
$\frac{1202}{2500}$ &  %B
$\frac{1203}{2500}$ &  %C
$\frac{1204}{2500}$ &  %D_A
$\cdots$ &
$\frac{1210}{2500}$ &  %D_G
$\frac{1211}{2500}$ &  %E
$\frac{1212}{2500}$ &  %F
$>\frac{1300}{2500}$ & %Q
$\frac{1278}{2500}$\\ %R
\hline
\end{tabular}
\end{center}
\caption{The degrees of the parts of the graphon $W_P(z)$.}
\label{tab-parts}
\end{table}

%The next definition introduces a particular way of partitioning the interval $[0,1)$ into disjoint intervals.
Let $J_k$, $k\in\NN$, be the interval $[1-2^{-k+1},1-2^{-k})$,
so $J_1=[0,1/2)$, $J_2=[1/2,3/4)$, $J_3=[3/4,7/8)$, etc.
Note that the intervals $J_k$, $k\in\NN$, give a partition of $[0,1)$.
If $x\in\RR$, define $\coord{x}$ as the unique integer $k$ such that $x-\lfloor x\rfloor\in J_k$.

\begin{figure} 
\begin{center}
\epsfbox{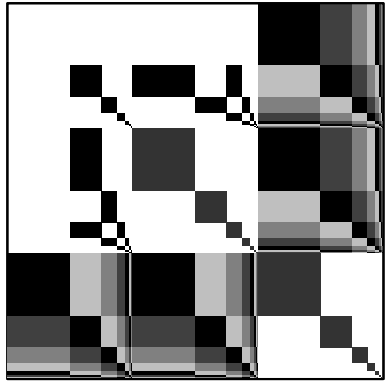}
\end{center}
\caption{Visualization of the graphon $W_F$, which is in the tile $D_G\times D_G$ of the graphon $W_P(z)$.}
\label{fig-W_F}
\end{figure}

We start by defining an auxiliary graphon $W_F$, which is visualized in Figure~\ref{fig-W_F}.
Fix a bounding sequence $P=(p_i,l_i,u_i)_{i \in \NN}$, and
recall that $\pi_{i,j}$ is the coefficient of the monomial $z^{M_j}$ in the polynomial $p_i$.
Let $\pi_{i,j}^+$ be equal to $\pi_{i,j}$ if $\pi_{i,j}\ge 0$, and $0$ otherwise;
similarly, let $\pi_{i,j}^-$ be equal to $-\pi_{i,j}$ if $\pi_{i,j}<0$, and $0$ otherwise.
Note that the definition of a bounding sequence
implies that $|\pi_{i,j}|\le 2^{-2^j}/9\le 2^{-2j}/9$ for all $i,j\in\NN$.

The value of $W_F(x,y)$ is equal to $1$ for $(x,y)\in [0,2/3)^2\setminus [1/3,2/3)^2$ if
\begin{itemize}
\item $x,y\in [1/6,1/3)$ and $\coord{3x}=\coord{3y}$,
\item $x\in [1/3,2/3)$, $y\in [1/6,1/3)$, and either $M_{\coord{3x}}=M_{\coord{3y}}\setminus\{\min M_{\coord{3y}}\}$ or $M_{\coord{3x}}=\{\min M_{\coord{3y}}\}$, or
\item $y\in [1/3,2/3)$, $x\in [1/6,1/3)$, and either $M_{\coord{3y}}=M_{\coord{3x}}\setminus\{\min M_{\coord{3x}}\}$ or $M_{\coord{3y}}=\{\min M_{\coord{3x}}\}$.
\end{itemize}
Recall that both $M_{\coord{3x}}$ and $M_{\coord{3y}}$ are multisets and
$M_{\coord{3y}}\setminus\{\min M_{\coord{3y}}\}$
is the multiset obtained from $M_{\coord{3y}}$ by removing a single instance of $\min M_{\coord{3y}}$.
The value of $W_F(x,y)$ is equal to $0$ everywhere else on $(x,y)\in [0,2/3)^2\setminus [1/3,2/3)^2$.
The value of $W_F(x,y)$ for the remaining $(x,y)\in [0,1)^2$ is defined as follows:
$$W_F(x,y)=\left\{\begin{array}{cl}
           9\cdot 2^{2\coord{3y}}\cdot\pi^+_{\coord{3x},\coord{3y}} & \mbox{if $x\in [2/3,1)$ and $y\in [0,1/3)$,} \\
           9\cdot 2^{2\coord{3y}}\cdot\pi^-_{\coord{3x},\coord{3y}} & \mbox{if $x\in [2/3,1)$ and $y\in [1/3,2/3)$,} \\
           9\cdot 2^{2\coord{3x}}\cdot\pi^+_{\coord{3y},\coord{3x}} & \mbox{if $x\in [0,1/3)$ and $y\in [2/3,1)$,} \\
           9\cdot 2^{2\coord{3x}}\cdot\pi^-_{\coord{3y},\coord{3x}} & \mbox{if $x\in [1/3,2/3)$ and $y\in [2/3,1)$,} \\
           1-u_{\coord{3x}} & \mbox{if $x,y\in [1/3,2/3)$ and $\coord{3x}=\coord{3y}$,}\\
           0 & \mbox{if $x,y\in [1/3,2/3)$ and $\coord{3x}\not=\coord{3y}$,}\\
           l_{\coord{3x}} & \mbox{if $x,y\in [2/3,1)$ and $\coord{3x}=\coord{3y}$, and}\\
           0 & \mbox{if $x,y\in [2/3,1)$ and $\coord{3x}\not=\coord{3y}$.}\\
           \end{array}\right.$$
By Theorem~\ref{thm-universal},
there exists a finitely forcible graphon $W_0$ such that $W_{F}$ is a subgraphon of $W_0$.
The graphon $W_0$, which is depicted in Figure~\ref{fig-universal}, is a partitioned graphon with 10 parts $A,\ldots,G,P,Q,R$.
Let $\iota_X^0$ be the linear map from $[0,1)$ to the part $X$ of the graphon $W_0$.
Since the graphon $W_F$ is embedded in the tile $G\times G$ of $W_0$,
it holds that $W_0(\iota^0_G(x),\iota_G^0(y))=W_F(x,y)$ for all $(x,y)\in [0,1)^2$.
Also note that the definition of the graphon $W_0$ depends on $P$ only, i.e., it does not depend on $z\in [0,1]^\NN$.

We are now ready to define the graphon $W_P(z)$.
Fix $z\in [0,1]^\NN$ (in addition to the sequence $P$, which we have already fixed).
For the rest of this section, we will just write $W_P$ for $W_P(z)$, which allows us to write $W_P(x,y)$ for the value of the graphon $W_P(z)$ for $(x,y)\in [0,1)^2$.
For each part $X$ of $W_P$, let $\iota_X$ be the linear map from $[0,1)$ to $X$.
For $Z,Z'\in\{A,\ldots,G\}$, set $W_P(\iota_{D_Z}(x),\iota_{D_{Z'}}(y))=W_0(\iota^0_Z(x),\iota^0_{Z'}(y))$ for all $(x,y)\in [0,1)^2$.
The values of $W_P$ on the remaining pairs of the parts $A$, $B$, $C$, $D_A$, $D_B$, \dots, $D_G$, $E$, and $F$
are given in Table~\ref{tab-def-WP}.

The description of the tiles $C\times E$ and $D_G\times E$
may seem harder to immediately grasp.
Let us examine e.g. the condition $0 \le 2^{\coord{3x}}(3y-1)+2\le p_{\coord{3x}}(z)$ in the first case
in the description of the tile $C\times E$.
Observe that the condition is equivalent to
\[y\in\left[\frac{1-2^{-(\coord{3x}-1)}}{3},\frac{1-2^{-(\coord{3x}-1)}+2^{-\coord{3x}}\cdot p_{\coord{3x}}(z)}{3}\right].\]
In particular, the condition can only be satisfied if $\coord{3x}=\coord{3y}$.
The other conditions in the description of the tile $C\times E$ and
two of the conditions in the description of the tile $D_G\times E$
can be expressed analogously.

\begin{table}
\begin{center}
\scalebox{0.80}{
\renewcommand*{\arraystretch}{1.2}
\begin{tabular}{|c|c|c|l|}
\hline
Tile $X\times Y$ & $W_P(\iota_X(x),\iota_Y(y))$ & Domain of $(x,y)$ & Condition \\
\hline
$A\times A$ & & & \\
$\vdots$ & $1$ & $[0,1)^2$ & $\lfloor 3x\rfloor=\lfloor 3y\rfloor$\\
$A\times E$ & & & \\
\hline
%\multirow{2}{*}{$B\times B$} & & & \\
%         & & $[0,1/3)^2\cup$ & \\
%$\vdots$ & $1$ & $[1/3,2/3)^2\cup$ & $\coord{3x}=\coord{3y}$ \\
%\multirow{2}{*}{$B\times E$} & & $[2/3,1)^2$ & \\
% & & & \\
$B\times B$ & & $[0,1/3)^2\cup$ & \\
$\vdots$ & $1$ & $[1/3,2/3)^2\cup$ & $\coord{3x}=\coord{3y}$ \\
$B\times E$ & & $[2/3,1)^2$ & \\
\hline
\multirow{4}{*}{$C\times C$} & $z^{M_{\coord{3x}}}$ & $[0,1/3)^2$ & $\coord{3x}=\coord{3y}$\\
            & $z^{M_{\coord{3x}}}$ & $[1/3,2/3)^2$ & $\coord{3x}=\coord{3y}$\\
            & $1-z^{M_{\coord{3x}}}$ & $[0,1/3)\times [1/3,2/3)$ & $\coord{3x}=\coord{3y}$\\
            & $1-z^{M_{\coord{3x}}}$ & $[1/3,2/3)\times [0,1/3)$ & $\coord{3x}=\coord{3y}$\\
\hline
$C\times D_A$ & $1$ & $[0,1/3)\times [1/3,2/3)$ & $\coord{3x}=\coord{3y}$ \\
\hline
$C\times D_B$ & $1$ & $[0,1/3)\times [2/3,1)$ & $\coord{3x}=\coord{3y}$ \\
\hline
$C\times D_C$ & $1$ & $[0,1/6)^2$ & \\
\hline
\multirow{2}{*}{$C\times D_G$} & \multirow{2}{*}{$W_F(x,y)$} & $[0,1/3)\times [0,2/3)\cup$ & \\
              & & $[0,2/3)\times [2/3,1)$ & \\
\hline
\multirow{4}{*}{$C\times E$}
            & $1$ & $[0,1/3)\times [0,1/3)$ & $0 \le 2^{\coord{3x}}(3y-1)+2\le p_{\coord{3x}}(z)$ \\
            & $1$ & $[0,1/3)\times [1/3,2/3)$ & $0 \le
            2^{\coord{3x}}(3y-2)+2\le 1-p_{\coord{3x}}(z)$ \\
            & $1$ & $[1/3,2/3)\times [0,1/3)$ & $1 \ge 2^{\coord{3x}}(3y-1)+2\ge p_{\coord{3x}}(z)$ \\
            & $1$ & $[1/3,2/3)\times [1/3,2/3)$ & $1 \ge 2^{\coord{3x}}(3y-2)+2\ge 1-p_{\coord{3x}}(z)$ \\
\hline
\multirow{4}{*}{$D_G\times E$} & \multirow{2}{*}{$1$} & \multirow{2}{*}{$[1/3,2/3)\times [0,1/3)$} & $\coord{3x}=\coord{3y}$ and \\
                               & & & $2^{\coord{3y}}(3x-1)+2\le u_{\coord{3y}}$ \\
              & \multirow{2}{*}{$1$} & \multirow{2}{*}{$[2/3,1)\times [0,1/3)$} & $\coord{3x}=\coord{3y}$ and \\
	                       & & & $2^{\coord{3y}}(3x-1)+2\le l_{\coord{3y}}$ \\
\hline
$A\times F$ & \multirow{6}{*}{$1$} & \multirow{6}{*}{$[0,1)^2$} & \multirow{6}{*}{$x+y\ge 1$} \\
$B\times F$ & & & \\
$D_A\times F$ & & & \\
$\vdots$ & & & \\
$D_G\times F$ & & & \\
$F\times F$ & & & \\
\hline
\end{tabular}
}
\end{center}
\caption{The definition of the values $W_P(x,y)$ for $x,y$ from the parts $A$, $B$, $C$, $D_A$, $D_B$, \dots, $D_G$, $E$, and $F$
         except for the tiles $D_Z\times D_{Z'}$, $Z,Z'\in\{A,\ldots,G\}$.
	 The definition for symmetric pairs of tiles is omitted, e.g., the values on the tile $A\times B$ define the values on the tile $B\times A$. In each case, the tile is set to $0$ if none of the conditions are satisfied.}\label{tab-def-WP}
\end{table}

We now set $W_P(\iota_X(x),\iota_Q(y))$ to be equal to the following integral
$$W_P(\iota_X(x),\iota_Q(y))=\frac{1}{12}\left(12-\sum_{Z\in\{A,B,C,D_A,\ldots,D_G,E,F\}}\;\int\limits_{[0,1)}W_P(\iota_X(x),\iota_Z(z))\dd z\right)$$
for all $X\in\{A,B,C,D_A,\ldots,D_G,E,F\}$ and all $x,y\in [0,1)$.
Note that the sum above has $12$ terms, so its value is always between $0$ and $1$ (inclusively).
Finally, we set $W_P$ to be equal to $1$ on the tiles $Q\times Q$ and $Q \times R$,
to $1/25,\ldots,12/25$ on the tiles $A\times R$, $B\times R$, $C\times R$, $D_A\times R$, $D_B\times R$, \dots, $D_G\times R$, $E\times R$, and $F\times R$, respectively, and
to zero on the tile $R\times R$.
This finishes the definition of the graphon $W_P$.

We next prove the following lemma on the structure of the just defined graphons $W_P(z)$.

\begin{lemma}
\label{lm-dep-z}
Let $P=(p_i,l_i,u_i)_{i\in\NN}$ be a bounding sequence.
For any $z,z'\in [0,1]^\NN$,
the graphons $W_P(z)$ and $W_P(z')$
are the same everywhere except for the tiles $C\times C$ and $C\times E$.
\end{lemma}

\begin{proof}
Inspecting the definition of the graphon $W_P(z)$,
it is easy to notice that the tiles $C\times C$ and $C\times E$ are indeed
the only tiles among the tiles $X\times Y$, $X,Y\in\{A,B,C,D_A,\ldots,D_G,E,F,R\}$
with the structure depending on the choice of $z$.
However, the changes inside the tiles $C\times C$ and $C\times E$
could possibly result in a change of the tile $C\times Q$ or $E\times Q$.
This can happen only if there exists $x\in C$ such that
the values of $\deg_{W_P(z)}^C(x)$ and $\deg_{W_P(z)}^E(x)$ are not constant as $z$ varies, or
there exists $x\in E$ such that the value of $\deg_{W_P(z)}^C(x)$ varies.
Observe, however, that the following holds:
$$
\begin{array}{ccll}
\deg_{W_P(z)}^C(\iota_C(x)) & = & \left\{\begin{array}{c} 2^{-\coord{3x}}/3 \\ 0 \end{array}\right. & 
                                                \begin{array}{l} \mbox{for $x\in [0,2/3)$,}\\ \mbox{for $x\in [2/3,1)$,}\\ \end{array} \\                              
\deg_{W_P(z)}^E(\iota_C(x)) & = & \left\{\begin{array}{c} 2^{-\coord{3x}}/3 \\ 0 \end{array}\right. & 
                                                \begin{array}{l} \mbox{for $x\in [0,2/3)$,}\\ \mbox{for $x\in [2/3,1)$,}\\ \end{array} \\                              
\deg_{W_P(z)}^C(\iota_E(x)) & = & \left\{\begin{array}{c} 2^{-\coord{3x}}/3 \\ 0 \end{array}\right. & 
                                                \begin{array}{l} \mbox{for $x\in [0,2/3)$,}\\ \mbox{for $x\in [2/3,1)$.}\\ \end{array} \\
\end{array}
$$
Hence, none of the relative degrees depend on $z$,
which implies that the tiles $C\times Q$ and $E\times Q$ also do not depend on $z$.
\end{proof}

Lemma~\ref{lm-dep-z} implies that for each $X\in\{A,B,C,D_A,\ldots,D_G,E,F,Q,R\}$,
the tile $X\times Q$ is the same in all graphons $W_P(z)$, $z\in [0,1]^\NN$,
i.e., the tile is independent of the choice of $z$.
In particular, the degree of the part $Q$ does not depend on $z\in [0,1]^\NN$.

We next prove the following claim about the graphons $W_P(z)$.

\begin{lemma}
\label{lm-ff}
Let $P$ be a bounding sequence.
The graphons $W_P(z)$ and $W_P(z')$, $z,z'\in [0,1]^\NN$,
are weakly isomorphic if and only if $z=z'$.
\end{lemma}

\begin{proof}
Fix $z\in [0,1]^\NN$.
In $W_P(z)$, the vertices of degree $1203/2500$ uniquely determine the part $C$ of the graphon.
The part $C$ can be uniquely split into disjoint measurable subsets $C_i$, $i\in\NN_0$ with $|C_0|=1/3$ and $|C_i|=2^{-i+1}/3$ as follows. If $x \in C_0$, then $\deg_W^C(x)=0$, and if $x \in C_i$ for $i \in \NN$, then $\deg_W^C(x)=2^{-i}/3$. Inside the tile $C \times C$, the
graphon is non-zero only on the sets $C_i\times C_i$, $i\in\NN$. 
The structure of $W_P$ restricted to $C_i\times C_i$ uniquely determines the value of $z^{M_i}$.
In particular, the values of all $z_i$, $i\in\NN$, are uniquely determined in this way.
It follows that if two graphons $W_P(z)$ and $W_P(z')$ are weakly isomorphic,
then $z=z'$.
\end{proof}

The remainder of this section is devoted to the proof of the following theorem.

\begin{theorem}
\label{thm-ff}
There exist graphs $H_1,\ldots,H_{\ell}$ and an integer $D$ with the following property.
For any bounding sequence $P=(p_i,l_i,u_i)_{i\in\NN}$,
there exists a polynomial $q$ of degree at most $D$ in $\ell$ variables such that the following two statements are equivalent for any graphon $W$.
\begin{itemize}
\item The graphon $W$ is weakly isomorphic to a graphon $W_P(z)$ for some $z\in [0,1]^\NN$ such that
      $p_i(z)\in [l_i,u_i]$ for all $i\in\NN$.
\item It holds that $q(d(H_1,W),\ldots,d(H_\ell,W))=0$.
\end{itemize}
\end{theorem}

\subsection{Proof of Theorem~\ref{thm-ff} -- general setting}
\label{sub-setting}

Theorem~\ref{thm-ff} follows from the following statement:
there exists a family $\CC$ of ordinary and decorated density constraints such that
\begin{itemize}
\item the graphs appearing in $\CC$ do not depend on $P$, and
\item a graphon $W$ satisfies all constraints in $\CC$
if and only if $W$ is weakly isomorphic to a graphon $W_P(z)$ for some $z\in [0,1]^\NN$, with
$p_i(z)\in [l_i,u_i]$ for all $i\in\NN$.
\end{itemize}
\noindent Indeed, if we find such a family $\CC$, Lemma~\ref{lm-decorated} would imply that
there exists a family $\CC'$ of ordinary density constraints with the same properties.
Each constraint in $\CC'$ can be thought of as being a polynomial $p$ in the densities of graphs appearing in the constraint
such that $p=0$ if and only if the constraint is satisfied.
The sought polynomial $q$ can then be set to be the sum 
of $p^2$ taken over all constraints in $\CC$.
Hence, we focus on finding a family $\CC$ with the above properties.

The constraints forming the family $\CC$ will be presented in this and the following subsections
together with the related parts of the proof of Theorem~\ref{thm-ff}.
The family $\CC$ contains the ordinary density constraints that are satisfied
precisely by graphons with the same part sizes and degrees as graphons $W_P(z)$, $z\in [0,1]^\NN$;
such ordinary density constraints exist by Lemma~\ref{lm-decorated}.
Note that the sizes and the degrees of the parts do not depend on $z\in [0,1]^\NN$ (see the remark after Lemma~\ref{lm-dep-z}).

Suppose that $W$ is a graphon satisfying all constraints in $\CC$.
Then $W$ is a partitioned graphon with parts corresponding to those of $W_P$, and we will write $A,B,C,D_A,\ldots,D_G,E,F,Q,R$ for the parts of $W$.
We will show that there exists a choice of $z\in [0,1]^\NN$ satisfying all constraints implied by the bounding sequence $P$ such that
$W$ and $W_P(z)$ are weakly isomorphic.
To do so, we will show that there exist a vector $z\in [0,1]^\NN$ and a measure preserving map $g:[0,1)\to [0,1)$ such that
$W(x,y)=W_P(z)(g(x),g(y))$ for almost all $(x,y)\in [0,1)^2$.

Let $A^P$, $B^P$, $C^P$, $D^P_A$, \dots, $D^P_G$, $E^P$, $F^P$, $Q^P$, $R^P$ be the half-open subintervals of $[0,1)$
forming the parts of the graphon $W_P$ (we use
the superscripts to make a clear distinction between the parts of $W$ and the parts of $W_P$).
The Monotone Reordering Theorem~\cite[Proposition A.19]{bib-lovasz-book} implies that
for every $X\in\{A,B,D_A,\ldots,D_G,F,Q,R\}$, there exist a measure preserving map $\varphi_X:X\to [0,|X|)$ and
a non-decreasing function $f_X:[0,|X|)\to\RR$ such that
$$f_X(\varphi_X(x))=\deg_W^F(x)=\frac{1}{|F|}\int_F W(x,y)\dd y$$
for almost every $x\in X$.
In addition,
there exist measure preserving maps $\varphi_C:C\to [0,1/25)$ and $\varphi_E:E\to [0,1/25)$ and
non-decreasing functions $f_C:[0,1/25)\to\RR$ and $f_E:[0,1/25)\to\RR$ such that
\begin{equation}
f_C(\varphi_C(x))=900\int_{N_W^A(x)}\deg_W^F(z)\dd z-\deg_W^B(x) \qquad \mbox{and}
\label{eq-fC}
\end{equation}
\begin{eqnarray}
f_E(\varphi_E(x')) & = & 4500\int_{N_W^A(x')}\deg_W^F(z)\dd z-5\deg_W^B(x')-\deg_W^{D_G}(x')\nonumber\\
                   & & +\int_{N_W^C(x')}\deg_W^{D_A}(z)\dd z
\label{eq-fE}
\end{eqnarray}
for almost every $x\in C$ and $x'\in E$.

\begin{figure}
\begin{center}
\epsfbox{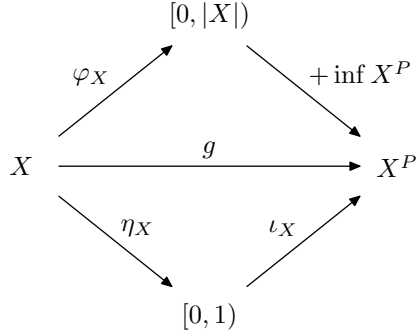}
\end{center}
\caption{The structure of maps used in the proof of Theorem~\ref{thm-ff}.}
\label{fig-maps}
\end{figure}

For $x\in X$, where $X\in\{A,B,C,D_A,\ldots,D_G,E,F,Q,R\}$, we set $g(x)$ to $\inf X^P+\varphi_X(x)$.
Since each $X^P$ is a subinterval of length of $|X|$, $g$ is a measure preserving map from $[0,1)$ to $[0,1)$.
In Subsections~\ref{sub-coord}--\ref{sub-clean}, we will show that $W(x,y)=W_P(g(x),g(y))$ almost everywhere
for a suitable choice of $z\in [0,1]^\NN$.
Finally, for $X\in\{A,B,C,D_A,\ldots,D_G,E,F,Q,R\}$, 
define $\eta_X:X\to[0,1)$ as $\eta_X(x)=\varphi(X)/|X|$.
Note that $\iota_X(\eta_X(x))=g(x)$ for $x\in X$.
The mutual relations of the just defined maps are visualized in Figure~\ref{fig-maps}.

\subsection{Proof of Theorem~\ref{thm-ff} -- coordinate system}
\label{sub-coord}

The half-graphon $W_{\Delta}$ is the graphon such that
$W_{\Delta}(x,y)=1$ if $x+y\ge 1$ and $W_{\Delta}(x,y)=0$, otherwise.
The half-graphon $W_{\Delta}$ is finitely forcible~\cite{bib-diaconis09+,bib-lovasz11+}.
By Lemma~\ref{lm-subforcing}, there exists a collection of decorated constraints that
is satisfied if and only if the tile $F\times F$ is weakly isomorphic to the half-graphon $W_{\Delta}$.
The definitions of $\varphi_F$ and $f_F$ then imply that for almost all $(x,y)\in F\times F$,
$W(x,y)\in\{0,1\}$, and $W(x,y)=1$ iff $f_F(\varphi_F(x))+f_F(\varphi_F(y))\ge 1$.
It follows that $W(x,y)=W_P(g(x),g(y))$ for almost all $(x,y)\in F\times F$ and
$f_F(z)=25z$ for almost every $z\in [0,1/25)$.

\begin{figure}
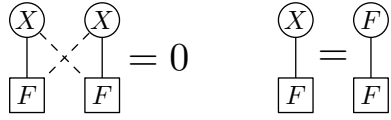

\begin{center}
\epsfbox{exturing-8.mps}
\hskip 10mm
\epsfbox{exturing-9.mps}
\end{center}
\caption{Decorated constraints forcing the structure of the tiles $X\times F$, where $X\in\{A,B,D_A,\ldots,D_G\}$.}
\label{fig-coord}
\end{figure}

Fix $X\in\{A,B,D_A,\ldots,D_G\}$ and consider the two decorated constraints depicted in Figure~\ref{fig-coord}. The first constraint implies that $W(x,y) \in \{0,1\}$ for almost all $(x,y)\in X\times F$. In particular, for almost all $y \in F$, it holds that $W(x,y) \in \{0,1\}$ for almost all $x \in X$. Furthermore, for almost every pair $y,y' \in F$, we have $N_X(y) \sqsubseteq N_X(y')$ or $N_X(y') \sqsubseteq N_X(y)$, which implies that for almost all $y \in F$, it holds that for almost every $y' \in F$, we have $N_X(y) \sqsubseteq N_X(y')$ or $N_X(y') \sqsubseteq N_X(y)$. 
The second constraint implies that $\deg^X(y)=\deg^F(y)=f_F(\varphi_X(y))=25\varphi_X(y)$ for almost every $y\in F$. Let $F'$ consist of those points in $y \in F$ for which
\begin{itemize}
	\item $W(x,y) \in \{0,1\}$ for almost all $x \in X$,
	\item $N_X(y) \sqsubseteq N_X(y')$ or $N_X(y') \sqsubseteq N_X(y)$ for almost all $y' \in F$, and
	\item $\deg^X(y)=25\varphi_F(y)$,
\end{itemize} 
and note that $|F\setminus F'|=0$. 
Fix $y \in F'$, and let $N=N_X(y)$ and $d=\deg^X(y)=|N|=f_F(\varphi_F(y))$. 
For almost every $y' \in F$ with $\varphi_F(y')>\varphi_F(y)$, we have that
$\deg^X(y')=25\varphi_F(y')>d$ and $N_X(y') \sqsupseteq N_X(y)$,
which implies that $W(x,y')=1$ for almost every $x \in N$. Since the measure of points in $F$ with $\varphi_F(y')>\varphi_F(y)$ is $\frac{1-d}{|F|}$, 
this implies that for almost every $x \in N$, we have $d_F(x) \ge 1-d$. 
An analogous argument implies that for almost every $x \in X \setminus N$, $d_F(x) \le 1-d$.
Since $f_X(\varphi_X(x))=\deg_F(x)$ for almost every $x \in X$, we have obtained that
for almost all $(x,y)\in X\times F$, $W(x,y)=1$ iff $f_X(\varphi_X(x))+f_F(\varphi_F(y))\ge 1$.
It follows that $W(x,y)=W_P(g(x),g(y))$ for almost all $(x,y)\in X\times F$ and
that $f_X(z)=25z$ for almost every $z\in [0,1/25)$.
In particular, it holds that $\eta_X(X)=f_X(\varphi_X(x))$ for $x\in X$ and $X\in\{A,B,D_A,\ldots,D_G,F\}$,
i.e., we can think of the relative degree $\deg^F(x)=\eta_X(x)$ as the coordinate of the vertex $x\in X$.

\subsection{Proof of Theorem~\ref{thm-ff} -- black box}
\label{sub-blackbox}

We now inspect the proof of Theorem~\ref{thm-universal} presented in~\cite{bib-universal};
the graphon $W_0$ from the statement of the theorem is depicted in Figure~\ref{fig-universal}.
The proof starts by presenting constraints that introduce the partition and the coordinate system
analogously to Subsections~\ref{sub-setting} and~\ref{sub-coord}.
In particular, the relative degrees of vertices in the graphon $W_0$ from Theorem~\ref{thm-universal} with respect to the part $P$
play the role of the coordinates that we have introduced in Subsection~\ref{sub-coord}.
The proof of Theorem~\ref{thm-universal} then continues by forcing the structure of the tiles
$X\times Y$, $X,Y\in\{A,\ldots,G\}$, of the graphon $W_0$.
This is done in Sections 3--5 of~\cite{bib-universal} with the empty tiles $B\times G$, $C\times G$ and $D\times G$
handled at the beginning of Section 6.
We now include all decorated constraints depicted in Figures 5, 6, 8--11, 13, 15, 16, 18, 20 and 21 in~\cite{bib-universal}, and
the first constraint in Figure 22 from~\cite{bib-universal} to the family $\CC$
with the decoration $X$, $X\in\{A,\ldots,G\}$, replaced with $D_X$, and
the decoration $P$ replaced with $F$.
The arguments presented in Sections 3--5 of~\cite{bib-universal} apply in the same way in our setting.
In particular, we get that $W(x,y)=W_P(g(x),g(y))$ for almost all $(x,y)\in X\times Y$ and all $X,Y\in\{D_A,\ldots,D_G\}$.

\subsection{Proof of Theorem~\ref{thm-ff} -- auxiliary tiles}
\label{sub-aux}

In this section, we present the constraints forcing most of the tiles whose structure
does not depend on either $P$ or $z$.
We start by forcing the density of the tiles that are zero to be zero,
which is enforced by the constraints in Figure~\ref{fig-zero}.

\begin{figure}
\begin{center}
\epsfbox{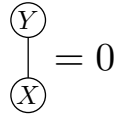}
\end{center}
\caption{Decorated constraints forcing the zero tiles; $(X,Y)$ is one of the pairs
         $(C,D_D)$, $(C,D_E)$, $(C,D_F)$, $(C,F)$, $(E,D_A)$, $(E,D_B)$, $(E,D_C)$, $(E,D_D)$, $(E,D_E)$, $(E,D_F)$,
	 $(E,E)$, $(E,F)$ and $(R,R)$.}
\label{fig-zero}
\end{figure}

\begin{figure}
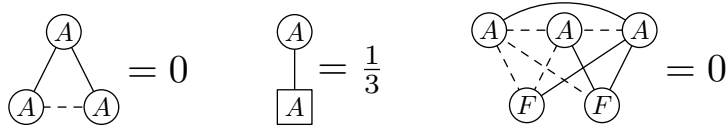

\begin{center}
\epsfbox{exturing-11.mps}
\hskip 10mm
\epsfbox{exturing-10.mps}
\hskip 10mm
\epsfbox{exturing-12.mps}
\end{center}
\caption{Decorated constraints forcing the structure of the tile $A\times A$.}
\label{fig-AA}
\end{figure}

Consider the decorated constraints depicted in Figure~\ref{fig-AA}.
The first constraint implies that the following two properties hold for almost every $x \in A$:
\begin{itemize}
\item $W(y,z)=1$ for almost any pair $(y,z) \in N_W^A(x) \times N_W^{A}(x)$,
\item $W(y,z)=0$ for almost any pair $(y,z) \in N_W^A(x) \times N_{1-W}^A(x)$.
\end{itemize}
Let $Z$ be the set of $x\in A$ that have these two properties.
In particular, the sets $N_W^{A}(x)$ and $N_{1-W}^A(x)$ intersect in a set of measure zero for every $x\in Z$,
which implies that $W(x,y)$ is equal to $0$ or $1$ for every $x \in Z$ and almost every $y\in A$.
Furthermore, for any $x,x' \in Z$, the sets $N_W^A(x)$ and $N_W^A(x')$ either differ on a set of measure zero, or their intersection has measure zero. Also note that given $x \in Z$, the set of $x' \in Z$ such that $N_W^A(x)$ and $N_W^A(x')$ differ on a set of measure zero is a measurable set.

Since having the same neighborhood up to measure zero is an equivalence relation,
this implies that we can partition $Z$ into a collection $\JJ$ of disjoint measurable subsets so that
if $x$ and $x'$ belong to the same set,
then their neighborhoods differ on a set of measure zero, and
if they belong to different sets in $\JJ$,
then their neighborhoods intersect in a set of measure zero.
Furthermore, the conditions imply that for $x \in Z$, almost every $y \in N_W^{Z}(x)$ has the property that $N_W^Z(x)$ and $N_W^Z(y)$ differ on a set of measure zero. 
This implies that for almost every $x,y\in A$:
$W(x,y)=1$ if and only if there exists $J\in\JJ$ that contains both $x$ and $y$.
The second constraint implies that almost every $x\in A$ belongs to some $J\in\JJ$ and the measure of $J$ is $|A|/3$.
Hence, $\JJ$ contains exactly three sets and each has measure $|A|/3$.

Finally, the last constraint implies that for each $J \in \JJ$ and for almost all $x \in A$,
if the measure of $x'\in J$ with $\eta_A(x')<\eta_A(x)$ is positive and
the measure of $x''\in J$ with $\eta_A(x'')>\eta_A(x)$ is positive,
then $x\in J$.
This implies that each $J\in\JJ$ differs from a preimage of an interval under $\eta_A$ in a set of measure zero.
Since each set in $\JJ$ has measure $1/3$,
the three sets contained in $\JJ$ differ from $\eta_A^{-1}([0,1/3))$, $\eta_A^{-1}([1/3,2/3))$ and
$\eta_A^{-1}([2/3,1))$ on a set of measure zero.
It follows that $W(x,y)=W_P(g(x),g(y))$ for almost all $(x,y)\in A\times A$.

\begin{figure}
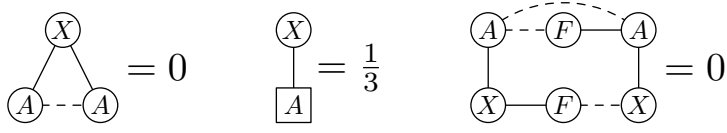

\begin{center}
\epsfbox{exturing-13.mps}
\hskip 10mm
\epsfbox{exturing-14.mps}
\hskip 10mm
\epsfbox{exturing-15.mps}
\end{center}
\caption{Decorated constraints forcing the structure of the tiles $A\times X$, where $X\in\{B,C,D_A,\ldots,D_G,E\}$.}
\label{fig-AX}
\end{figure}

Fix $X\in\{B,C,D_A,\ldots,D_G,E\}$ and consider the decorated constraints depicted in Figure~\ref{fig-AX}.
The first constraint implies that almost every $x \in X$ has a neighborhood that is almost entirely contained in $\eta_A^{-1}([0,1/3))$, $\eta_A^{-1}([1/3,2/3))$, or $\eta_A^{-1}([2/3,1))$.
This implies that there exist disjoint measurable subsets $J_1$, $J_2$ and $J_3$ of $X$ such that
$N^X(x)\sqsubseteq J_1$ for almost every $x\in\eta_A^{-1}([0,1/3))$,
$N^X(x)\sqsubseteq J_2$ for almost every $x\in\eta_A^{-1}([1/3,2/3))$ and
$N^X(x)\sqsubseteq J_3$ for almost every $x\in\eta_A^{-1}([2/3,1))$.
The second constraint implies that $\deg^X(x)=1/3$ for almost every $x\in A$,
which implies that $|J_1|=|J_2|=|J_3|=1/3$ and, up to a set of measure zero,
$W(x,y)=1$ for $(x,y)\in A\times X$ if and only if either $\eta_A(x)\in [0,1/3)$ and $y\in J_1$, or
$\eta_A(x)\in [1/3,2/3)$ and $y\in J_2$, or $\eta_A(x)\in [2/3,1)$ and $y\in J_3$.
The last constraint trivially holds if $X=C$ or $X=E$, since in this case the tile $X \times F$ is $0$ almost everywhere.
If $X\not\in\{C,E\}$, the constraint implies that 
$\eta_X(y_1)\le\eta_X(y_2)\le\eta_X(y_3)$ for almost any $y_1\in J_1$, $y_2\in J_2$ and $y_3\in J_3$.
Hence, we can assume that $J_1=\eta_X^{-1}([0,1/3))$, $J_2=\eta_X^{-1}([1/3,2/3))$ and $J_3=\eta_X^{-1}([2/3,1))$.
We conclude that $W(x,y)=W_P(g(x),g(y))$ for almost all $(x,y)\in A\times X$,
where $X\in\{B,D_A,\ldots,D_G\}$.

If $X\in\{C,E\}$, we argue as follows. 
For $x\in X$, the first integral in \eqref{eq-fC} and \eqref{eq-fE} can have a value of $1/450$, $3/450$ or $5/450$.
The differences between these values after multiplying by 900 or 4500, respectively,
are larger than the maximum possible variation of the rest of the expressions in \eqref{eq-fC} and \eqref{eq-fE} (it is easy to see that they vary within an interval of length $1$ and $7$, respectively).
It follows that $J_1=\eta_X^{-1}([0,1/3))$, $J_2=\eta_X^{-1}([1/3,2/3))$ and $J_3=\eta_X^{-1}([2/3,1))$,
which implies that $W(x,y)=W_P(g(x),g(y))$ for almost all $(x,y)\in A\times X$, $X\in\{C,E\}$.

\begin{figure}
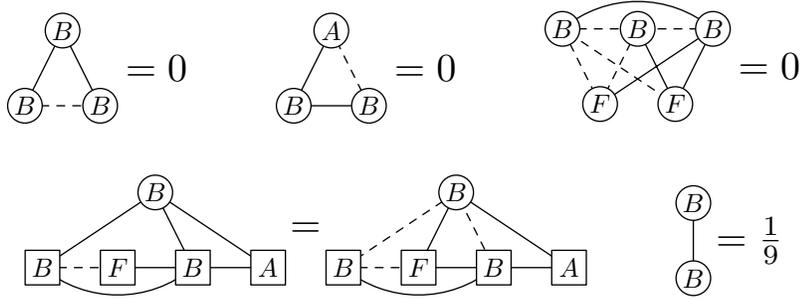

\begin{center}
\epsfbox{exturing-17.mps}
\hskip 10mm
\epsfbox{exturing-18.mps}
\hskip 10mm
\epsfbox{exturing-19.mps}
\vskip 6mm
\epsfbox{exturing-20.mps}
\hskip 10mm
\epsfbox{exturing-21.mps}
\end{center}
\caption{Decorated constraints forcing the structure of the tile $B\times B$.}
\label{fig-BB}
\end{figure}

We now turn our attention to the tile $B\times B$ and
consider the decorated constraints depicted in Figure~\ref{fig-BB}.
We start with the three constraints on the first line.
Following the arguments presented for the first constraint in Figure~\ref{fig-AA},
we conclude that
there exists a collection $\JJ$ of disjoint measurable subsets of $B$ with positive measure such that 
for almost every $x,y\in B$,
$W(x,y)=1$ if and only if there exists $J\in\JJ$ that contains both $x$ and $y$.
The second constraint implies that $J\sqsubseteq\eta_B^{-1}([0,1/3))$, $J\sqsubseteq\eta_B^{-1}([1/3,2/3))$, or
$J\sqsubseteq\eta_B^{-1}([2/3,1))$ for each $J\in\JJ$.
Finally, the third constraint implies that for each $J\in\JJ$,
there exists an interval $J'\subseteq [0,1)$ such that $J$ and $\eta_B^{-1}(J')$
differ on a set of measure zero (by the same argument as for the tile $A\times A$).
Let $\JJ'$ be the set of such intervals $J'$ for each $J\in\JJ$.
Note that each interval in $\JJ'$ is a subinterval of $[0,1/3)$, $[1/3,2/3)$ or $[2/3,1)$.
The first constraint on the second line implies that
if $J'=[a,b)$ is a subinterval of $[(i-1)/3,i/3)$, $i\in [3]$, then
$$b-a=i/3-b\,\mbox{.}$$
Since the density of the tile $B \times B$ is equal to the sum of $|J'|^2$ for $J'\in\JJ'$, the density of $B \times B$ is at most $1/9$, with equality if and only if
each $J'\in\JJ'$ is of the form
$[(i-2^{-(j-1)})/3,(i-2^{-j})/3)$ for some $i\in [3]$ and $j\in\NN$.
We conclude from the last constraint that $W(x,y)=W_P(g(x),g(y))$ for almost all $(x,y)\in B\times B$.

\begin{figure}
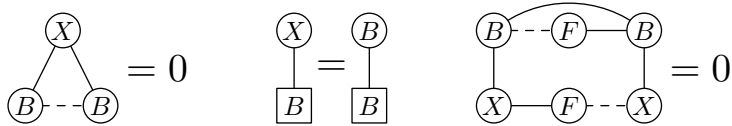

\begin{center}
\epsfbox{exturing-22.mps}
\hskip 10mm
\epsfbox{exturing-23.mps}
\hskip 10mm
\epsfbox{exturing-24.mps}
\end{center}
\caption{Decorated constraints forcing the structure of the tiles $B\times X$, where $X\in\{C,D_A,\ldots,D_G,E\}$.}
\label{fig-BX}
\end{figure}

Fix $X\in\{C,D_A,\ldots,D_G,E\}$ and consider the decorated constraints depicted in Figure~\ref{fig-BX}.
Let $\JJ'$ be the set containing the intervals from the analysis of the tile $B\times B$.
Following the arguments for Figure \ref{fig-AX}, the first constraint implies that for each $J'\in\JJ'$,
there exists a measurable subset $J$ such that
$N^X(x)\sqsubseteq J$ for each $x\in\eta_B^{-1}(J')$. Furthermore,
the subsets $J$ are disjoint for different $J'\in\JJ'$.
The second constraint implies that $\deg^X(x)=\deg^B(x)=|J'|$ for almost every $x\in\eta_B^{-1}(J')$,
which implies that $|J|=|J'|\cdot |X|$ for every $J'\in\JJ'$ (since the sum of the measures of the intervals in $\JJ'$ is one), and
$N^X(x)\sqsupseteq J$.
If $X\in\{D_A,\ldots,D_G\}$, the last constraint implies that
each $J$ is a preimage of an interval and these intervals follow the order of the intervals in $\JJ'$ (note that as before,
this constraint trivially holds if $X=C$ or $X=E$).
It follows that $J$ and $\eta_X^{-1}(J')$ differ on a set of measure zero.
We conclude that $W(x,y)=W_P(g(x),g(y))$ for almost all $(x,y)\in B\times X$,
where $X\in\{D_A,\ldots,D_G\}$.
If $X=C$, then \eqref{eq-fC} implies that $J$ is a preimage of an interval and
these intervals follow the order of the intervals in $\JJ'$,
which again leads to the conclusion that $W(x,y)=W_P(g(x),g(y))$ for almost all $(x,y)\in B\times C$.
The analysis of the tile $B\times E$ will be finished in Subsection~\ref{sub-poly}.

\begin{figure}
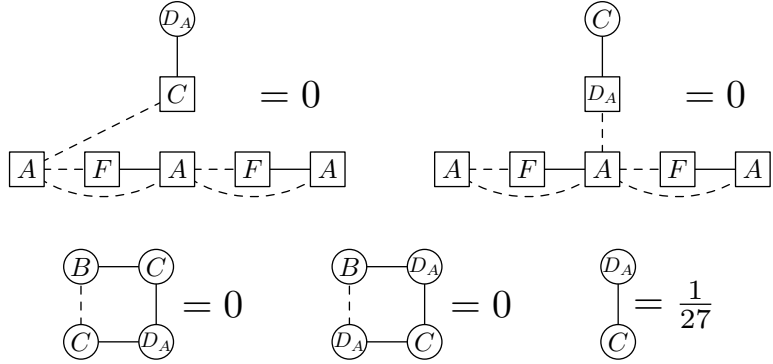

\begin{center}
\epsfbox{exturing-25.mps}
\hskip 10mm
\epsfbox{exturing-26.mps}
\vskip 6mm
\epsfbox{exturing-27.mps}
\hskip 10mm
\epsfbox{exturing-28.mps}
\hskip 10mm
\epsfbox{exturing-29.mps}
\end{center}
\caption{Decorated constraints forcing the structure of the tile $C\times D_A$.}
\label{fig-CDA}
\end{figure}

We now analyze the tile $C\times D_A$.
Consider the decorated constraints depicted in Figure~\ref{fig-CDA}.
Let $a_1$, $a_2$ and $a_3$ be the three $A$-roots from the constraints on the first line in the figure.
Almost any choice of the roots satisfies that $\eta_A(a_1)<\eta_A(a_2)<\eta_A(a_3)$.
Since the three roots are non-adjacent, it follows that almost any choice of them satisfies that
$\eta_A(a_i)\in [(i-1)/3,i/3)$, $i\in [3]$, because of the structure of the tile $A\times A$.
For the first constraint, the structure of the tile $A\times C$ implies that
almost any choice of the $C$-root $x$ satisfies that $\eta_C(x)\in [1/3,1)$.
Likewise, the structure of the tile $A\times D_A$ implies that
almost any choice of the $D_A$-root $y$ in the second constraint satisfies that $\eta_{D_A}(y)\in [0,1/3)\cup [2/3,1)$.
It follows that $W(x,y)=0$ for almost all $(x,y)\in (C\times D_A)\setminus(\eta_C^{-1}([0,1/3))\times\eta_{D_A}^{-1}([1/3,2/3)))$.

We now focus on the constraints on the second line in Figure~\ref{fig-CDA}.
The first constraint implies that
the neighborhood of almost every $y \in D_A$ is contained in the set
$$\{x \in C\mbox{ such that }\lfloor3\eta_C(x)\rfloor=j\mbox{ and }\coord{3\eta_C(x)}=k\}$$
for some integers $j$ and $k$.
Note that by the previous paragraph we must have $j=0$.
This implies that there exist disjoint measurable subsets $J_i\subseteq D_A$, $i\in\NN$, such that
$N^{D_A}(x)\sqsubseteq J_{\coord{3\eta_C(x)}}$ for almost every $x\in C$.
Likewise, the second constraint implies that there exist disjoint measurable subsets $J'_i\subseteq C$, $i\in\NN$, such that
$N^{C}(y)\sqsubseteq J'_{\coord{3\eta_{D_A}(y)}}$ for almost every $y\in D_A$.
Hence, there exists a function $f:\NN\to\NN_0$ that is injective on $f^{-1}(\NN)$ such that
the following holds for almost every $(x,y)\in C\times D_A$:
$W(x,y)>0$ only if $f(\coord{3\eta_C(x)})=\coord{3\eta_{D_A}(y)}$.
However, the last constraint on the second line can hold only if $f(i)=i$ for all $i\in\NN$ (otherwise, the value would be strictly less than $1/27$).
It follows that $W(x,y)=W_P(g(x),g(y))$ for almost all $(x,y)\in C\times D_A$.
In a very analogous way, the constraints depicted in Figure~\ref{fig-CDB}
guarantee that $W(x,y)=W_P(g(x),g(y))$ for almost all $(x,y)\in C\times D_B$.

\begin{figure}
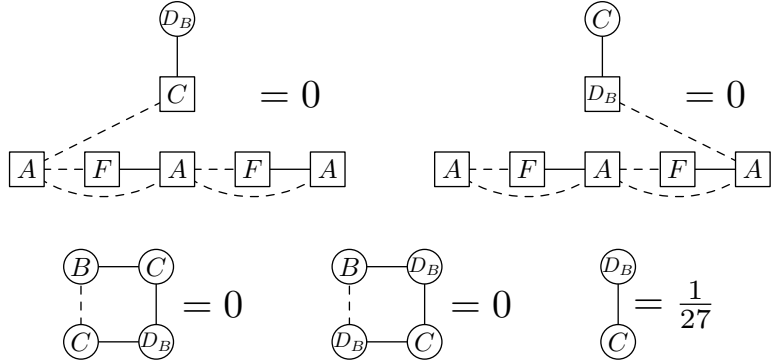

\begin{center}
\epsfbox{exturing-30.mps}
\hskip 10mm
\epsfbox{exturing-31.mps}
\vskip 6mm
\epsfbox{exturing-32.mps}
\hskip 10mm
\epsfbox{exturing-33.mps}
\hskip 10mm
\epsfbox{exturing-34.mps}
\end{center}
\caption{Decorated constraints forcing the structure of the tile $C\times D_B$.}
\label{fig-CDB}
\end{figure}

\begin{figure}
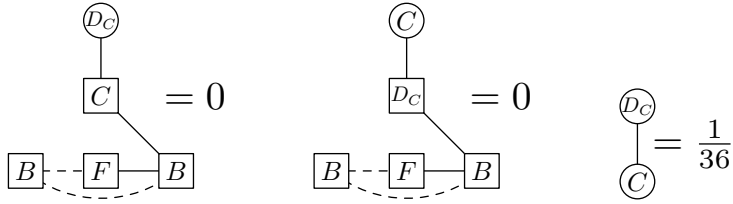

\begin{center}
\epsfbox{exturing-35.mps}
\hskip 10mm
\epsfbox{exturing-36.mps}
\hskip 10mm
\epsfbox{exturing-37.mps}
\end{center}
\caption{Decorated constraints forcing the structure of the tile $C\times D_C$.}
\label{fig-CDC}
\end{figure}

We now consider the decorated constraints depicted in Figure~\ref{fig-CDC} and analyze the tile $C\times D_C$.
Let $b_1$ and $b_2$ be the two $B$-roots in the first constraint.
The choices of the roots vary through almost all pairs $(b_1,b_2)\in B\times B$ such that
either $\lfloor 3\eta_B(b_1)\rfloor<\lfloor 3\eta_B(b_2)\rfloor$, or
$\lfloor 3\eta_B(b_1)\rfloor=\lfloor 3\eta_B(b_2)\rfloor$ and $\coord{3\eta_B(b_1)}<\coord{3\eta_B(b_2)}$.
Hence, the first constraint implies that $\deg^{D_C}(x)=0$ for any $x\in C$ with $\eta_C(x)\in [1/6,1)$.
Similarly, the second constraint implies that $\deg^C(y)=0$ for any $y\in D_C$ with $\eta_{D_C}(y)\in [1/6,1)$.
Since the density of the tile $C\times D_C$ is $1/36$ by the third constraint,
it follows that $W(x,y)=1$ for almost every $(x,y)\in\eta_C^{-1}([0,1/6))\times\eta_{D_C}^{-1}([0,1/6))$, and
$W(x,y)=0$ for almost every $(x,y)\in (C\times D_C)\setminus (\eta_C^{-1}([0,1/6))\times\eta_{D_C}^{-1}([0,1/6)))$.
We conclude that $W(x,y)=W_P(g(x),g(y))$ for almost all $(x,y)\in C\times D_C$.

In the subsequent sections, we finish the proof by forcing the structure of the tiles $C\times C$, $C\times D_G$, $C\times E$ and $D_G\times E$, forcing all tiles involving parts $Q$ and $R$, and finishing the analysis of the tile $B\times E$.

\subsection{Proof of Theorem~\ref{thm-ff} -- bounding polynomials}
\label{sub-poly}

\begin{figure}
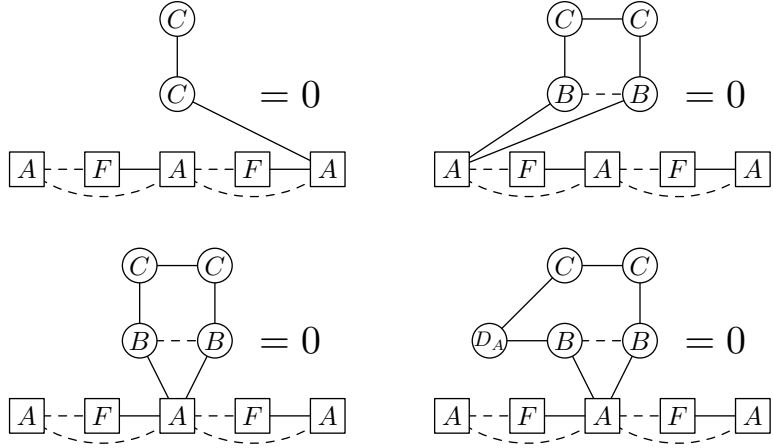

\begin{center}
\epsfbox{exturing-38.mps}
\hskip 10mm
\epsfbox{exturing-39.mps}
\vskip 6mm
\epsfbox{exturing-40.mps}
\hskip 10mm
\epsfbox{exturing-41.mps}
\end{center}
\caption{Decorated constraints forcing the general structure of the tile $C\times C$.}
\label{fig-CC-gen}
\end{figure}

We next force some general structure of the tiles $C\times C$, %$C\times D_G$,
$C\times E$ and $D_G\times E$.
Consider the decorated constraints depicted in Figure~\ref{fig-CC-gen}.
In all four constraints, almost every choice of the $A$-roots $a_1$, $a_2$ and $a_3$
satisfies that $\eta_A(a_1)\in [0,1/3)$, $\eta_A(a_2)\in [1/3,2/3)$ and $\eta_A(a_3)\in [2/3,1)$.
The first constraint then implies that $\deg^C(x)=0$ for almost every $x\in\eta_C^{-1}([2/3,1))$.
In the second constraint, almost any pair of non-adjacent $B$-vertices $b_1$ and $b_2$ adjacent to $a_1\in\eta_A^{-1}([0,1/3))$ satisfies that $\coord{3\eta_B(b_1)}\not=\coord{3\eta_B(b_2)}$.
It follows that $W(x,y)=0$ for almost every $(x,y)\in\eta_C^{-1}([0,1/3))\times\eta_C^{-1}([0,1/3))$
such that $\coord{3\eta_C(x)}\not=\coord{3\eta_C(y)}$.
Similarly, the first constraint on the second line yields that
$W(x,y)=0$ for almost every $(x,y)\in\eta_C^{-1}([1/3,2/3))\times\eta_C^{-1}([1/3,2/3))$
such that $\coord{3\eta_C(x)}\not=\coord{3\eta_C(y)}$.
Finally, in the last constraint, almost any neighbor $z\in D_A$ of $b_1$ satisfies that
$\eta_{D_A}(z)\in [1/3,2/3)$ and $\coord{3\eta_{D_A}(z)}=\coord{3\eta_B(b_1)}$.
It follows that $W(x,y)=0$ for almost every $(x,y)\in\eta_C^{-1}([0,1/3))\times\eta_C^{-1}([1/3,2/3))$
such that $\coord{3\eta_C(x)}\not=\coord{3\eta_C(y)}$.
We conclude that $W(x,y)=0$ for almost every pair $(x,y)\in C\times C$,
with the possible exception of those $(x,y)$ such that
$\eta_C(x)\in [0,2/3)$, $\eta_C(y)\in [0,2/3)$, and $\coord{3\eta_C(x)}=\coord{3\eta_C(y)}$.

\begin{figure}
\begin{center}
\epsfbox{exturing-42.mps}
\hskip 10mm
\epsfbox{exturing-43.mps}
\vskip 6mm
\epsfbox{exturing-44.mps}
\hskip 10mm
\epsfbox{exturing-45.mps}
\vskip 6mm
\epsfbox{exturing-46.mps}
\hskip 10mm
\epsfbox{exturing-47.mps}
\end{center}
\caption{Decorated constraints forcing the general structure of the tile $C\times E$.}
\label{fig-CE-gen}
\end{figure}

\begin{figure}
\begin{center}
\epsfbox{exturing-48.mps}
\hskip 10mm
\epsfbox{exturing-49.mps}
\vskip 6mm
\epsfbox{exturing-50.mps}
\hskip 10mm
\epsfbox{exturing-51.mps}
\end{center}
\caption{Decorated constraints forcing the general structure of the tile $D_G\times E$.}
\label{fig-DGE-gen}
\end{figure}

In Subsection~\ref{sub-aux},
we showed that
there exist disjoint measurable subsets $J_{i,j}\subseteq E$, $i\in [3]$ and $j\in\NN$, with $|J_{i,j}|=2^{-j}/3$ such that
the following holds for almost every $x\in B$:
$W(x,y)=1$ for almost every $y\in J_{\lfloor 3\eta_B(x)\rfloor+1,\coord{3\eta_B(x)}}$ and
$W(x,y)=0$ for almost every $y\not\in J_{\lfloor 3\eta_B(x)\rfloor+1,\coord{3\eta_B(x)}}$.
Following the lines of the analysis of the constraints in Figure~\ref{fig-CC-gen} with respect to the tile $C\times C$,
we conclude that the decorated constraints depicted in Figure~\ref{fig-CE-gen} imply that
$W(x,y)=0$ for almost every pair $(x,y)\in C\times E$ that does not satisfy that
$\eta_C(x)\in [0,2/3)$ and $y\in J_{1,\coord{3\eta_C(x)}}\cup J_{2,\coord{3\eta_C(x)}}$.
Similarly, the decorated constraints depicted in Figure~\ref{fig-DGE-gen} imply that
$W(x,y)=0$ for almost every pair $(x,y)\in D_G\times E$ that does not satisfy that
$\eta_{D_G}(x)\in [1/3,1)$ and $y\in J_{1,\coord{3\eta_{D_G}(x)}}$.

We are now ready to finish the analysis of $B \times E$.
The structure of the tiles $C \times E$ and $D_G \times E$ implies that for almost every $x\in E$,
$\deg_W^C(x) \le \deg_W^B(x)$ and $\deg_W^{D_G}(x)\le 2\deg_W^B(x)$. This implies that 
\[5\deg_W^B(x)+\deg_W^{D_G}(x)\nonumber -\int_{N_W^C(x)}\deg_W^{D_A}(z)\dd z \in [4\deg_W^B(x),7\deg_W^B(x)]\]
for almost every $x\in E$.
In particular, for almost all $x\in J_{i,j}$ and $x'\in J_{i,j'}$ with $i\in [3]$ and $j<j'$, the intervals $[4\deg_W^B(x),7\deg_W^B(x)]$ and $[4\deg_W^B(x'),7\deg_W^B(x')]$ are disjoint, since 
$\deg_W^B(x') \ge 2\deg_W^B(x)$.
This implies that $\eta_E^{-1}(J_{i,j})$ and $[(i-2^{-j-1})/3,(i-2^{-j})/3)$ differ on a set of measure zero for all $i\in [3]$ and $j\in\NN$.
Consequently, $W(x,y)=W_P(g(x),g(y))$ for almost all $(x,y)\in B\times E$.

\begin{figure}
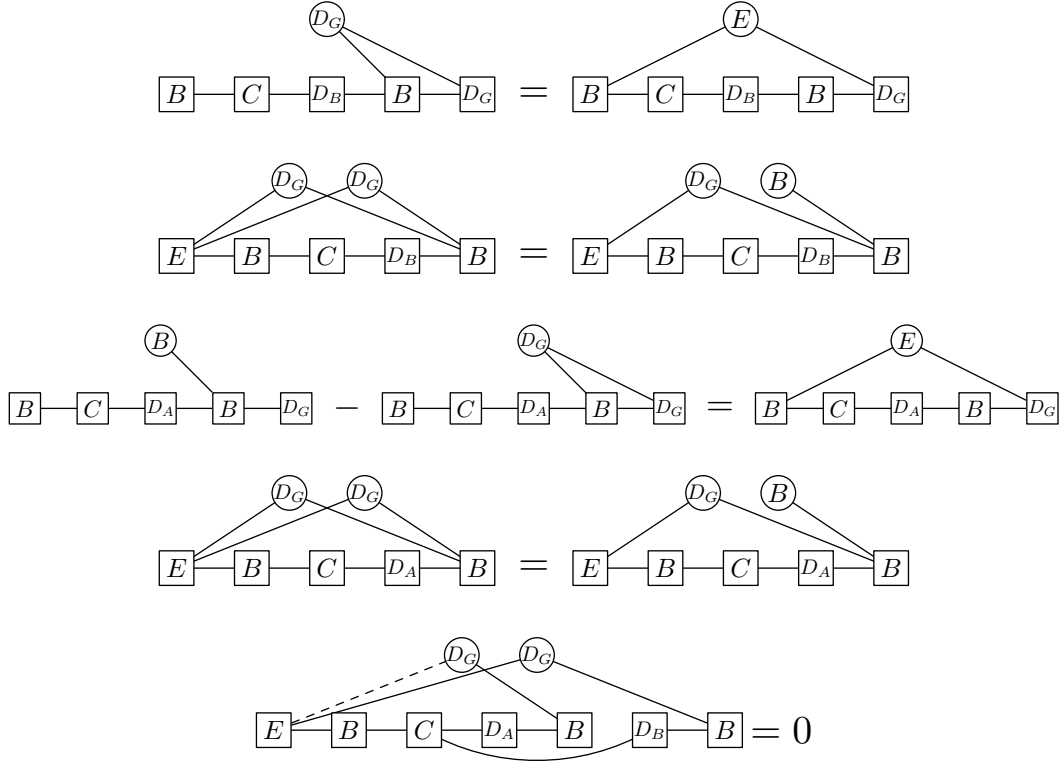

\begin{center}
\epsfbox{exturing-53.mps}
\vskip 6mm
\epsfbox{exturing-55.mps}
\vskip 6mm
\epsfxsize 140mm
\epsfbox{exturing-52.mps}
\vskip 6mm
\epsfbox{exturing-54.mps}
\vskip 6mm
\epsfbox{exturing-56.mps}
\end{center}
\vskip -1mm
\caption{Decorated constraints forcing the specific structure of the tile $D_G\times E$.}
\label{fig-DGE}
\end{figure}

We now analyze the decorated constraints depicted in Figure~\ref{fig-DGE}.
In the first constraint, almost every choice of the $B$-roots $b_1$ and $b_2$ and the $D_G$-root $d$
satisfies that $\eta_B(b_1)\in [0,1/3)$, $\eta_B(b_2)\in [2/3,1)$, $\eta_{D_G}(d)\in [2/3,1)$, and
$\coord{3\eta_B(b_1)}=\coord{3\eta_B(b_2)}=\coord{3\eta_{D_G}(d)}$;
fix such a choice of roots and let $i$ be the common value of $\coord{3\eta_B(b_1)}=\coord{3\eta_B(b_2)}=\coord{3\eta_{D_G}(d)}$.
Together with the structure of $D_G \times D_G$, the constraint implies that $\deg^E(x)=l_i\cdot 2^{-i}/3$ for almost every $x\in\eta_{D_G}^{-1}([2/3,1))$
where $i=\coord{3\eta_{D_G}(d)}$ (the value $l_i$ is given by the bounding sequence $P$).

In the second constraint, almost every choice of the $E$-root $e$ and the two $B$-roots $b_1$ and $b_2$
satisfies that $\eta_E(e)\in [0,1/3)$, $\eta_B(b_1)\in [0,1/3)$, $\eta_B(b_2)\in [2/3,1)$, and
$\coord{3\eta_E(e)}=\coord{3\eta_B(b_1)}=\coord{3\eta_B(b_2)}$.
Let $\alpha$ be the relative degree of $e$ with respect to $\eta_{D_G}^{-1}([2/3,1))$.
The constraint implies that $\alpha^2=\alpha 2^{-i}$, where $i=\coord{3\eta_B(b_2)}$.
Therefore, the constraint is satisfied if and only if $\alpha\in\{0,2^{-i}\}$.
Hence, almost every vertex $x\in\eta_E^{-1}([0,1/3))$ satisfies that
either $W(x,y)=1$ for almost every $y\in\eta_{D_G}^{-1}([2/3,1))$ with $\coord{3\eta_E(x)}=\coord{3\eta_{D_G}(y)}$, or
$W(x,y)=0$ for almost every $y\in\eta_{D_G}^{-1}([2/3,1))$ with $\coord{3\eta_E(x)}=\coord{3\eta_{D_G}(y)}$.
It follows that there exist $J_{1,i}^l\subseteq J_{1,i}$, $i\in\NN$, such that
$|J_{1,i}^l|=l_i  \cdot |J_{1,i}|$ and
for almost every $x\in\eta_E^{-1}([0,1/3))$ and $y\in\eta_{D_G}^{-1}([2/3,1))$
it holds that $W(x,y)=1$ if and only if $x\in J_{1,\coord{3\eta_{D_G}(y)}}^l$, and
$W(x,y)=0$ otherwise.

An analogous argument applies with respect to the third and the fourth constraints;
note that the values on the densities in the diagonal squares in the $\eta_{D_G}^{-1}([1/3,2/3))\times \eta_{D_G}^{-1}([1/3,2/3))$ tiles are $1-u_i$.
Hence, the third and the fourth constraints imply
the existence of $J_{1,i}^u\subseteq J_{1,i}$, $i\in\NN$, such that
$|J_{1,i}^u|=u_i\cdot |J_{1,i}|$ and
for almost every $x\in\eta_E^{-1}([0,1/3))$ and $y\in\eta_{D_G}^{-1}([1/3,2/3))$
it holds $W(x,y)=1$ if and only if $x\in J_{1,\coord{3\eta_{D_G}(y)}}^u$ and
$W(x,y)=0$ otherwise.

We next consider the final constraint in Figure~\ref{fig-DGE}.
Almost every choice of the $E$-root $e$ and the $B$-roots $b_1$, $b_2$ and $b_3$ satisfies that
$\eta_E(e)\in [0,1/3)$, $\eta_B(b_i)\in [(i-1)/3,i/3)$, $i\in\{1,2,3\}$, and
$\coord{3\eta_E(e)}=\coord{3\eta_B(b_1)}=\coord{3\eta_B(b_2)}=\coord{3\eta_B(b_3)}$.
Hence, the constraint is satisfied if and only if $J_{1,i}^l\sqsubseteq J_{1,i}^u$ for all $i\in\NN$.
We now observe that the third term in \eqref{eq-fE} for almost every $x\in E$
is $2^{-i+1}/3$ (if $x\in J_{1,i}^l$), $2^{-i}/3$ (if $x\in J_{1,i}^u\setminus J_{1,i}^l$) or $0$ (if $x\not\in J_{1,i}^u$)
where $i=\coord{3\eta_E(x)}$.
In particular, it dominates the last term in \eqref{eq-fE}, which cannot exceed $2^{-2i+2}/9$.
It follows that
for every $i\in\NN$
it holds $\eta_E(x)\le \eta_E(x')\le \eta_E(x'')$
for almost all $x\in J_{1,i}^l$, $x'\in J_{1,i}^u\setminus J_{1,i}^l$ and $x''\in J_{1,i}\setminus J_{1,i}^u$.
Consequently,
it holds that $W(x,y)=W_P(g(x),g(y))$ for almost all $(x,y)\in E\times D_G$.

\begin{figure}
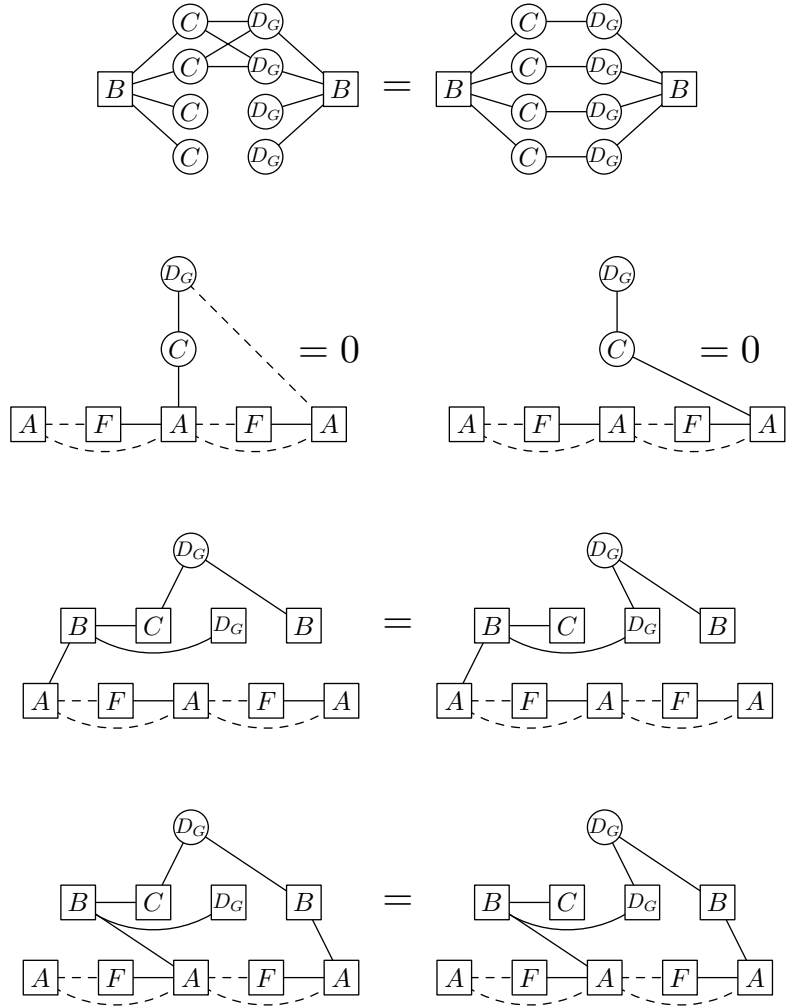

\begin{center}
\epsfbox{exturing-59.mps}
\vskip 10mm
\epsfbox{exturing-77.mps}
\hskip 10mm
\epsfbox{exturing-76.mps}
\vskip 10mm
\epsfbox{exturing-57.mps}
\vskip 10mm
\epsfbox{exturing-58.mps}
\end{center}
\caption{Decorated constraints forcing the structure of the tile $C\times D_G$.}
\label{fig-CDG}
\end{figure}

We now consider the three decorated constraints depicted in Figure~\ref{fig-CDG}.
For $(i,j) \in \{1,2,3\} \times \NN$, let 
\[I_{i,j}=\left[\frac{i-2^{1-j}}{3},\frac{i-2^{-j}}{3}\right).\] 
We start with the first constraint.
Let $b_1$ and $b_2$ be the two $B$-roots and assume that
they belong to $\eta_B^{-1}(I_{i_1,j_1})$ and $\eta_B^{-1}(I_{i_2,j_2})$, respectively, $(i_1,j_1),(i_2,j_2)\in\{1,2,3\}\times\NN$. 
For almost every choice of the roots $b_1$ and $b_2$,
almost all the choices of the $C$-vertices belong to $\eta_C^{-1}(I_{i_1,j_1})$ and
almost all the choices of the $D_G$-vertices to $\eta_{D_G}^{-1}(I_{i_2,j_2})$.
This implies that the density of $4$-cycles in the tile $C\times D_G$ restricted to $I_{i_1,j_1} \times I_{i_2,j_2}$
is equal to the fourth power of the density of the tile restricted to $I_{i_1,j_1} \times I_{i_2,j_2}$.
This implies (see for example the proof of Claim~11.63 in~\cite{bib-lovasz-book}) that
$W$ is equal to a constant $\xi_{i_1,j_1,i_2,j_2}$ almost everywhere
on each rectangle $\eta_C^{-1}(I_{i_1,j_1}) \times \eta_{D_G}^{-1}(I_{i_2,j_2})$ for each $(i_1,j_1),(i_2,j_2)\in\{1,2,3\}\times\NN$.
The two constraints on the second line imply that
$\xi_{i_1,j_1,i_2,j_2}=0$ for $(i_1,i_2)\in\{(2,1),(2,2),(3,1),(3,2),(3,3)\}$.
The constraint on the third line in Figure~\ref{fig-CDG} implies that
the values $\xi_{i_1,j_1,i_2,j_2}$
are equal to the corresponding values inside the tile $D_G\times D_G$
for $(i_1,i_2)\in\{(1,1),(1,2),(1,3)\}$, and
the constraint on the fourth line implies the same for $(i_1,i_2)=(2,3)$.
It now follows that $W(x,y)=W_P(g(x),g(y))$ for almost all $(x,y)\in C\times D_G$.

\subsection{Proof of Theorem~\ref{thm-ff} -- variable tiles}
\label{sub-var}

\begin{figure}
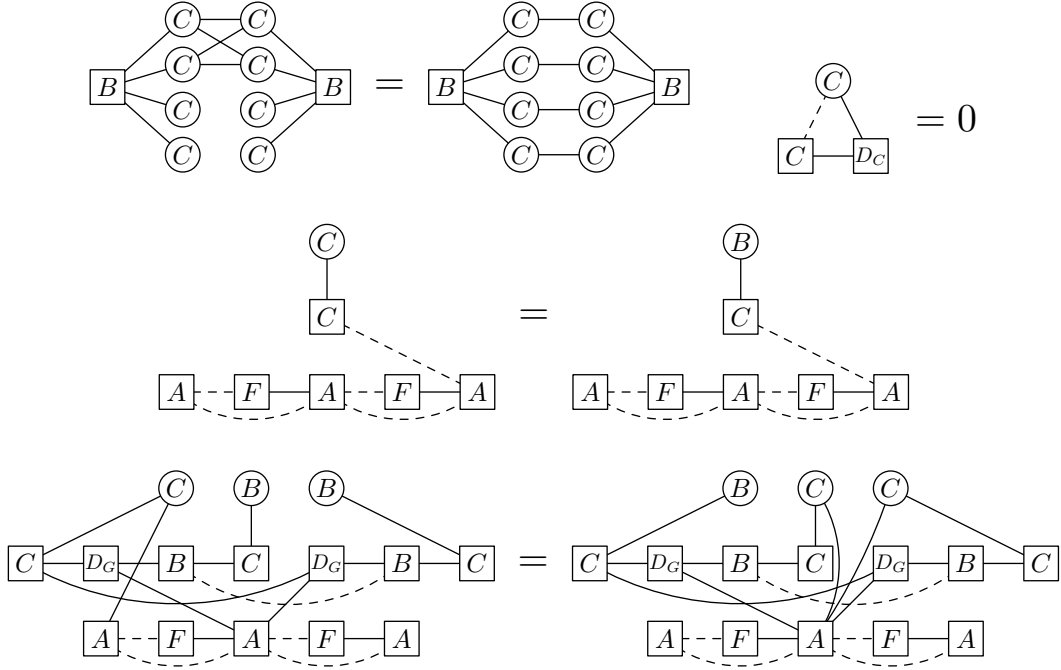

\begin{center}
\epsfbox{exturing-60.mps}
\hskip 10mm
\epsfbox{exturing-61.mps}
\vskip 6mm
\epsfbox{exturing-62.mps}
\vskip 6mm
\epsfbox{exturing-63.mps}
\end{center}
\caption{Decorated constraints forcing the specific structure of the tile $C\times C$.}
\label{fig-CC}
\end{figure}

In the previous subsection, we showed that
$W(x,y)=0$ for almost every $(x,y)\in\eta_C^{-1}([0,1/3))\times\eta_C^{-1}([1/3,2/3))$
such that $\coord{3\eta_C(x)}\not=\coord{3\eta_C(y)}$, and
we also showed that $W(x,y)=0$ for almost every $x\in\eta_C^{-1}([2/3,1))$ and $y\in C$.
The first constraint in Figure~\ref{fig-CC} is analogous to the first constraint in Figure~\ref{fig-CDG} and
it implies that there exist $\alpha^{11}_i$, $\alpha^{12}_i$ and $\alpha^{22}_i$, $i\in\NN$, such that
the following holds for almost every $(x,y)\in C\times C$:
$$W(x,y)=\left\{\begin{array}{cl}
         \alpha^{11}_i & \mbox{if $\lfloor 3\eta_C(x)\rfloor=0$, $\lfloor 3\eta_C(y)\rfloor=0$ and $i=\coord{3\eta_C(x)}=\coord{3\eta_C(y)}$,} \\
         \alpha^{12}_i & \mbox{if $\lfloor 3\eta_C(x)\rfloor=0$, $\lfloor 3\eta_C(y)\rfloor=1$ and $i=\coord{3\eta_C(x)}=\coord{3\eta_C(y)}$,} \\
         \alpha^{12}_i & \mbox{if $\lfloor 3\eta_C(x)\rfloor=1$, $\lfloor 3\eta_C(y)\rfloor=0$ and $i=\coord{3\eta_C(x)}=\coord{3\eta_C(y)}$,} \\
         \alpha^{22}_i & \mbox{if $\lfloor 3\eta_C(x)\rfloor=1$, $\lfloor 3\eta_C(y)\rfloor=1$ and $i=\coord{3\eta_C(x)}=\coord{3\eta_C(y)}$,} \\
	 0 & \mbox{otherwise.}
         \end{array}\right.$$
The second constraint on the first line in Figure~\ref{fig-CC} yields that $\alpha^{11}_1=1$.
The constraint on the second line in the figure implies that for every $i\in\NN$
$$\frac{\alpha^{11}_i+\alpha^{12}_i}{3\cdot 2^i}=\frac{1}{3\cdot 2^i}\qquad\mbox{and}\qquad\frac{\alpha^{12}_i+\alpha^{22}_i}{3\cdot 2^i}=\frac{1}{3\cdot 2^i}.$$
It follows that $\alpha^{11}_i=\alpha^{22}_i=1-\alpha^{12}_i$.
For $i\in\NN$, set $z_i=\alpha^{11}_{j}$ for $j$ with $M_j=\{i\}$.

We now analyze the last constraint depicted in Figure~\ref{fig-CC}.
For almost every choice of $C$-roots $c_1$, $c_2$ and $c_3$, the adjacencies to the middle $A$-root imply that
$\eta_C(c_1)\in [0,1/3)$, $\eta_C(c_2)\in [1/3,2/3)$ and $\eta_C(c_3)\in [1/3,2/3)$.
The structure of the tile $C\times D_G$ implies that
$\coord{3\eta_C(c_1)}\ge 2$ and that
$M_{\coord{3\eta_C(c_2)}}$ is either $\{\min M_{\coord{3\eta_C(c_1)}}\}$ or
$M_{\coord{3\eta_C(c_1)}}\setminus \{\min M_{\coord{3\eta_C(c_1)}}\}$.
Similarly,
$M_{\coord{3\eta_C(c_3)}}$ is either $\{\min M_{\coord{3\eta_C(c_1)}}\}$ or
$M_{\coord{3\eta_C(c_1)}}\setminus \{\min M_{\coord{3\eta_C(c_1)}}\}$.
Since the two $B$-roots are not adjacent, it follows that $M_{\coord{3\eta_C(c_2)}}\not=M_{\coord{3\eta_C(c_3)}}$.
By the symmetry of the graphs in the constraint, we can assume that
$M_{\coord{3\eta_C(c_2)}}=\{\min M_{\coord{3\eta_C(c_1)}}\}$ and
$M_{\coord{3\eta_C(c_3)}}=M_{\coord{3\eta_C(c_1)}}\setminus \{\min M_{\coord{3\eta_C(c_1)}}\}$ further in the analysis.

For such a choice of the $C$-roots $c_1$, $c_2$ and $c_3$,
the constraint implies that
$$\frac{\alpha^{11}_{\coord{3\eta_C(c_1)}}}{27\cdot 2^{\coord{3\eta_C(c_1)}+\coord{3\eta_C(c_2)}+\coord{3\eta_C(c_3)}}} = 
  \frac{\alpha^{22}_{\coord{3\eta_C(c_2)}}\alpha^{22}_{\coord{3\eta_C(c_3)}}}{27\cdot 2^{\coord{3\eta_C(c_1)}+\coord{3\eta_C(c_2)}+\coord{3\eta_C(c_3)}}}\,\mbox{.}
$$
Since $\alpha^{11}_i=\alpha^{22}_i$ for each $i \in \NN$, this yields that
$\alpha^{11}_{\coord{3\eta_C(c_1)}} = \alpha^{11}_{\coord{3\eta_C(c_2)}} \alpha^{11}_{\coord{3\eta_C(c_3)}}$.

Let $k=\coord{3\eta_C(c_1)}$.
If $|M_k|=1$, then $\coord{3\eta_C(c_2)}=k$ and $\coord{3\eta_C(c_3)}=1$.
In this case, the constraint always holds.
If $|M_k|>1$, then the constraint implies that
$\alpha^{11}_k=\alpha^{11}_{k'}\cdot\alpha^{11}_{k''}$
where $k'$ and $k''$ are such that $M_{k'}=\{\min M_k\}$ and $M_{k''}=M_k\setminus\{\min M_k\}$.
By induction on the size of $M_k$, we get that $\alpha^{11}_k=z^{M_k}$ for every $k\in\NN$.
For our choice of $z_i$, $i\in\NN$,
it follows that $W(x,y)=W_P(g(x),g(y))$ for almost all $(x,y)\in C\times C$.
Note that the values of $z_i$, $i\in\NN$, were uniquely determined by the structure of the graphon $W$.

\begin{figure}
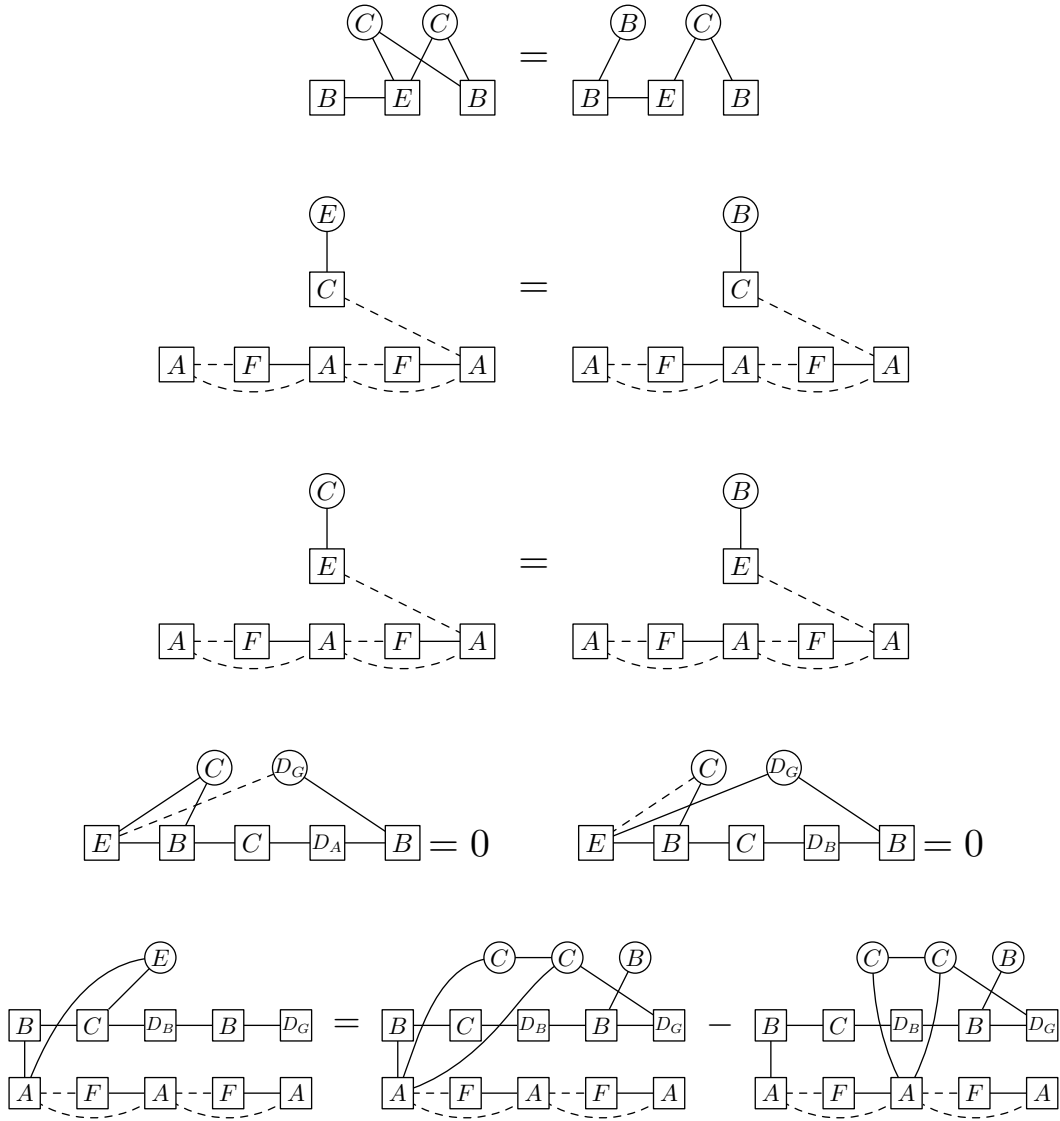

\begin{center}
\epsfbox{exturing-64.mps}
\vskip 10mm
\epsfbox{exturing-65.mps}
\vskip 10mm
\epsfbox{exturing-66.mps}
\vskip 10mm
\epsfbox{exturing-67.mps}
\hskip 10mm
\epsfbox{exturing-68.mps}
\vskip 10mm
\epsfxsize 140mm
\epsfbox{exturing-69.mps}
\end{center}
\caption{Decorated constraints forcing the specific structure of the tile $C\times E$.}
\label{fig-CE}
\end{figure}

We now focus on the tile $C\times E$ and consider the constraints depicted in Figure~\ref{fig-CE}.
In Subsection~\ref{sub-poly}, we showed that
$W(x,y)=0$ for almost every $(x,y)\in C\times E$ except for those $(x,y)$ such that
$\eta_C(x)\in [0,2/3)$ and $y\in J_{1,\coord{3\eta_C(x)}}\cup J_{2,\coord{3\eta_C(x)}}$.
Let us consider the first constraint depicted in the figure.
For almost every choice of $E$-root $y$ and right $B$-root $x$,
both sides of the constraint are zero unless $\eta_B(x)\in [0,2/3)$ and $y\in J_{1,\coord{3\eta_B(x)}}\cup J_{2,\coord{3\eta_B(x)}}$.
If $\eta_B(x)\in [0,1/3)$ and $y\in J_{1,\coord{3\eta_B(x)}}\cup J_{2,\coord{3\eta_B(x)}}$,
then the constraint implies that
$$\left(\deg^{\eta_C^{-1}([0,1/3))}(y)\right)^2=2^{-\coord{3\eta_B(x)}}\deg^{\eta_C^{-1}([0,1/3))}(y)/3\,\mbox{;}$$
if $\eta_B(x)\in [1/3,2/3)$ and $y\in J_{1,\coord{3\eta_B(x)}}\cup J_{2,\coord{3\eta_B(x)}}$,
then the constraint implies that
$$\left(\deg^{\eta_C^{-1}([1/3,2/3))}(y)\right)^2=2^{-\coord{3\eta_B(x)}}\deg^{\eta_C^{-1}([1/3,2/3))}(y)/3\,\mbox{.}$$
Hence, there exist $J_{11,i}\subseteq J_{1,i}$, $J_{12,i}\subseteq J_{1,i}$,
$J_{21,i}\subseteq J_{2,i}$ and $J_{22,i}\subseteq J_{2,i}$ such that
$W(x,y)=1$ for almost every $(x,y)\in C\times E$ such that
either $\eta_C(x)\in [0,1/3)$ and $y\in J_{11,\coord{3\eta_C(x)}}\cup J_{21,\coord{3\eta_C(x)}}$ or
$\eta_C(x)\in [1/3,2/3)$ and $y\in J_{12,\coord{3\eta_C(x)}}\cup J_{22,\coord{3\eta_C(x)}}$, and
$W(x,y)=0$ for almost all other $(x,y)\in C\times E$.

The second constraint depicted in Figure~\ref{fig-CE} yields that
$\deg^E(x)=\deg^B(x)=2^{-\coord{\eta_C(x)}}/3$
for almost every $x\in C$ with $\eta_C(x)\in [0,2/3)$.
Hence, $|J_{11,i}|+|J_{21,i}|=2^{-i}/3$ and $|J_{12,i}|+|J_{22,i}|=2^{-i}/3$ for every $i\in\NN$.
The third constraint in Figure~\ref{fig-CE} yields that
$\deg^C(y)=\deg^B(y)=2^{-\coord{\eta_E(y)}}/3$
for almost every $y\in E$ with $\eta_E(y)\in [0,2/3)$.
This implies that we can assume that
the sets $J_{11,i}$ and $J_{12,i}$ are disjoint,
the sets $J_{21,i}$ and $J_{22,i}$ are disjoint, and
$J_{11,i}\cup J_{12,i}=J_{1,i}$ and $J_{21,i}\cup J_{22,i}=J_{2,i}$ for every $i\in\NN$.

The left constraint on the fourth line implies that almost every $y\in J_{11,i}$ satisfies that $y\in J^u_{1,i}$, $i\in\NN$, and
the right constraint implies that almost every $y\in J^l_{1,i}$ satisfies that $y\in J_{11,i}$, $i\in\NN$
(the sets $J^l_{1,i}$ and $J^u_{1,i}$ were defined in Subsection~\ref{sub-poly}).
It follows that $J^l_{1,i}\sqsubseteq J_{11,i}\sqsubseteq J^u_{1,i}$ for all $i\in\NN$,
which in particular implies that $|J^l_{1,i}|\le |J_{11,i}|\le |J^u_{1,i}|$,
i.e., $3\cdot 2^i\cdot |J_{11,i}|\in [l_i,u_i]$.
Because of the last term in \eqref{eq-fE},
we get that $\eta_E(y)\le\eta_E(y')$ for almost all $y\in J_{11,i}$ and $y'\in J_{12,i}$.
Similarly, we get that $\eta_E(y)\le\eta_E(y')$ for almost all $y\in J_{21,i}$ and $y'\in J_{22,i}$.
To show that $W(x,y)=W_P(g(x),g(y))$ for almost all $(x,y)\in C\times E$,
it remains to establish that $|J_{11,i}|=2^{-i}p_i(z)/3$ for every $i\in\NN$.
To do so, we analyze the last constraint in Figure~\ref{fig-CE}.

Almost every choice of the $C$-root $x$ and the $D_G$-root $x'$ in the last constraint in Figure~\ref{fig-CE}
satisfies that $\eta_C(x)\in [0,1/3)$, $\eta_{D_G}(x')\in [2/3,1)$ and $\coord{3\eta_C(x)}=\coord{3\eta_{D_G}(x')}$.
Fix such a choice of roots and let $i=\coord{3\eta_C(x)}=\coord{3\eta_{D_G}(x')}$.
Recalling that $C \times \eta_{D_G}^{-1}([2/3,1))$ encodes the coefficients $\pi_{i,j}^+,\pi_{i,j}^-$ of the polynomials,
we can see that the constraint implies that $|J_{11,i}|$ is equal to
$$\sum_{k\in\NN}9\cdot 2^{2k}\pi^+_{i,k}\cdot\frac{1}{3\cdot 2^k}\cdot\frac{z^{M_k}}{3\cdot 2^k}\cdot\frac{1}{3\cdot 2^i}- 
  \sum_{k\in\NN}9\cdot 2^{2k}\pi^-_{i,k}\cdot\frac{1}{3\cdot 2^k}\cdot\frac{z^{M_k}}{3\cdot 2^k}\cdot\frac{1}{3\cdot 2^i}=\frac{p_i(z)}{3\cdot 2^i}\,\mbox{.}$$
It follows that $|J_{11,i}|=\frac{p_i(z)}{3\cdot 2^i}$ for every $i\in\NN$.
Hence, we can conclude that $W(x,y)=W_P(g(x),g(y))$ for almost all $(x,y)\in C\times E$.
Since we have already established that $3\cdot 2^i\cdot |J_{11,i}|\in [l_i,u_i]$, we also obtain that $p_i(z)\in [l_i,u_i]$.

\subsection{Proof of Theorem~\ref{thm-ff} -- cleaning up}
\label{sub-clean}

In Subsections~\ref{sub-setting}--\ref{sub-var},
we showed that $W(x,y)=W_P(g(x),g(y))$ for almost all $(x,y)\in (A\cup B\cup C\cup D_A\cup\ldots\cup D_G\cup E\cup F)^2$.
We now focus on the remaining tiles and
consider the decorated constraints depicted in Figures~\ref{fig-Q} and~\ref{fig-R}.

\begin{figure}
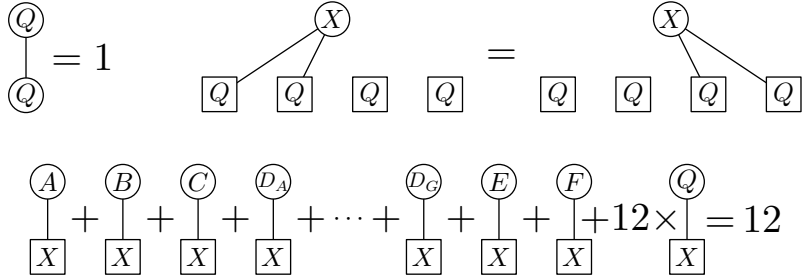

\begin{center}
\epsfbox{exturing-71.mps}
\hskip 10mm
\epsfbox{exturing-74.mps}
\vskip 6mm
\epsfbox{exturing-70.mps}
\end{center}
\caption{Decorated constraints forcing the structure of the tiles $X\times Q$ and $Q\times Q$,
         where $X\in\{A,B,C,D_A,\ldots,D_G,E,F\}$.}
\label{fig-Q}
\end{figure}

The first constraint in Figure~\ref{fig-Q} implies that $W(x,y)=1$ for almost every $(x,y)\in Q\times Q$.
Next fix $X\in\{A,B,C,D_A,\ldots,D_G,E,F\}$ and consider the second constraint on the first line in Figure~\ref{fig-Q}.
The constraint yields that the following holds for almost all $x,x',x'',x'''\in Q$:
$$\int_X W(x,y)W(x',y)\dd y=\int_X W(x'',y)W(x''',y)\dd y\,\mbox{,}$$
which implies by Lemma~\ref{lm-typ-pairs} that there exists $K\in\RR$ such that
$$\int_X W(x,y)W(x',y)\dd y=K\qquad\mbox{and}\qquad\int_X W(x,y)^2\dd y=K$$
for almost all $x,x',\in Q$.
However, this is only possible if there exists a measurable function $\xi:X\to [0,1]$ such that
$W(x,y)=\xi(y)$ for almost all $(x,y)\in Q\times X$.
In particular, $\deg^Q(y) = \xi(y)$ for almost all $y\in X$. 
The constraint on the second line implies that
$$\xi(y)=\frac{12-\deg^A(y)-\deg^B(y)-\ldots-\deg^E(y)-\deg^F(y)}{12}$$
%$$\xi(y)=\frac{6-\deg^A(y)-\deg^B(y)-\deg^C(y)-\deg^{D_A}(y)-\ldots-\deg^{D_G}(y)-\deg^E(y)-\deg^F(y)}{7}$$
for almost all $y\in X$.
Since $W(x,y)=W_P(g(x),g(y))$ for almost all $(x,y)\in (A\cup B\cup C\cup D_A\cup\ldots\cup D_G\cup E\cup F)^2$,
we get that $\xi(y)=W_P(x,g(y))$ for every $x\in Q$ and almost every $y\in X$.
We conclude that $W(x,y)=W_P(g(x),g(y))$ for almost all $(x,y)\in X\times Q$, where we recall that $X \in \{A,B,C,D_A,\ldots,D_G,E,F\}$.

\begin{figure}
\begin{center}
\epsfbox{exturing-72.mps}
\hskip 10mm
\epsfbox{exturing-73.mps}
\end{center}
\caption{Decorated constraints forcing the structure of the tiles $X\times R$,
         where $(X,k)$ is one of $(A,1)$, $(B,2)$, $(C,3)$, $(D_A,4)$, \dots, $(D_G,10)$, $(E,11)$, $(F,12)$, $(Q,25)$ and $(R,0)$.}
\label{fig-R}
\end{figure}

Fix $(X,k)$ to be one of the following pairs:
$(A,1)$, $(B,2)$, $(C,3)$, $(D_A,4)$, \dots, $(D_G,10)$, $(E,11)$, $(F,12)$, $(Q,25)$ and $(R,0)$.
Consider the decorated constraints depicted in Figure~\ref{fig-R}.
The two constraints imply that
$$\frac{1}{|R|}\int_R W(x,y)\dd y=\frac{k}{25}\qquad\mbox{and}\qquad\frac{1}{|R|}\int_R W(x,y)W(x',y)\dd y=\frac{k^2}{625}$$
for almost all $x,x'\in X$.
The second equality yields by Lemma~\ref{lm-typ-pairs} that
$$\frac{1}{|R|}\int_R W(x,y)^2\dd y=\frac{k^2}{625}$$
for almost all $x\in X$.
Hence, the Cauchy-Schwarz Inequality implies that it holds for almost all $x\in X$
that $W(x,y)$ as a function of $y\in R$ is constant almost everywhere on $R$.
It follows that $W(x,y)=k/25$ for almost all $(x,y)\in X\times R$.
We can now conclude that $W(x,y)=W_P(g(x),g(y))$ for almost all $(x,y)\in [0,1)\times [0,1)$,
which finishes the proof of Theorem~\ref{thm-ff}.

\section{Dependence on parameters}
\label{sec-analyze}

In this section,
we analyze the dependence of the structure of the graphon $W_P(z)$ and
the densities of graphs in this graphon on a bounding sequence $P$ and a vector $z\in [0,1]^\NN$.
We start with an additional definition:
the $L_1$-distance between two bounding sequences $P$ and $P'$ is equal to
\[\sup_{z\in [0,1]^\NN}\int_{[0,1)^2}\left|W_P(z)(x,y)-W_{P'}(z)(x,y)\right|\dd x\dd y.\]
We remark that we view the graphons in the definition above purely as functions from $[0,1)^2$ to $\RR$,
i.e., we do not consider rearrangements of their underlying spaces as
in the standard definition of the $L_1$-distance between graphons.

\begin{lemma}
\label{lm-add-P}
Let $P$ be a bounding sequence.
For every $\eps>0$, there exists an integer $k_P$ such that
the $L_1$-distance between $P$ and any $k_P$-strengthening of $P$ is at most $\eps$.
\end{lemma}

\begin{proof}
Fix $z\in [0,1]^\NN$ and a strengthening $P'$ of the bounding sequence $P$. We will again just write $W_P$ and $W_{P'}$ instead of $W_P(z)$ and $W_{P'}(z)$, respectively.
Our goal is to show that if $P'$ is a $k_P$-strengthening of $P$, for a sufficiently large integer $k_P$,
which does not depend on $P'$ and $z$,
then the $L_1$-distance between $W_P$ and $W_{P'}$, viewed as functions from $[0,1)^2$ to $\RR$, is at most $\eps$.
Note that the graphons $W_P$ and $W_{P'}$ can differ only in the following tiles:
$(D_A\cup\ldots\cup D_G)^2$, $C\times E$, $C\times D_G$, $D_G\times E$ and $X\times Q$, where $X\in\{C,D_A,\ldots,D_G,E\}$.

Let $k_0$ be the integer from Proposition~\ref{prop-change} applied for $\eps/6$.
Furthermore, let $W_F$ be the graphon for the bounding sequence $P$ from the definition of $W_P$, and
let $W'_F$ be the graphon for the bounding sequence $P'$.
For a dyadic square $S$, let $d_{W_F}(S)$ be its density in $W_F$,
i.e.,
\[d_{W_F}(S)=\frac{\int_S W_F(x,y)\dd x\dd y}{|S|}.\]
Finally, let $\eps_F$ be equal to the minimum of the values
\[\frac{\lfloor 2^{k_0}d_{W_F}(S)\rfloor+1}{2^{k_0}}-d_{W_F}(S),\]
taken over all dyadic squares $S$ of sizes $2^0,2^{-1},\ldots,2^{-k_0}$.
Note that $\eps_F$ is positive.
Set $k_1$ to be an integer such that $2^{-k_1}<2^{-k_0}\cdot\eps_F$.
Note that the choice of $k_1$ depends only on the bounding sequence $P$ and the value of $\eps$.
If the bounding sequences $P$ and $P'$ agree on the first $k_1$ elements,
then the definition of the graphons $W_F$ and $W_{F'}$ imply that
their $L_1$-distance is at most $2^{-k_1}$.
In addition, if $P'$ is a strengthening of $P$, then $W'_F(x,y)\ge W_F(x,y)$ for all $(x,y)\in [0,1)^2$.
Hence, if $P'$ is a $k_1$-strengthening of $P$,
then the density of each dyadic square of size at least $2^{-k_0}$ in $W'_F$
is at least the density of that square in $W_F$ and
the difference between the two densities is at most $\eps_F$.
In particular, the densities of the dyadic squares of $2^0,2^{-1},\ldots,2^{-k_0}$
in $W_F$ and $W'_F$ agree up to the first $k_0$ bits after the decimal point in the standard binary representation.
Consequently, the $L_1$-distance between the graphons $W_0$ and $W'_0$ from Theorem~\ref{thm-universal}
applied with $W_F$ and $W'_F$, respectively, is at most $\eps/6$.
We conclude that if $P'$ is a $k_1$-strengthening of $P$, then
$$\int_{(D_A\cup\ldots\cup D_G)^2}\left|W_P(x,y)-W_{P'}(x,y)\right|\dd x\dd y\le\frac{\eps}{6}\,\mbox{.}$$

Next, set $k_2$ to be an integer such that $2^{-k_2}\le\eps/6$.
Observe that if the bounding sequences $P$ and $P'$ agree on the first $k_2$ elements,
then the measure of the points in each of the tiles $C\times E$, $C\times D_G$ and $D_G\times E$
where $W_P$ and $W_{P'}$ differ is at most $2^{-k_2}$ of the total measure of the tile.
We set $k_P=\max\{k_1,k_2\}$ and conclude that if $P'$ is a $k_P$-strengthening of $P$,
then
\begin{equation}
\int_{(A\cup B\cup C\cup D_A\cup\ldots\cup D_G\cup E\cup F)^2}\left|W_P(x,y)-W_{P'}(x,y)\right|\dd x\dd y\le\frac{\eps}{3}\,\mbox{.}
\label{eq-addP1}
\end{equation}

Fix $x\in A\cup B\cup C\cup D_A\cup\ldots\cup D_G\cup E\cup F$.
The definition of the graphons $W_P$ and $W_{P'}$ yields that
$$\int_{A\cup B\cup C\cup D_A\cup\ldots\cup D_G\cup E\cup F\cup Q}W_P(x,y)-W_{P'}(x,y)\dd y=0.$$
Since the value of $W_P(x,y)$ is the same for all $y\in Q$ and the value of $W_{P'}(x,y)$ is also the same for all $y\in Q$,
it follows that
\begin{align}
\int\limits_Q \left|W_P(x,y)-W_{P'}(x,y)\right|&\dd y = \left|\int_Q W_P(x,y)-W_{P'}(x,y)\dd y\right| \nonumber\\ 
=&\left|\int\limits_{A\cup B\cup C\cup D_A\cup\ldots\cup D_G\cup E\cup F}\mkern-9mu W_P(x,y)-W_{P'}(x,y)\dd y \right| \nonumber\\
\le & \int\limits_{A\cup B\cup C\cup D_A\cup\ldots\cup D_G\cup E\cup F}\mkern-9mu \left|W_P(x,y)-W_{P'}(x,y)\right|\dd y.
\label{eq-addP2}
\end{align}
%
%\[
%\int\limits_Q \left|W_P(x,y)-W_{P'}(x,y)\right|\dd y = \left|\int\limits_{A\cup B\cup C\cup D_A\cup\ldots\cup D_G\cup E\cup F}\mkern-9mu W_P(x,y)-W_{P'}(x,y)\dd y \right|
%\]
%\begin{equation}
% \le  \int\limits_{A\cup B\cup C\cup D_A\cup\ldots\cup D_G\cup E\cup F}\mkern-9mu \left|W_P(x,y)-W_{P'}(x,y)\right|\dd y.
%\label{eq-addP2}
%\end{equation}
We derive from \eqref{eq-addP1} and \eqref{eq-addP2} that
\begin{equation}
\int_{(A\cup B\cup C\cup D_A\cup\ldots\cup D_G\cup E\cup F)\times Q}\left|W_P(x,y)-W_{P'}(x,y)\right|\dd x\dd y\le\frac{\eps}{3}\,\mbox{.}
\label{eq-addP3}
\end{equation}
The symmetric estimate holds for the integral over $Q\times (A\cup B\cup C\cup D_A\cup\ldots\cup D_G\cup E\cup F)$.
Since the graphons $W_P$ and $W_{P'}$ agree on the tile $Q\times Q$ and all tiles involving the part $R$,
we conclude using the estimates \eqref{eq-addP1} and \eqref{eq-addP3} that
their $L_1$-distance is at most $\eps$.
\end{proof}

We next analyze the dependence of the density of a graph $H$ in a graphon $W_P(z)$ on the vector $z$. 
We start by showing that this density is a countable sum of polynomials in $z$,
where the polynomials appearing in the sum have their coefficients sufficiently restricted.
In the statement of the next lemma,
we use the linear order on multisets of natural numbers defined in Section~\ref{sec-gensetup},
i.e., $M_i$ is the $i$-th multiset in this order.  

\begin{lemma}
\label{lm-densityinfinitesum}
For every graph $H$, there exists a countable set $\HH$,
a constant $c_H\in (0,1)$ and constants $\beta_{H',i}$, $H'\in\HH$ and $i\in\NN$, with the following properties.
For every bounding sequence $P=(p_j,l_j,u_j)_{j\in\NN}$, there exist polynomials
\begin{equation} \label{eq-poly-alpha}
q_{P,H'}(z)=\sum_{i\in\NN} \alpha_{P,H',i}\cdot z^{M_i}
\end{equation}
such that
\begin{equation}
\tau(H,W_P(z))=\sum_{H' \in \HH} q_{P,H'}(z)
\label{eq-densityinfinitesum}
\end{equation}
for every $z\in [0,1]^\NN$ that satisfies $p_j(z)\in [l_j,u_j]$ for all $j\in\NN$.
In addition, it holds that
$|\alpha_{P,H',i}| \le \beta_{H',i}$ for every $H'\in\HH$ and $i\in\NN$,
the sum $\sum_{H' \in \HH} \beta_{H',i}$ is finite for all $i\in\NN$, and
is at most $c_H^{\Sigma(M_i)}$ for all but finitely many $i\in\NN$.

Furthermore, for every $H'\in\HH$ and every real $\eps>0$,
there exist an integer $k_{P,H'}$ and a real $\eps_{P,H'}>0$ such that
if $P'$ is a bounding sequence that
agrees with $P$ on the first $k_{P,H'}$ elements and that
has $L_1$-distance at most $\eps_{P,H'}$ from $P$,
then the polynomials $q_{P,H}(z)$ and $q_{P',H}(z)$ are $\eps$-close.
\end{lemma}

\begin{proof}
Fix a graph $H$ and a bounding sequence $P$ for the proof of the lemma; let $v_1,\ldots,v_n$ be the vertices of $H$.
We start by defining the set $\HH$.
Let $C_{3i+j}\subseteq C$ be defined as follows:
$$C_{3i+j} = \iota_C\left( \left[\frac{j-2^{-i}}{3},\frac{j-2^{-(i+1)}}{3}\right) \right),$$
where $(i,j)\in\NN_0\times\{1,2,3\}$.
Note that the graphon $W_P(z)$ is constant on every set $C_k\times C_{k'}$, $k,k'\in\NN$, and
it is zero on the tile $C\times C$ outside the sets $C_{3k+1}\times C_{3k+1}$,
$C_{3k+1}\times C_{3k+2}$, $C_{3k+2}\times C_{3k+1}$ and $C_{3k+2}\times C_{3k+2}$, $k\in\NN$.
We also define $E_{3i+j}\subseteq E$ as
$$E_{3i+j} = \iota_E\left( \left[\frac{j-2^{-i}}{3},\frac{j-2^{-(i+1)}}{3}\right) \right)$$
for $(i,j)\in\NN_0\times\{1,2,3\}$ and
$$E_{3i+1,j} = \left\{\begin{array}{ll}
	\smallskip \iota_E\left( \left[\frac{1-2^{-i}}{3},\frac{1-2^{-i}+l_i 2^{-(i+1)}}{3}\right) \right) & \mbox{if $j=1$,} \\
	\smallskip \iota_E\left( \left[\frac{1-2^{-i}+l_i 2^{-(i+1)}}{3},\frac{1-2^{-i}+u_i 2^{-(i+1)}}{3}\right) \right) & \mbox{if $j=2$,} \\
	\iota_E\left( \left[\frac{1-2^{-i}+u_i 2^{-(i+1)}}{3},\frac{1-2^{-(i+1)}}{3}\right) \right) & \mbox{if $j=3$,} \\
	\end{array}\right.$$
for $(i,j)\in\NN\times\{1,2,3\}$.	
Finally, let 
$$D_{G,j}=\iota_{D_G}\left([(j-1)/3,j/3)\right)$$
for $j\in\{1,2,3\}$.
We define the set $\HH$ to be the set containing all copies of $H$
where each vertex is labeled with one of the following sets:
$A$, $B$, $C_k$, $D_A,\ldots,D_F$, $D_{G,1}$, $D_{G,2}$, $D_{G,3}$, $E_{k}$, $F$, $Q$ and $R$, $k\in\NN$.
	
Fix $H'\in\HH$ and let $X_i$ be the label of the vertex $v_i$ of $H'$. Form another graph $G$ on the vertices of $H$ as follows. For each vertex $v_i$, 
sample $x_i$ from $X_i$ uniformly and independently, and
join the vertices $v_i$ and $v_{i'}$ with probability equal to the value of the graphon $W_P(z)$ for $(x_i,x_{i'})$.
We next show that the probability that the graph $G$ is the graph $H$ (preserving vertices)
is a polynomial $p_{P,H'}(z)$ of $z$.

Let $V_{C}$ and $V_E$ consist of the vertices of $H'$ that are labeled with $C_k$ or $E_k$ for some $k\in\NN$, respectively, and
let $V_{CE}=V_C\cup V_E$.
Given two vertices of $H'$ such that at least one of them is not contained in $V_{CE}$, the probability that they 
are joined by an edge does not depend on $z\in [0,1]^\NN$ by Lemma~\ref{lm-dep-z}.
Let $s_{P,H'}$ be the probability that all such pairs of vertices have the same adjacencies in $G$ and $H$.
Furthermore,
let $r_{P,H'}(z)$ be the probability that $G$ and $H$ agree on $V_{CE}$ conditioned on the event that
the graphs $G$ and $H$ agree outside $V_{CE}$.
Clearly $p_{P,H'}(z)=s_{P,H'}\cdot r_{P,H'}(z)$.
Since for all $y\in [0,1)\setminus (C\cup E)$, the value of $W_P(z)$ for $(x,y)$ and $(x',y)$ is the same
if $x$ and $x'$ belong to the same set $C_k$, $E_{3k+1,1}$, $E_{3k+1,2}$, $E_{3k+1,3}$, $E_{3k+2}$ or $E_{3k+3}$, $k\in\NN$,
the event that $G$ and $H$ agree outside $V_{CE}$
does not restrict the values $x_i$ associated with vertices labeled with $C_k$, $E_{3k+2}$ or $E_{3k+3}$,
however,
it may restrict the values $x_i$ associated with vertices labeled with $E_{3k+1}$ to be (uniformly chosen)
from $E_{3k+1,1}$, $E_{3k+1,2}$, $E_{3k+1,3}$ or from a union of some of these three sets.

We next show that the function $r_{P,H'}(z)$ is a polynomial in $z$.
Two vertices of $H$ that are labeled with $C_k$ and $C_{k'}$, $k,k'\in\NN$,
are joined by an edge with probability equal to $0$, $z^M$ or $1-z^M$
for a finite multiset $M$ of positive integers;
the event that such a pair of vertices is joined by an edge is independent of the rest of the structure of $G$.
The probability that a vertex labeled with $E_{3k+1}$, $E_{3k+2}$ or $E_{3k+3}$, $k\in\NN$,
has the same adjacencies in $G$ and $H$ is equal to one of the following expressions:
$0$, $1$,
 $p_j(z)$, $1-p_j(z)$, $p_j(z)/u_j$, $(u_j-p_j(z))/u_j)$, $(p_j(z)-l_j)/(1-l_j)$, $(1-p_j(z))/(1-l_j)$,
$(p_j(z)-l_j)/(u_j-l_j)$ and $(u_j-p_j(z))/(u_j-l_j)$,
$j\in\NN$; here, we use that $p_j(z)\in [l_j,u_j]$ for every $j\in\NN$.
Moreover, these events are independent of each other
because the graphon $W_P(z)$ is zero on the tile $E\times E$, and
they are also independent of the adjacencies between the vertices labeled with $C_k$, $k\in\NN$,
because for every $y\in E$,
the value of $W_P(z)$ for $(x,y)$ and $(x',y)$ is the same if $x$ and $x'$ are from the same $C_k$, $k\in\NN$.
Hence, $r_{P,H'}(z)$ is a product of the following expressions:
$0$, $1$, $z^M$, $1-z^M$, 
 $p_j(z)$, $1-p_j(z)$, $p_j(z)/u_j$, $(u_j-p_j(z))/u_j)$, $(p_j(z)-l_j)/(1-l_j)$, $(1-p_j(z))/(1-l_j)$,
$(p_j(z)-l_j)/(u_j-l_j)$ and $(u_j-p_j(z))/(u_j-l_j)$,
where $M$ is a finite multiset of positive integers and $j\in\NN$.
Moreover, for each expression that appears with a denominator of $1-l_j$, $u_j$ or $u_j-l_j$,
there is a corresponding term in the product defining the constant $s_{P,H'}$.
We conclude that $p_{P,H'}(z)$ is a product of a constant from $[0,1]$ and terms equal to one of
the following expressions: $0$, $1$, $z^M$, $1-z^M$, $p_j(z)$, $1-p_j(z)$, $u_j-p_j(z)$ and $p_j(z)-l_j$,
where $M$ is a finite multiset of positive integers and $j\in\NN$.
Set
\begin{equation}
q_{P,H'}(z)=p_{P,H'}(z)\cdot\prod_{i=1}^n|X_i|\,\mbox{.}
\label{eq-approx}
\end{equation}
The definition of $\tau(H,W_P)$ now implies that
\[\tau(H,W_P(z))=\sum_{H'\in\HH}q_{P,H'}(z)\]
for every $z\in [0,1]^\NN$ that satisfies $p_j(z)\in [l_j,u_j]$ for all $j\in\NN$.

We next establish the existence of the constants $\beta_{H',i}$.
Fix a multiset $M_i$ of positive integers.
We are going to trace how the monomial $z^{M_i}$ can appear in the product defining the polynomial $p_{P,H'}(z)$ for $H'\in\HH$.

Recall that $V_C$ and
$V_E$ are the sets vertices of $H'$ labeled with $C_k$, $k\in\NN$ and $E_k$, $k\in\NN$, respectively. In the analysis above, we expressed $p_{P,H'}(z)$ as a product of
terms corresponding to vertices of $V_E$ and
terms corresponding to pairs of vertices from $V_C$.
So, consider a partition $\parti$ of $M_i$ into $|V_E|+\binom{|V_C|}{2}\le n^2$ multisets
$M^{v_j}$, $v_j\in V_E$, and $M^{v_jv_{j'}}$, $\{v_j,v_{j'}\}\subseteq V_C$.
Notice that for $k>0$, a pair of vertices $v_j$ and $v_{j'}$ from $V_C$ can contribute to the product defining $p_{P,H'}(z)$
by the monomial $z^{M_k}$ only if both $v_j$ and $v_{j'}$ are labeled by $C_{3k+1}$ or $C_{3k+2}$.
So, we say that a partition $\parti$ is \emph{consistent} with $H'$
if for each pair $v_j,v_j' \in V_C$ either $M^{v_jv_j'}$ is the empty set or $v_j,v_{j'} \in C_{3k+1} \cup C_{3k+2}$ and $M^{v_jv_j'}=M_k$.
Let $\partis_{H',i}$ be the set of all partitions of $M_i$ consistent with $H'$.

Observe that
if a vertex $v_j\in V_E$ contributes to the product defining $p_{P,H'}(z)$ by the monomial $z^{M_k}$,
then the corresponding coefficient in the product is at most $2^{-2^k}$.
In particular, it is at most $2^{-2^{\Sigma(M_k)}}$ by Observation~\ref{obs-monomialsadditive}.
We next set
\[\beta_{H',i}=\prod_{j=1}^n |X_j|\sum_{\parti \in \partis_{H',i}} \prod_{v_{j'} \in V_E} 2^{-2^{\Sigma(M^{v_{j'}})}}.\] 
The analysis above implies that $|\alpha_{P,H',i}| \le \beta_{H',i}$.

We next prove the existence of the constant $c_H$ and that the sum $\sum_{H' \in \HH} \beta_{H',i}$ is finite for all $i\in\NN$.
Consider a partition $\parti \in \partis_{H',i}$ and let
\[S=\sum_{v_j\in V_E}\Sigma(M^{v_j}).\]
Note that there must exist a pair of vertices $v_j,v_{j'}\in V_C$ with $M^{v_jv_{j'}}=M_k$ such that
$k\ge\Sigma(M_k)\ge(\Sigma(M_i)-S)/n^2$ (the first inequality holds by Observation~\ref{obs-monomialsadditive}).
So, every partition of $M_i$ consistent with $H'\in\HH$ can be constructed as follows:
we first choose an integer $S$ between $0$ and $\Sigma(M_i)$.
We then choose which pairs of vertices from $V_C$ correspond to the empty set and
which to the unique multiset $M_k$ such that both vertices in the pair are labeled by $C_{3k+1}$ or $C_{3k+2}$.
This can be done in at most $2^{n^2}$ ways.
Then,
the remaining at most $S$ elements of $M_i$ are distributed among the sets $M^{v_j}$, $v_j\in V_E$, in at most $n^{S}$ ways.

We are now ready to estimate the sum of $\beta_{H',i}$.
This will be done by considering each $H'\in\HH$ together with all partitions of $M_i$ that are consistent with $H'$.
First, we choose which vertices of $H$ will belong to $V_C$ and to $V_E$; this can be done in $3^n$ ways.
Next, if $V_C$ is non-empty, we choose $i_C$ such that the vertex $v_{i_C}\in V_C$ is labeled with $C_{3k+1}$ or $C_{3k+2}$
for the largest value of $k$ among the vertices from $V_C$; the index $i_C$ can be chosen in at most $n$ ways.
Then, we choose the value of $S$ and
consider each $H'\in\HH$ such that the vertex $v_{i_C}$ is labeled with $C_{3k+1}$ or $C_{3k+2}$ for some $k\ge (\Sigma(M_i)-S)/n^2$.
Finally, we choose a partition of $M_i$ that is consistent with $H'$.
We established above that this can be chosen in at most $2^{n^2}\cdot n^S$ ways.
Observe that if $V_E$ is non-empty, then there must be a vertex $v_j \in V_E$ such that $\Sigma(M^{v_j}) \ge S/n$,
in particular, the last product in the definition of $\beta_{H',i}$ is at most $2^{-2^{S/n}}$.
If $V_E$ is empty, which implies $S=0$,
the bound $2^{-2^{S/n}}$ is off by a factor of two, however,
this is compensated for by an overestimate in the term $2^{n^2}$, so the bound $2^{n^2} n^S 2^{-2^{S/n}}$ is still valid.
We conclude that
\begin{align}
\sum_{H' \in \HH} \beta_{H',i} \le& \sum_{H' \in \HH}  \prod_{j=1}^n |X_j| \sum_{S=0}^{\Sigma(M_i)} 2^{n^2} n^S 2^{-2^{S/n}} \nonumber\\
\le &n\cdot 3^n \sum_{S=0}^{\Sigma(M_i)}\sum_{k \ge \frac{\Sigma(M_i)-S}{n^2}}\quad
  \sum_{\substack{H'\in\HH\\ X_{i_C}\in\{C_{3k+1},C_{3k+2}\}}}\quad
  2^{n^2}\cdot n^S\cdot 2^{-2^{S/n}}
  \prod_{j=1}^n |X_j|\nonumber\\
= & n\cdot 3^n\cdot 2^{n^2} \sum_{S=0}^{\Sigma(M_i)} n^S \cdot 2^{-2^{S/n}} \sum_{k \ge \frac{\Sigma(M_i)-S}{n^2}}
  \sum_{\substack{H'\in\HH\\ X_{i_C}\in\{C_{3k+1},C_{3k+2}\}}} \prod_{j=1}^n |X_j| \nonumber\\
= & n\cdot 3^n\cdot 2^{n^2} \sum_{S=0}^{\Sigma(M_i)} n^S \cdot 2^{-2^{S/n}} \sum_{k \ge \frac{\Sigma(M_i)-S}{n^2}} 2|C_{3k+1}|
\sum_{\substack{H'\in\HH\\ X_{i_C}\in\{C_{3k+1},C_{3k+2}\}}} \prod_{j \ne i_C} |X_j| \nonumber\\
\le & n\cdot 3^n\cdot 2^{n^2} \sum_{S=0}^{\Sigma(M_i)}n^S\cdot 2^{-2^{S/n}}\cdot 2^{-(\Sigma(M_i)-S)/n^2}\nonumber\\
\le & n\cdot 3^n\cdot 2^{n^2} \sum_{S=0}^{\Sigma(M_i)}2^{nS - 2^{S/n} - \Sigma(M_i)/n^2}.\label{eq-sumS}
\end{align}
In particular, the sum is finite for every $i\in\NN$.

We next bound the exponents in the summands in \eqref{eq-sumS}.
If $S\le \Sigma(M_i)/(2n^3)$,
then we have:
\[ nS -2^{S/n} -\frac{\Sigma(M_i)}{n^2}
   \le nS-\frac{\Sigma(M_i)}{n^2}
   \le -\frac{\Sigma(M_i)}{2n^2}.\]
If $S\ge \Sigma(M_i)/(2n^3)$ and $2nS\le 2^{S/n}$,
then we obtain the following:
\[ nS-2^{S/n}-\frac{\Sigma(M_i)}{n^2}
   \le nS-2^{S/n} \le -nS \le -\frac{\Sigma(M_i)}{2n^2}.\]
Since $n$ is the number of vertices of $H$, i.e., $n$ is fixed,
there exists $\Sigma_0$ such that
if $\Sigma(M_i)\ge\Sigma_0$, then any $S\ge \Sigma(M_i)/(2n^3) \ge \Sigma_0/(2n^3)$ satisfies $2nS\le 2^{S/n}$.
We conclude that there exists $c_H\in (0,1)$ such that
\[n\cdot 3^n\cdot 2^{n^2}\cdot \sum_{S=0}^{\Sigma(M_i)}2^{nS -2^{S/n} -\Sigma(M_i)/n^2}
  \le n\cdot 3^n\cdot 2^{n^2}\cdot (\Sigma(M_i)+1)\cdot 2^{-\frac{\Sigma(M_i)}{2n^2}}\le c_H^{\Sigma(M_i)}\]
if $\Sigma(M_i)$ is sufficiently large.

It now remains to establish the last part of the statement of the lemma.
Fix $H'\in\HH$ and let $\eps>0$ be given.
Let $k_{P,H'}$ be the largest index $k$ such that a vertex of $H'$ is labeled with $E_{3k+1}$ or $E_{3k+2}$.
If two bounding sequences $P$ and $P'$ agree on the first $k_{P,H'}$ elements,
then the polynomials $r_{P,H'}(z)$ and $r_{P',H'}(z)$ are the same.
Let $D$ be such that the values of $r_{P,H'}(z)$ and all its first partial derivatives on $[0,1]^\NN$ belong to $[-D,D]$.
Next, choose $\eps_{P,H'}$ small enough such that
if the $L_1$-distance of the graphons $W_P(z)$ and $W_{P'}(z)$ is at most $\eps_{P,H'}$,
then the values $s_{P,H'}$ and $s_{P',H'}$ differ by at most $\eps/D$ (regardless of the choice of $z\in [0,1]^\NN$).
For this choice of $k_{P,H'}$ and $\eps_{P,H'}$,
the polynomials $q_{P,H'}(z)$ and $q_{P',H'}(z)$ are $\eps$-close
if the $L_1$-distance of $P$ and $P'$ is at most $\eps_{P,H'}$ and $P$ and $P'$ agree on the first $k_{P,H'}$ elements.
This finishes the proof of the lemma.
\end{proof}

We emphasize that $\tau(H,W_P(z))$ and the right side of \eqref{eq-densityinfinitesum}
are not necessarily equal for all $z\in [0,1]^\NN$,
in particular, for those $z\in [0,1]^\NN$ that do not satisfy the constraints implied by the bounding sequence $P$.

We next show that the density of each graph $H$ in $W_P(z)$ is equal to a totally analytic function of $z$
for those $z\in [0,1]^\NN$ that satisfy the constraints implied by the bounding sequence $P$.
We start with the following simple observation.

\begin{lemma}
\label{lm-expmonomboundfinitesum}
Let $M_i$ be the $i$-th multiset in the linear order defined in Section~\ref{sec-gensetup}.
For every real $c\in (0,1)$ and integer $k\in\NN_0$, the sum
\[\sum_{i=1}^\infty |M_i|^k c^{\Sigma(M_i)}\] 
is finite.
\end{lemma}

\begin{proof}
We group the summands into groups with the same value of $\Sigma(M_i)$.
The number of multisets $M_i$ with $\Sigma(M_i)=n$
is the number of ways that an integer $n$ can be expressed as a sum of positive integers (where the order does not matter).
This number is well-known to be at most $e^{10\sqrt{n}}$ if $n \ge n_0$ for some $n_0\in\NN$.
Since $|M_i|\le\Sigma(M_i)$,
the sum from the statement of the lemma is at most
\begin{equation}
\sum_{n=1}^{n_0-1}\sum_{M_i,\Sigma(M_i)=n}|M_i|^k c^{\Sigma(M_i)}+\sum_{n=n_0}^\infty e^{10\sqrt{n}}n^kc^{n}.
\label{eq-expmonomboundfinitesum}
\end{equation}
Since the first sum in \eqref{eq-expmonomboundfinitesum} consists of finitely many summands, 
the sum from the statement of the lemma is finite.
\end{proof}

We also need the following theorem,
which can be found, for example, in~\cite[Theorem~7.17]{bib-rudin-book}.
\begin{theorem}
\label{thm-rudin}
Let $I=[a,b]$ and $f_n:I \to \RR$ be a sequence of functions differentiable on $I$.
If $f_n'$ converges uniformly to a function $g$ on $I$ and
there exists a point $x_0$ in the interior of $I$ such that $f_n(x_0)$ converges,
then the functions $f_n$ converge uniformly to a differentiable function $f:I \to \RR$ and $f'=g$.
\end{theorem}

We are now ready to state a lemma that
will guarantee that the functions $t_{P,H}(z)$ defined further are totally analytic.

\begin{lemma}
\label{lm-monomialboundanalytic}
For $i \in \NN$, let $M_i$ be the $i$-th multiset in the linear order defined in Section~\ref{sec-gensetup}, and
let $\alpha_i$ be a real number.
Suppose that there exists $c\in (0,1)$ such that $|\alpha_i|\le c^{\Sigma(M_i)}$ for all but finitely many $i\in\NN$.
Then the function $t:[0,1]^\NN\to\RR$ defined as
\begin{equation}
t(z)=\sum_{i=1}^\infty \alpha_i z^{M_i}\label{eq-tpowerseries}
\end{equation}
is well-defined on $[0,1]^\NN$ and totally analytic.
Moreover, for any multiset $M$, the partial derivatives can be expressed as 
\begin{equation}
\frac{\partial^{k}}{(\partial z)^M} t(z)=\sum_{i=1}^\infty \alpha_i \frac{\partial^{k}}{(\partial z)^M}  z^{M_i}.\label{eq-tderivpowerseries}
\end{equation}
Finally, the infinite sums in \eqref{eq-tpowerseries} and \eqref{eq-tderivpowerseries} converge uniformly on $[0,1]^\NN$.
\end{lemma}

\begin{proof}
Fix $\eps>0$ such that $c(1+\eps)<1-\eps$.
Observe that if $z\in[-\eps,1+\eps]^\NN$, then $|z^{M_i}|\le (1+\eps)^{|M_i|}\le (1+\eps)^{\Sigma(M_i)}$ for every $i$,
in particular, $|\alpha_i z^{M_i}|\le (1-\eps)^{\Sigma(M_i)}$ for all but finitely many $i\in\NN$.
Lemma~\ref{lm-expmonomboundfinitesum} applied with $1-\eps$ and $k=0$ yields that
the infinite sum defining the function $t$ converges uniformly on $[-\eps,1+\eps]^\NN$,
in particular, the function $t$ is well-defined on $[0,1]^\NN$ and
the sum \eqref{eq-tpowerseries} converges uniformly on $[0,1]^\NN$.

Fix a finite multiset $M$ of positive integers and
define $\beta_i$ to be the coefficient of the derivative of $\alpha_i z^{M_i}$ with respect to all $z_j$, $j\in M$.
Observe that $\beta_i=0$ whenever $M$ is not a subset of $M_i$.
We claim that
\begin{equation}
\frac{\partial^{k}}{(\partial z)^M} t(z)=\sum_{i=1}^\infty \beta_i z^{M_i\setminus M},\label{eq-tderivpowerseries-inproof}
\end{equation}
where $k=|M|$.
Since $|\beta_i|\le |\alpha_i|\cdot |M_i|^k$ and
$|z^{M_i\setminus M}|\le (1+\eps)^{\Sigma(M_i)}$ for every $z\in [-\eps,1+\eps]^\NN$,
we get that $|\beta_i z^{M_i\setminus M}|\le |M_i|^k\cdot (1-\eps)^{\Sigma(M_i)}$ for every such $z$ and for all but finitely many $i\in\NN$.
Hence, Lemma~\ref{lm-expmonomboundfinitesum} applied with $1-\eps$ and $k$ implies that
the right side of \eqref{eq-tderivpowerseries-inproof} uniformly converges on $[-\eps,1+\eps]^\NN$.
Let $g(z)$ be its limit for $z\in [-\eps,1+\eps]^\NN$.

By induction on $k$, we show that \eqref{eq-tderivpowerseries-inproof} holds for all multisets $M$.
Since the base of the induction and the induction step follow the same line of arguments,
we just focus on the case $k=1$, i.e., $M=\{j\}$ for some $j\in\NN$.
Let $f_n(z)$ be the sum of the first $n$ terms on the right side in \eqref{eq-tpowerseries}.
Since $\frac{\partial}{\partial z_j}f_n(z)$
is the sum of the first $n$ terms on the right side in \eqref{eq-tderivpowerseries-inproof}, 
$\frac{\partial}{\partial z_j}f_n(z)$ converges to $g(z)$.
Since $f_n$ converges to $f$ everywhere on $[-\eps,1+\eps]^\NN$,
we conclude using Theorem~\ref{thm-rudin} (applied to each fixed choice of $z_i \in [-\eps,1+\eps]$, $i\not\in M$) that
$\frac{\partial}{\partial z_j}t(z)$ is $g(z)$.
It follows that the function $t(z)$ is totally analytic.
\end{proof}

We are now ready to define the function $t_{P,H}(z)$.
Fix a bounding sequence $P$ and a graph $H$, and
let $M_i$ be the $i$-th multiset in the linear order defined in Section~\ref{sec-gensetup}.
We apply Lemma~\ref{lm-densityinfinitesum} and
let $\alpha_{P,i}$ be the sum of the coefficients $\alpha_{P,H',i}$ of $z^{M_i}$ 
in the polynomials $q_{P,H'}(z)$ in \eqref{eq-densityinfinitesum},
where the sum is taken over $H'\in\HH$.
Note that there exists $c_H\in (0,1)$ such that $|\alpha_{P,i}|\le c_H^{\Sigma(M_i)}$ for all but finitely many $i\in\NN$.
By Lemma~\ref{lm-monomialboundanalytic},
the function
\[t_{P,H}(z)=\sum_{i=1}^\infty \alpha_{P,i} z^{M_i}\]
is well-defined on $[0,1]^\NN$ and totally analytic.

We are now ready to prove Lemma~\ref{lm-densityanalytic}.
\begin{proof}[Proof of Lemma~\ref{lm-densityanalytic}]
We have already shown that the function $t_{P,H}(z)$ is totally analytic.
Hence, we need to show that
$t_{P,H}(z)=\tau(H,W_P(z))$ for every $z\in[0,1]^\NN$ that satisfies $p_j(z) \in [l_j,u_j]$ for all $j \in \NN$.
By Lemma~\ref{lm-densityinfinitesum}, we need to show that
\[t_{P,H}(z)=\sum_{H' \in \HH} q_{P,H'}(z),\]
where $\HH$ is the countable set $\HH$ from the statement the lemma.
This is equivalent to establishing that
\[t_{P,H}(z)=\sum_{i=1}^\infty \alpha_{P,i} z^{M_i}
            =\sum_{i=1}^\infty\sum_{H' \in \HH}\alpha_{P,H',i} z^{M_i}
	    =\sum_{H' \in \HH}\sum_{i=1}^\infty\alpha_{P,H',i} z^{M_i}
	    =\sum_{H' \in \HH} q_{P,H'}(z),\]
where $\alpha_{P,H',i}$ is the coefficient of $z^{M_i}$ in $q_{P,H'}(z)$.
The first, second and fourth equalities follow from the definitions.
The third equality holds because the sum of the absolute values of the terms is finite. Indeed, since $z\in [0,1]^\NN$,
we have
\begin{equation}
\sum_{i=1}^\infty\sum_{H' \in \HH}\left|\alpha_{P,H',i} z^{M_i}\right|\le 
\sum_{i=1}^\infty\sum_{H' \in \HH}\left|\alpha_{P,H',i}\right|.\label{eq-densityanalytic}
\end{equation}
Since all but finitely many terms $\sum_{H' \in \HH}\left|\alpha_{P,H',i}\right|$
are at most $c_H^{\Sigma(M_i)}$, which is at most $c_H^i$ by Observation~\ref{obs-monomialsadditive},
the sum \eqref{eq-densityanalytic} is indeed finite.
This completes the proof of the lemma.
\end{proof}

To prove Lemma~\ref{lm-WPvaries}, we need an additional lemma.

\begin{lemma}
\label{lm-approx-stronger}
For every bounding sequence $P$, graph $H$, and real $\eps>0$,
there exist a real $\eps_{P,H}>0$ and an integer $k_{P,H}$ with the following property.
If $P'$ is a bounding sequence that agrees with $P$ on the first $k_{P,H}$ elements and
has $L_1$-distance at most $\eps_{P,H}$ from $P$,
then the functions $t_{P,H}(z)$ and $t_{P',H}(z)$ are $\eps$-close.
\end{lemma}

\begin{proof}
We first show that there exists a finite subset $\HH'$ of $\HH$ that can be chosen independently of $P$ such that
the function $\wt_{P,H}(z)$ defined as
\begin{equation}
\wt_{P,H}(z)=\sum_{H'\in\HH'}q_{P,H'}(z) \label{eq-approx-q}
\end{equation}
is $\eps/4$-close to $t_{P,H}(z)$.
Here $q_{P,H'}(z)$ are the polynomials \eqref{eq-poly-alpha} from the statement of Lemma \ref{lm-densityinfinitesum}. Recall that $|\alpha_{P,H',i}| \le \beta_{H',i}$, and for all but finitely $i\in\NN$:
\[\sum_{H'\in\HH}\beta_{H',i}\le c_H^{\Sigma(M_i)},\]
where $c_H\in (0,1)$ is the constant from the statement of Lemma~\ref{lm-densityinfinitesum}.
Let $i_0$ be such that
\[\sum_{i=i_0+1}^\infty\sum_{H'\in\HH}i\cdot \beta_{H',i}\le
  \sum_{i=i_0+1}^\infty i\cdot c_H^{\Sigma(M_i)}\le
  \frac{\eps}{16}\,\mbox{.}\]
Note that such an $i_0$ exists by Lemma~\ref{lm-expmonomboundfinitesum} and
can be chosen so that it does not depend on $P$.
Finally choose $\HH'$ to be a finite subset of $\HH$ such that
\[\sum_{i=1}^{i_0}\sum_{H'\in\HH \setminus\HH'}\beta_{H',i}\cdot i\le\frac{\eps}{8}.\]
Note that such a set $\HH'$ exists and can be chosen independently of $P$
because for each $i \in \NN$, the sum $\sum_{H' \in \HH} \beta_{H',i}$ is finite.
The choice of $\HH'$ guarantees that for any bounding sequence $P$,
\[\sum_{i=1}^{i_0}\left|\alpha_{P,i}-\sum_{H'\in\HH'}\alpha_{P,H',i}\right|\cdot i
  \sum_{i=1}^{i_0}\left|\sum_{H'\in\HH \setminus\HH'}\alpha_{P,H',i}\right|\cdot i\le\frac{\eps}{8},\]
where $\alpha_{P,i}$ is the coefficient of $z^{M_i}$ in $t_{P,H}(z)$.

We now show that the functions $t_{P,H}(z)$ and $\wt_{P,H}(z)$ are $\eps/4$-close.
Let $z\in [0,1]^\NN$. Using \eqref{eq-tpowerseries}, we have the following estimates:
\begin{align*}
\left|t_{P,H}(z)-\wt_{P,H}(z)\right| & = \left|\sum_{i=1}^\infty \alpha_{P,i}z^{M_i}-\sum_{i=1}^\infty\sum_{H'\in\HH'}\alpha_{P,H',i}z^{M_i}\right| \\
                                     & \le \sum_{i=1}^\infty \left|\alpha_{P,i}-\sum_{H'\in\HH'}\alpha_{P,H',i}\right|\cdot \left|z^{M_i}\right| \\
				     & \le \sum_{i=1}^\infty \left|\alpha_{P,i}-\sum_{H'\in\HH'}\alpha_{P,H',i}\right|\\
				     & \le \frac{\eps}{8}+\sum_{i=i_0+1}^\infty \left|\alpha_{P,i}-\sum_{H'\in\HH'}\alpha_{P,H',i}\right|\\
				     & \le \frac{\eps}{8}+\sum_{i=i_0+1}^\infty \left|\alpha_{P,i}\right|+\sum_{i=i_0+1}^\infty \left|\sum_{H'\in\HH'}\alpha_{P,H',i}\right|\\
				     & \le \frac{\eps}{8}+2\sum_{i=i_0+1}^\infty \sum_{H'\in\HH}\left|\alpha_{P,H',i}\right| \le \frac{\eps}{4}\,\mbox{.}
\end{align*}
Likewise, for any $k\in\NN$, using \eqref{eq-tderivpowerseries}, we obtain the following estimates for the derivative with respect to $z_k$:
\begin{align*}
\left|\frac{\partial}{\partial z_k}t_{P,H}(z)-\frac{\partial}{\partial z_k}\wt_{P,H}(z)\right|
     & = \left|\sum_{i=1}^\infty\alpha_{P,i} \frac{\partial}{\partial z_k}z^{M_i}-\sum_{i=1}^\infty\sum_{H'\in\HH'}\alpha_{P,H',i} \frac{\partial}{\partial z_k}z^{M_i}\right| \\
     & \le \sum_{i=1}^\infty \left|\left(\alpha_{P,H,i}-\sum_{H'\in\HH'}\alpha_{P,H',i}\right)\frac{\partial}{\partial z_k}z^{M_i}\right|\\
     %& \le \sum_{i=1}^\infty \left|\alpha_{P,H,i}-\sum_{H'\in\HH'}\alpha_{P,H',i}\right|\cdot |M_i| \cdot \left|z^{M_i}\right|\\
     & \le \sum_{i=1}^\infty \left|\alpha_{P,H,i}-\sum_{H'\in\HH'}\alpha_{P,H',i}\right| \cdot |M_i|\\
     & \le \frac{\eps}{8}+\sum_{i=i_0+1}^\infty \left|\alpha_{P,H,i}-\sum_{H'\in\HH'}\alpha_{P,H',i}\right| \cdot i \\
     %& \le \frac{\eps}{8}+\sum_{i=i_0+1}^\infty i \cdot |\alpha_{P,H,i}|+ \sum_{i=i_0}^\infty i \cdot \left|\sum_{H'\in\HH'}\alpha_{P,H',i}\right| \\
     & \le \frac{\eps}{8}+2\sum_{i=i_0+1}^\infty \sum_{H'\in\HH'} i \cdot \left|\alpha_{P,H',i}\right| \le \frac{\eps}{4}\,\mbox{.}
\end{align*}

Let $K$ be the size of $\HH'$.
For each $H'\in\HH$, let $k_{P,H'}$ and $\eps_{P,H'}$ be the values
from the second part of the statement of Lemma~\ref{lm-densityinfinitesum} for $\eps/2K$.
Set $k_{P,H}$ to be the maximum of the values $k_{P,H'}$ and
set $\eps_{P,H}$ to be the minimum of the values $\eps_{P,H'}$,
in both cases taken over $H'\in\HH'$.
Note that the values of $k_{P,H}$ and $\eps_{P,H}$ depend on $P$ and $H$ only.

Consider a bounding sequence $P'$ that agrees with $P$ on the first $k_{P,H}$ elements and that
has $L_1$-distance from $P$ at most $\eps_{P,H}$. Let
\[\wt_{P',H}(z)=\sum_{H'\in\HH'}q_{P',H'}(z),\]
where $q_{P',H'}(z)$ are the polynomials from the statement of Lemma~\ref{lm-densityinfinitesum}.
Lemma~\ref{lm-densityinfinitesum} yields that
the polynomials $q_{P,H'}(z)$ and $q_{P',H'}(z)$ are $\eps/2K$-close for every $H'\in\HH'$.
It follows that the functions $\wt_{P,H}(z)$ and $\wt_{P',H}(z)$ are $\eps/2$-close. Recall that the choice of $\HH'$ was independent of $P$, thus,
the functions $t_{P,H}(z)$ and $\wt_{P,H}(z)$ are $\eps/4$-close, and
the functions $t_{P',H}(z)$ and $\wt_{P',H}(z)$ are also $\eps/4$-close.
We conclude that the functions $t_{P,H}(z)$ and $t_{P',H}(z)$ are $\eps$-close.
\end{proof}

We are now ready to prove Lemma~\ref{lm-WPvaries}.

\begin{proof}[Proof of Lemma~\ref{lm-WPvaries}]
Suppose that a bounding sequence $P$, a graph $H$ and a real $\eps>0$ are given.
Apply Lemma~\ref{lm-approx-stronger} with $P$, $H$ and $\eps$ to get a real $\eps_{P,H}>0$ and an integer $k_{P,H}$, and
apply Lemma~\ref{lm-add-P} with $\eps_{P,H}$ to get an integer $k_P$.
We set $k=\max\{k_{P,H},k_P\}$.
Let $P'$ be a bounding sequence that is a $k$-strengthening of $P$.
Lemma~\ref{lm-add-P} yields that the $L_1$-distance of $P$ and $P'$ is at most $\eps_{P,H}$.
Hence, Lemma~\ref{lm-approx-stronger} implies that the functions $t_{P,H}(z)$ and $t_{P',H}(z)$ are $\eps$-close.
\end{proof}

\section{Concluding remarks}
\label{sec-concl}

Notice that it was not essential in the proof of Theorem~\ref{thm-main} that
the functions $t_k$ represent densities of $G_k$ in $W_P(z)$.
The key property that allowed us to apply the machinery in Section \ref{sec-diag}
was that the functions $t_k$ are totally analytic functions in a feasible region.
While this is not true for all continuous functions on the space of graphons (for example,
we need a certain kind of differentiability), 
many functions besides those corresponding to densities of graphs have this property as well.
An example of another function amenable to our methods is the entropy function of a graphon,
which is defined for a graphon $W$ as follows:
$$H(W)=\int_{[0,1]^2}W(x,y)\log W(x,y)+(1-W(x,y))\log (1-W(x,y))\dd x\dd y\,\mbox{.}$$
The entropy of a graphon $W$ measures the number of graphs close in the cut distance to $W$~\cite{bib-CV11}, and
graphons maximizing the entropy among those observing some density conditions
correspond to a typical graph satisfying those conditions.
There are many classical results on the typical structure of graphs avoiding induced subgraphs,
e.g.~\cite{bib-ABBM11,bib-ERK76,bib-KMRS14,bib-KPR87,bib-PS91,bib-PS92,bib-PS93}, and
there is also a good amount of literature, e.g.~\cite{bib-CD13,bib-LZ15,bib-RS13,bib-RS15,bib-RRS14},
on the typical structure of graphs with general density conditions and the related phase transition phenomenon.

Theorem~\ref{thm-main} can be extended to require that every graphon in the family $\WW$ has the same non-zero entropy,
i.e., the structure of graphs close to graphons in $\WW$ is not exponentially concentrated
around a single ``typical'' graphon. More precisely, we obtain the following theorem.

\begin{theorem}
\label{thm-entropy}
There exists a family of graphons $\WW$, graphs $H_1,\ldots,H_{\ell}$ and reals $d_1,\ldots,d_{\ell}$ such that
\begin{itemize}
\item a graphon $W$ is weakly isomorphic to a graphon contained in $\WW$ if and only if $d(H_i,W)=d_i$ for every $i\in [\ell]$,
\item all graphons in $\WW$ have the same non-zero entropy, and
\item no graphon in $\WW$ is finitely forcible.
\end{itemize}
\end{theorem}

We remark that
we cannot use directly the family $\WW_P$ of graphons $W_P(z)$ to prove Theorem~\ref{thm-entropy}.
The reason is that the derivative of the entropy function $H(p)=p\log p+(1-p)\log (1-p)$,
which is integrated in the definition of graphon entropy, is not finite for $p\in\{0,1\}$, and
the values in some parts of the tile $C\times C$ can approach $0$ or $1$ as a function of $z\in [0,1]^\NN$.
An easy way to overcome this is to use the family of graphons $\Wt_P(z)$ defined as $\Wt_P(z)=W_P(z/2+1/4)$.
Then the values inside the tile $C\times C$, which depend on $z\in [0,1]^\NN$, stay bounded away from $0$ and $1$.
Note that this issue does not concern the tile $C\times E$:
all values in the tile $C\times E$ are either $0$ or $1$,
which implies that the integral of the entropy function over the tile $C\times E$ is always zero,
i.e., the integral is independent of $z\in [0,1]^\NN$.

The proof of Theorem~\ref{thm-main} presented here is not constructive
since the bounding sequence $P$ used in the proof is not explicitly constructed.
It is worth noting that Cooper et al. in~\cite{bib-turing},
which is an earlier version of \cite{bib-universal},
showed that any graphon in which the density of all dyadic squares can be approximated by a Turing machine
can be embedded in a finitely forcible graphon.
While the main result of~\cite{bib-turing} does not imply that any graphon can be embedded in a finitely forcible graphon,
it still provides an alternative way of proving Theorem~\ref{thm-main},
since the diagonalization arguments can be simulated by a suitably chosen Turing machine.
Although this would make the proof of Theorem~\ref{thm-main} more technical and significantly longer,
we would obtain a proof of Theorem~\ref{thm-main} where all the densities $d_1,\ldots,d_{\ell}$
can be taken to be rational,
i.e., we would obtain a counterexample to Conjecture~\ref{conj-main} involving only constraints that are finite in their nature.
In addition, such a proof of Theorem~\ref{thm-main} would also allow explicitly constructing the densities $d_1,\ldots,d_{\ell}$.

Finally, we would like to make a short remark about the orders of the graphs $H_1,\ldots,H_{\ell}$ in the statement of Theorem~\ref{thm-main}.
Each decorated constraint that we use involves graphs with at most $15$ vertices (including the constraints that
are included from the proof Theorem \ref{thm-universal} in \cite{bib-universal}).
Examining the proof of Lemma~\ref{lm-decorated} yields that the orders of graphs get multiplied by the number of parts increased by one,
which is $15$ in our case.
We conclude that all graphs $H_1,\ldots,H_{\ell}$ in the statement of Theorem~\ref{thm-main} have at most $225$ vertices.

\section*{Acknowledgement}

The authors would like to thank the anonymous reviewer for carefully reading the submitted manuscript and
the comments, which helped to improved the presentation of the results.

\newcommand{\advances}{Adv. Math. }
\newcommand{\annals}{Ann. of Math. }
\newcommand{\cpc}{Combin. Probab. Comput. }
\newcommand{\dam}{Discrete Appl. Math.}
\newcommand{\dcg}{Discrete Comput. Geom. }
\newcommand{\discrete}{Discrete Math. }
\newcommand{\eujc}{European J.~Combin. }
\newcommand{\gfa}{Geom. Funct. Anal. }
\newcommand{\jcta}{J.~Combin. Theory Ser.~A }
\newcommand{\jctb}{J.~Combin. Theory Ser.~B }
\newcommand{\rsa}{Random Structures Algorithms }
\newcommand{\sidma}{SIAM J.~Discrete Math. }
\newcommand{\tcs}{Theoret. Comput. Sci. }

\end{document}